\documentclass[11pt]{amsart}
\usepackage{amssymb,amsmath,geometry,latexsym,graphics,graphicx,tabularx,shapepar,enumerate,
hyperref,stmaryrd}
\usepackage[dvipsnames]{xcolor}
\usepackage{tikz}
\usepackage{tikz-cd}
\usetikzlibrary{matrix, calc, arrows}

\newcommand{\ZZ}{\mathbb{Z}}

\newcommand{\FF}{\mathbb{F}}
\newcommand{\CC}{\mathbb{C}}
\newcommand{\QQ}{\mathbb{Q}}
\newcommand{\RR}{\mathbb{R}}
\newcommand{\HH}{\mathbb{H}}
\newcommand{\EE}{\mathbb{E}}

\newcommand{\NN}{\mathbb{N}}

\newcommand{\TT}{\mathbb{T}}
\newcommand{\ev}{\operatorname{ev}}

\newcommand{\sg}{\mathfrak{s}}

\DeclareMathOperator{\sign}{sign}

\newtheorem{Theorem}{Theorem}[section]
\newtheorem{Lemma}[Theorem]{Lemma}

\newtheorem{Proposition}[Theorem]{Proposition}

\newtheorem{Corollary}[Theorem]{Corollary}
\newtheorem{Definition}[Theorem]{Definition}

\newtheorem{Conjecture}[Theorem]{Conjecture}

\title[Trivial multiple zeta values]{Trivial multiple zeta values in Tate algebras}
\date{August 19, 2020}
\author{O. Gezm\.{i}\c{s} \& F. Pellarin}

\address{O. Gezm\.{i}\c{s}\\Department of Mathematics, National Tsing Hua University, Hsinchu City 30042, Taiwan R.O.C.}

\address{F. Pellarin\\
Institut Camille Jordan, UMR 5208\\
Site de Saint-Etienne \\
23 rue du Dr. P. Michelon \\
42023 Saint-Etienne, France}
\thanks{The first author was supported by MOST Grant 109-2811-M-007-553}

\begin{document}

\maketitle

\begin{abstract}
We study {\em trivial multiple zeta values in Tate algebras}. These are particular examples of the {\em multiple zeta values in Tate algebras} introduced by the second author. If the number of variables involved is 'not large' in a way that is made precise in the paper, we can endow the set of trivial multiple zeta values with a structure of module over a non-commutative polynomial ring with coefficients in the rational fraction field over $\FF_q$. We determine the structure of this module in terms of generators and we show how in many cases, this is sufficient for the detection of linear relations between Thakur's multiple zeta values.
\end{abstract}

\tableofcontents

\section{Introduction}\label{introduction}

Let $p$ be a prime number, we set $q=p^e$ with $e\in\NN^*:=\NN\setminus\{0\}=\ZZ_{>0}$. We consider the local field $K_\infty=\FF_q((\frac{1}{\theta}))$ with $\theta$ an indeterminate. We denote by $$|\cdot|:K_\infty\rightarrow \RR_{\geq 0}$$ a non-trivial multiplicative valuation on $K_\infty$ (the image in $\RR_{>0}$ is of the form $c^\ZZ$ for $c>1$). The $\FF_q$-subalgebra $A:=\FF_q[\theta]$ of $K_\infty$ is discrete and cocompact. In particular, if $r>0$ and $(n_1,\ldots,n_r)\in(\NN^*)^r$, the series 
\begin{equation}\label{thakur-mzv}
\zeta_A(n_1,\ldots,n_r)=\sum_{\begin{smallmatrix}d_1>\cdots>d_r\geq 0\\ a_i\in A^+, |a_i|=|\theta^{d_i}|,\\ \forall i=1,\ldots,r\end{smallmatrix}}\frac{1}{a_1^{n_1}\cdots a_r^{n_r}}\end{equation}
converges in $K_\infty$, where $A^+$ denotes the subset of $A$ of monic polynomials in $\theta$. We denote by $\zeta_A(n_1,\ldots,n_r)$ its sum in $K_\infty$. These elements are often called {\em multiple zeta values of Thakur} as they have been first considered by Thakur in \cite{THA0}. They are the object of increasingly intensive study by several authors (the present paper only cites a part of the relevant articles in this topic). One of the main questions that push the research around these series is to give a theoretic interpretation for the structure of the subalgebra $\mathcal{Z}(K)$ of the $K$-algebra $K_\infty$
generated by $\zeta_A(0):=1$ and all the elements $\zeta_A(n_1,\ldots,n_r)$, where $K=\FF_q(\theta)$. The $K$-algebra $\mathcal{Z}(K)$ exhibits similarities with the $\QQ$-sub-algebra $\mathcal{Z}(\QQ)$ of $\RR$ generated by $1\in \RR$ and Euler-Zagier's multiple zeta values $\zeta(n_1,\ldots,n_r)\in\RR_{>0}$ (with $n_1\geq 2$, $n_2,\ldots,n_r\geq 1$). 

For instance, Thakur proved in \cite{THA2} that the $\FF_p$-sub-algebra of $K_\infty$ generated by $1\in K_\infty$ and all the elements in (\ref{thakur-mzv}) is equal to the $\FF_p$-vector space generated by these same elements. Furthermore, Chang proved in \cite{CHA} that the $K$-algebra generated by $1$ and the Thakur's multiple zeta values in (\ref{thakur-mzv})
is graded by the {\em weights} (see \S \ref{first-settings} for the definition of weight). The analogous property for $\mathcal{Z}(\QQ)$ is a conjecture (Goncharov's conjecture).
The periods of mixed Tate motives over $\ZZ$ are known to be polynomials with rational coefficients in $(2\pi \sqrt{-1})^{-1}$ and $\zeta(w)$ where $w$ corresponds to a {\em Lyndon word} (this is a striking consequence of Brown's \cite{BRO}). 

Similarly, Thakur formulated a very interesting hypothesis on a minimal set of generators of the $K$-algebra of his multiple zeta values, see \cite[Conjecture 8.2]{ThJTNB}. These series also occur as {\em periods} of Anderson's $A$-motives which are iterated extensions of tensor powers of 'Carlitz motives' and may therefore be called 'mixed Carlitz-Tate motives', 
see \cite{AND&THA1}. In more recent literature various Anderson's $A$-modules have been constructed to give Thakur's multiple zeta values an interpretation as entries of 'logarithms' of certain special points (see for example Anderson and Thakur, Chang and Mishiba, and even more recently, by Green and Ngo Dac, in \cite{AND&THA1,CHA&MIS1,CHA&MIS2,GRE&NGO} and the references therein). Some of these works and others that we do not mention have been crucial for example in important advances on transcendental properties, as well as on $\mathfrak{P}$-adic variants (with $\mathfrak{P}$ a place of $K$) that culminated with the proof of an analogue of a conjecture of Furusho, see \cite{CHA&MIS2}.

The study of Euler-Zagier multiple zeta values has gone deeper in the core of the structure of the $\QQ$-algebra $\mathcal{Z}(\QQ)$ in some sense as there are several conjectures available to help predicting part of its algebraic structure. The existence of 'double shuffle relations' illustrates this well. In \cite{IHA&KAN&ZAG} Ihara, Kaneko and Zagier highlight what we can call a 'folkloric conjecture' asserting that these relations, extended in a suitable way allowing divergent sums and integrals in the definitions, generate all the non-trivial algebraic relations between the multiple zeta values $\zeta(n_1,\ldots,n_r)\in\RR$. This conjecture has in fact been addressed or mentioned by several authors, more or less independently. We can mention Goncharov, Kontsevich, Racinet,  with the risk of missing others. Fortunately, the article \cite{IHA&KAN&ZAG} refers to enough many works to virtually cover them all.

There is no such an analogue statement for Thakur's multiple zeta values $\zeta_A(n_1,\ldots,n_r)\in K_\infty$; in fact, there is no analogue of integral shuffle relations available, to built a theory around double shuffle relations (but see Todd's conjecture in \cite{Todd}, see also \cite[Conjecture 8.3]{ThJTNB}). 
The aim of this paper is to introduce some kind of partial substitute of the 'double shuffle relations'. The construction we propose differs in many points from the classical setting and involves certain module structures over the non-commutative polynomial ring with coefficients in the rational fraction field over $\FF_q$. Here are the essential ingredients, collected in such a way that what follows can be seen as a plan of the paper (note however, that this is not a complete presentation of its content). 

\subsubsection*{First ingredient}
 In \cite{PEL1} the second author introduced a notion of {\em deformation of multiple zeta values} (of Thakur) in Tate algebras, which are entire functions of several variables $t_i$, $i\in\NN^*$. 
The definition is reviewed in \S \ref{mzvtate}. Here is an example in the case of one variable $t_j$ with $j>0$, to help the reader to learn the notation and the objects that are used in the paper:
$$\zeta_A\begin{pmatrix}\{j\} & \emptyset &\cdots & \emptyset \\ 
1 & q-1 & \cdots & q^{k-1}(q-1)\end{pmatrix}=\sum_{\begin{smallmatrix} a_1,\ldots,a_k\in A\\
\text{monic}\\
|a_1|>\cdots>|a_k|\end{smallmatrix}}\frac{a_1(t_j)}{a_1a_2^{q-1}\cdots a_{k}^{q^{k-1}(q-1)}}\in \widehat{K[t_j]}_{\|\cdot\|}.$$
Here, $a(t_j)\in\FF_q[t_j]$ is the image of $a$ by the unique injective $\FF_q$-algebra morphism
$A\rightarrow\FF_q[t_j]$ that sends $\theta$ to $t_j$. On the right of the above displayed formula, $\widehat{K[t_j]}_{\|\cdot\|}$ denotes the standard Tate algebra in the variable $t_j$, with coefficients in $K_\infty$, completion of $K[t_j]$ for the Gauss norm $\|\cdot\|$. Note that this multiple zeta value is determined by a 
set of data ('composition array') $\Big(\begin{smallmatrix}\{j\} & \emptyset &\cdots & \emptyset \\ 
1 & q-1 & \cdots & q^{k-1}(q-1)\end{smallmatrix}\Big)$ in which the first line is a $k$-tuple of finite subsets of the positive integers and the second line is a $k$-tuple of positive integers, like for 
Thakur's multiple zeta values that occur precisely when the first line is made of copies of the empty set, so that we have, comparing the notations of (\ref{thakur-mzv}) and ours:
$$\zeta_A\begin{pmatrix}\emptyset &\cdots & \emptyset \\ 
 n_1 & \cdots & n_r\end{pmatrix}=\zeta_A(n_1,\ldots,n_r).$$
In ibid. it was proved that the $\FF_p$-algebra 
and the $\FF_p$-vector space they generate, are equal. Let $\mathcal{Z}$ be this algebra. The corresponding product structure $\odot_\zeta$ on the free $\FF_p$-vector space $\boldsymbol{C}$ generated by 
these multiple zeta values in Tate algebras is called the {\em harmonic product}. Thanks to Chang's proof of the analogue of Goncharov's conjecture Theorem \ref{resultofchang} (original reference in \cite{CHA}), the algebra 
$\mathcal{Z}$ is graded by a monoid $\boldsymbol{S}$ introduced in Definition \ref{Monoid-of-weights}. The non-neutral elements of $\boldsymbol{S}$ are couples $(n,\Sigma)$ (often displayed `vertically' $\binom{\Sigma}{n}$) with $n\in\NN^*$ and $\Sigma:\NN^*\rightarrow\NN$ with finite support. We call them the {\em degrees.} We say that $\Sigma$ is a {\em finite subset} if $\Sigma(x)\in \{0,1\}$ for all $x \in \mathbb{N}^{*}$. In all the following, we will therefore identify, by abuse of notation, finite subsets of $\NN^*$ with their characteristic maps (See \S2.1 for more details). 
In the case of $\Sigma=\emptyset$ the constant zero map, the elements of $\mathcal{Z}$ which have degree 
$(n,\emptyset)$ are generated by Thakur's multiple zeta values. This determines a unitary subalgebra 
$\mathcal{Z}_\emptyset$ of $\mathcal{Z}$. Depending on the variables involved (i.e. depending on the choice of $\Sigma$), we have, more generally, graded $\mathcal{Z}_\emptyset$-modules $\mathcal{Z}_\Sigma$ (see (\ref{Sigma-space})) (\footnote{More precisely, graded modules over the graded algebra $\mathcal{Z}_\emptyset$.}).

\subsubsection*{Second ingredient}
In parallel, we introduce in \S \ref{Multiple-polylogarithms} a variant of the {\em multiple polylogarithms} considered by Chang in \cite{CHA}, functions of the variables $X_i$, $i\in\NN^*$. Here is an example of such series, in the variable $X_j$ with $j$ fixed:
$$\lambda_A\begin{pmatrix}\{j\} & \emptyset &\cdots & \emptyset \\ 
1 & q-1 & \cdots & q^{k-1}(q-1)\end{pmatrix}=\sum_{\begin{smallmatrix} a_1,\ldots,a_k\in A\\
\text{monic}\\
|a_1|>\cdots>|a_k|\end{smallmatrix}}\frac{X_j^{q^{\deg_\theta(a_1)}}}{a_1a_2^{q-1}\cdots a_{k}^{q^{k-1}(q-1)}}\in \widehat{K[X_j]}_{\|\cdot\|}^{\operatorname{lin}},$$
where $(\cdot)^{\operatorname{lin}}$ denotes the $\FF_q$-subspace of $\widehat{K[X_j]}_{\|\cdot\|}$ generated by the $\FF_q$-linear formal series.

They generate an $\FF_p$-algebra $\mathcal{L}$ which is also equal to the underlying $\FF_p$-vector space, so that we have on the $\boldsymbol{S}$-graded $\FF_p$-vector space $\boldsymbol{C}$, a multiplication 
product $\odot_\lambda$. But $\mathcal{L}(K)$, the $K$-algebra generated by our multiple polylogarithms, is not graded if we simply transfer the $\boldsymbol{S}$-grading of $\mathcal{Z}(K)$. We can still construct, similarly, graded $\mathcal{Z}_\emptyset(K)$-modules $\mathcal{L}_\Sigma(K)$ and one of our results is the following (see Theorem \ref{Z-L}).

\medskip

\noindent{\bf Theorem A.}{\em  ~If $\Sigma$ is a finite subset of $\NN^*$ of cardinality $<q$, then
there is an isomorphism of graded $\mathcal{Z}_\emptyset(K)$-modules $\mathcal{F}_\Sigma:
\mathcal{Z}_\Sigma(K)\rightarrow\mathcal{L}_\Sigma(K)$.}

\medskip

If the reader knows the reference \cite{ANG&PEL&TAV}, he/she might not be very surprised by this result. Similar constructions are performed there (the {\em star action} of a Tate algebra), with the target space being inside certain Banach algebras $\mathbb{B}_\Sigma$, notably with applications to class number formulas and Anderson's log-algebraic theorems (see also \cite{ANG&PEL&TAV}). However, an important difference with this reference is that the algebra $\mathcal{L}$ is tailored to also allow evaluation at $X_i=1$, $i\in\NN^*$, operation which is not well defined in the algebra $\mathbb{B}_\Sigma$, essentially because $1$ is not in the topological adherence of {\em Carlitz's circle} $\mathfrak{C}:=(-\theta)^{\frac{1}{q-1}}\FF_q[[\frac{1}{\theta}]]\subset\CC_\infty$ if $q\neq2$. 
The proof uses a formula by Perkins (\ref{eq1}) based on Lagrange's interpolation, that can be generalised to the case $|\Sigma|\geq q$.

To give an example of how Theorem A works and to allow the reader to have a first impression of the interplay of the main objects of the paper, we choose $\Sigma=\{j\}$ a singleton
(\footnote{We also use the notation $\sg$ for $\Sigma=\{j\}$ in the case of a singleton if we do not want to stress too much on the choice of $j$.}). Then, we have the element $$\zeta_A\binom{\{j\}}{1}=\prod_P\left(1-\frac{P(t_j)}{P}\right)^{-1}\in \widehat{K[t_j]}_{\|\cdot\|}$$ (completion for the Gauss valuation $\|\cdot\|$ extending $|\cdot|$, trivial on $\FF_q[t_j]$), the product being indexed by the irreducible monic elements of $A$. We have $\zeta_A\binom{\{j\}}{1}\in\mathcal{Z}_{\{j\}}$ (it has degree $\binom{\{j\}}{1}$) and 
$$\mathcal{F}_{\{j\}}\left(\zeta_A\binom{\{j\}}{1}\right)=\lambda_A\binom{\{j\}}{1}=\sum_{a\in A^+}a^{-1}X_j^{q^{\deg_\theta(a)}}\in\mathcal{L}_{\{j\}}$$ again of degree $\binom{\{j\}}{1}$ (the sum runs over the set $A^+$ of the monic elements of $A$), see (\ref{image-basic-zeta}). Also, with the explicit definition of $\mathcal{F}_\Sigma$ given in \S \ref{decomposition-theorem}, one sees easily that, for example,
$$\mathcal{F}_{\{j\}}\left(\zeta_A\begin{pmatrix}\{j\} & \emptyset &\cdots & \emptyset \\ 
1 & q-1 & \cdots & q^{k-1}(q-1)\end{pmatrix}\right)=\lambda_A\begin{pmatrix}\{j\} & \emptyset &\cdots & \emptyset \\ 
1 & q-1 & \cdots & q^{k-1}(q-1)\end{pmatrix},\quad k\geq 1.$$ However, it turns out that $\mathcal{F}_{\{j\}}(\zeta_A\binom{\{j\}}{q^k})\neq\lambda_A\binom{\{j\}}{q^k}$ if $k>0$. We also observe that 
Theorem A does not extend to the case $|\Sigma|\geq q$. See Remark \S \ref{limits-work} for more details.

\subsubsection*{Third ingredient}
In the $K$-algebra $\mathcal{Z}(K)$ we entail the subalgebra $\mathcal{Z}^{\operatorname{triv}}(K)$  of {\em trivial elements}, see the definition in \S \ref{Trivial-elements}, which, at a very first approximation, are the elements of $\mathcal{Z}(K)$ which have the property that, multiplied by certain products involving the {\em function of Anderson and Thakur} 
 \begin{equation}\label{omega-def}\omega(t):=(-\theta)^{\frac{1}{q-1}}\prod_{i\geq 0}\left(1-\frac{t}{\theta^{q^i}}\right)^{-1},\end{equation}
yield rational functions in the variables $t_j$, $j\in\NN^*$ (this will not be our operative definition). 
As already mentioned, the elements of $\mathcal{Z}(K)$ and therefore also of $\mathcal{Z}^{\operatorname{triv}}(K)$ are entire functions of several variables. The intersection of 
$\mathcal{Z}^{\operatorname{triv}}(K)$ with $\mathcal{Z}_\Sigma(K)$ gives rise to a 
graded $\mathcal{Z}_\emptyset(K)$-module $\mathcal{Z}_\Sigma^{\operatorname{triv}}(K)$. Note that $\mathcal{Z}_\emptyset^{\operatorname{triv}}(K)=
\mathcal{Z}_\emptyset(K).$
We also have finite dimensional $\FF_p$-vector spaces $\mathcal{Z}_{n,\Sigma}^{\operatorname{triv}}$ defined by selecting those $\FF_p$-linear combinations of our multiple zeta values of weight $n$ and type $\Sigma$ that are trivial in the above sense. The elements of 
$\mathcal{Z}_{n,\Sigma}^{\operatorname{triv}}$ allow to construct non-trivial 
$K$-linear dependence relations in the space $\mathcal{L}$ as the next result highlights. 
Indeed, in \S \ref{defi-E} we are going to describe, under the assumption $|\Sigma|<q$, an alternative definition $\mathcal{E}_\Sigma$ of
the isomorphisms $\mathcal{F}_\Sigma$ so that we have diagrams of $K$-linear maps:
$$
\begin{tikzpicture}
         \matrix (m) [matrix of math nodes,row sep=3em,column sep=1em,minimum width=2em]{
              \mathcal{Z}^{\operatorname{triv}}_{n,\Sigma}(K) &   & \mathcal{L}^{\operatorname{triv}}_{n,\Sigma}(K). \\
         };
         \path[-stealth]
            ([yshift=-3pt]m-1-1.east) edge node [above,yshift=1.0ex] {$\mathcal{E}_\Sigma$} ([yshift=-3pt]m-1-3.west);
         \path[-stealth]
            ([yshift=3pt]m-1-1.east) edge node [below,,yshift=-1.0ex] {$\mathcal{F}_\Sigma$} ([yshift=3pt]m-1-3.west);
        \end{tikzpicture}
$$
The two maps are equal, but are defined in different ways so that they naturally give rise to different lifts
$E_\Sigma$ and $F_\Sigma$ on the free modules generated by the elements $\zeta_A(\mathcal{C})$ and $\lambda_A(\mathcal{C})$. Therefore, the difference maps
$E_\Sigma-F_\Sigma$ have images in the kernel of $\lambda_A$ viewed as an augmentation 
map, a construction that can be viewed as an {\em analogue of the double shuffle relations.}
The next result involves only one variable. See Theorem \ref{Theorem-C} for a more general and precise statement).

\medskip

\noindent {\bf Theorem B.}{\em \;To any element $f\in\mathcal{Z}^{\operatorname{triv}}_{n,\{j\}}$
such that $f(\theta^{q^i})\neq 0$ for some $i>0$ we can associate a non-trivial $K$-linear dependence relation
in $\mathcal{L}(K)$.}

\medskip

Theorem B is obtained thanks to a complete investigation of the structure of the 
$\mathcal{Z}_\emptyset(K)$-module $\mathcal{Z}_\Sigma^{\operatorname{triv}}(K)$ in the case $|\Sigma|<q$. The main structural result is the following (see Theorem \ref{structure-of-trivial-MZV} and \S \ref{link-relations}).
We consider the non-commutative polynomial algebra $K[\boldsymbol{\tau}]$ with the commutation rule
$\boldsymbol{\tau} c=c^q\boldsymbol{\tau}$, $c\in K$.

\medskip

\noindent{\bf Theorem C.}{\em ~ Let $\Sigma$ be subset of $\NN^*$ of cardinality $<q$. There is a structure of left $K[\boldsymbol{\tau}]$-module over $\mathcal{Z}_\Sigma^{\operatorname{triv}}(K)$ compatible with the $\boldsymbol{S}$-grading. We have 
an isomorphism of $\mathcal{Z}_\emptyset(K)$-modules $\mathcal{Z}^{\operatorname{triv}}_{\Sigma}(K)\cong\otimes_{i\in\Sigma}\mathcal{Z}^{\operatorname{triv}}_{\{i\}}(K)$ and for any $i\in\Sigma$, $\zeta_A\binom{\{i\}}{1}$ freely generates $\mathcal{Z}^{\operatorname{triv}}_{\{i\}}(K)$ with its left $K[\boldsymbol{\tau}]$-module structure and the right $\mathcal{Z}_\emptyset(K)$-module structure given by the harmonic product.}

\medskip

To give again examples, we point out that it is possible to show that for all $k\geq 0$, $\zeta_A\binom{\{j\}}{q^k}\in\mathcal{Z}_{\{j\}}^{\operatorname{triv}}(K)$. Now, for how the structure of $K[\boldsymbol{\tau}]$-module of Theorem C is defined, we have for $k\geq1$ (Lemma \ref{identity-xi}) the following example: 
\begin{multline*}
\boldsymbol{\tau}^k\zeta_A\binom{\{j\}}{1}=(t_j-\theta)\cdots(t_j-\theta^{q^{k-1}})\zeta_A\binom{\{j\}}{q^k}=\\ =(-1)^k(\theta^{q^k}-\theta)\cdots(\theta^{q^k}-\theta^{q^{k-1}})\zeta_A\begin{pmatrix}\{j\} & \emptyset &\cdots & \emptyset \\ 
1 & q-1 & \cdots & q^{k-1}(q-1)\end{pmatrix}.\end{multline*}

Still under the hypothesis $|\Sigma|<q$, the image of $\mathcal{Z}_\Sigma^{\operatorname{triv}}(K)$
by $\mathcal{F}_\Sigma$ equals $\mathcal{Z}_\emptyset(K)[\lambda_A\binom{\{j\}}{1}:j\in\Sigma]^{\operatorname{lin}}$, where the exponent $(\cdot)^{\operatorname{lin}}$
means that we are considering all the formal series that are $\FF_q$-linear in each variable
$X_i$ with $i\in\Sigma$ and the latter algebra inherits a structure of $K[\boldsymbol{\tau}]$-module thanks to 
Corollary \ref{Theorem-E-F} and Lemma \ref{Esigma-functional}. Composing $\mathcal{F}_\Sigma$ with the evaluation 
$X_i=1$ for $i\in \Sigma$ we have thus a morphism of $K[\boldsymbol{\tau}]$-modules, compatible with the grading structures:
$$\mathcal{Z}_\Sigma^{\operatorname{triv}}(K)\xrightarrow{\mathcal{G}_\Sigma}\mathcal{Z}_\emptyset(K),$$ where $\boldsymbol{\tau} f=f^q$ for $f\in\mathcal{Z}_\emptyset(K)$. This morphism is studied in \S \ref{link-relations}. It is not likely to be surjective, however, we are led to the next (see Conjecture \ref{injectivity}):

\medskip

\noindent{\bf Conjecture D.}{\em ~$\mathcal{G}_\Sigma$ is injective.}

\medskip

Most likely, the above, to be proved, requires tools of diophantine nature. If true, it implies that non-trivial $K[\boldsymbol{\tau}]$-relations among elements of $\mathcal{Z}_\Sigma^{\operatorname{triv}}(K)$ always transfer to non-trivial $K$-linear relations among Thakur's multiple zeta values. 
Examples are provided in \S \ref{triviality-properties}. We rediscover well known linear dependence relations (such as some by Lara Rodr\'iguez and Thakur in \cite{LAR&THA} or Todd in \cite{Todd}) but we are also able to construct new families of relations, examples of which are described in the text (see \S \ref{family}). We are led to the conclusion that the
$K[\boldsymbol{\tau}]$-module structure of $\mathcal{Z}^{\operatorname{triv}}_\Sigma(K)$ comes here into play  as a 
substitute of double shuffle relations. 
Note that Conjecture D, applied in conjonction with Theorem B, allows to construct several non-trivial linear dependence relations in $\mathcal{Z}_n(K)$. Do we get the full $K$-vector space they generate?

The paper contains several other conjectures that can be seen by the reader as challenges for further researches. They do not influence our investigations except Conjecture D above. 

\subsubsection{Remark}\label{limits-work} The map $\mathcal{F}_\Sigma$ is not well defined and does not give rise to an isomorphism between $\mathcal{Z}_\Sigma$ and $\mathcal{L}_\Sigma$ when $\Sigma$ is a finite subset of $\NN^*$ such that $|\Sigma|\geq q$. The obstruction does not come from Perkins' formula in \cite[Proposition 2.17]{PER} that we stated for $|\Sigma|<q$, but  also holds for $|\Sigma|\leq 2(q-1)$. In fact, it is possible to extend this formula to the case of all $\Sigma$ finite subset of $\NN^*$ (unpublished). The nature of the generalised formula allows to construct $\mathcal{F}_\Sigma$ for all $\Sigma$ finite set but only as a function
$\mathcal{Z}_\Sigma\rightarrow K[[X_i:i\in\Sigma]]^{\operatorname{lin}}$ and there are elements 
$f\in\mathcal{Z}_\Sigma$ such that $\mathcal{F}_\Sigma(f)\not\in\widehat{K[X_i:i\in\Sigma]}_{\|\cdot\|}^{\operatorname{lin}}$ (therefore {\em a fortiori} the image is not entirely contained in $\mathcal{L}_\Sigma$), see \S \ref{some-elements-notin-E0}. Not only, but the resulting map $\mathcal{Z}_\Sigma\rightarrow K[[X_i:i\in\Sigma]]^{\operatorname{lin}}$ is no longer injective if $|\Sigma|\geq q$. Also, the module 
$\mathcal{Z}^{\operatorname{triv}}_{\Sigma}(K)$ can be defined for any finite subset $\Sigma\subset\NN^*$ regardless of its cardinality, and the map $\mathcal{F}_\Sigma$ can be defined, also for $|\Sigma|\geq q$, as a map $\mathcal{Z}^{\operatorname{triv}}_{\Sigma}(K)\rightarrow\mathcal{L}_\Sigma(K)$. However, it is easy to show that this restricted map is still not injective. For example, if $|\Sigma|=q$ the element
$$\zeta_A\binom{\Sigma}{1}=\sum_{\begin{smallmatrix}a \in A\\ \text{monic}\end{smallmatrix}}\frac{\prod_{j\in\Sigma}a(t_j)}{a}$$
is non-zero and belongs to the kernel of $\mathcal{F}_\Sigma$. 
More work is required to analyse the case of general $\Sigma$. Note that many statements of the paper do not extend without relevant modification in their formulations outside the assumption $|\Sigma|<q$. 

\subsubsection*{Acknowledgements} The authors are thankful to Hidekazu Furusho, Nathan Green and Tuan Ngo Dac for fruitful remarks and useful suggestions.

\subsection{General notation}

The paper will adopt the following standard notations that we recall here to ease the reading.
\begin{itemize}
\item $\NN=\ZZ_{\geq 0}$, $\NN^*=\ZZ_{>0}$.
\item $q=p^e$, $p$ a prime number and $e>0$. 
\item $\mathbb{F}_p\subset \mathbb{F}_q$, the finite fields with $p$ and $q$ elements respectively.
\item $A=\FF_q[\theta]$, where $\theta$ is an indeterminate over $\mathbb{F}_q$.
\item $K=\FF_q(\theta)$.
\item $K_\infty=\FF_q((\frac{1}{\theta}))$.
\item $\CC_\infty=\widehat{K_\infty^{ac}}$, completion of an algebraic closure of $K_\infty$, with multiplicative valuation $|\cdot|$.
\item $\exp_C,\log_C$ respectively the Carlitz exponential and the Carlitz logarithm.
\item $b_i(t)=(t-\theta)\cdots(t-\theta^{q^{i-1}})\in A[t]$.
\item $l_i=(\theta-\theta^q)\cdots(\theta-\theta^{q^{i}})\in A$.
\item $D_i=(\theta^{q^i}-\theta^{q^{i-1}})\cdots(\theta^{q^i}-\theta)\in A$.
\end{itemize}
Other notations are introduced and used, but the above are the most common. Additionally, we adopt the convention that empty sums are zero and empty products are one.
\section{Multiple sums after Thakur}\label{first-settings}

The first task is to introduce a {\em monoid of weights} for our multiple sums. In \S \ref{multiple-sums} we define {\em multiple sums} \`a la Thakur and prove standard facts about them in a wider generality than needed in the rest of the paper. Theorem \ref{mytheoremshuffle} describes the harmonic product of multiple sums.

\subsection{Monoid of degrees}\label{Monoid-of-weights}
A {\em finite weighted subset} $\Sigma$ of $\NN^*$ is a map $\Sigma:\NN^*\rightarrow\NN$ such that $\Sigma(x)=0$ for all but finitely many $x\in\NN^*$. We denote by $\mathfrak{S}$ the set of all finite weighted subsets. It is countable and it naturally carries the structure of a semi-ring with the multiplication $\cdot$ defined by the  addition in the target set and the addition $+$ given by the maximum in the target set. 
The neutral element for both the multiplication and the addition in $\mathfrak{S}$ is the map $\emptyset$ which sends every element $x\in\NN^*$ to $0$. The {\em cardinality} $|\Sigma|$ of a finite weighted subset $\Sigma$ as above is defined to be the integer
$$|\Sigma|:=\sum_{n\in\NN^*}\Sigma(n)$$
(note that the sum is finite). Also, if $x\in\NN^*$ and $\Sigma\in\mathfrak{S}$, we say that 
$x\in\Sigma$ if $\Sigma(x)\neq0$.

A finite weighted subset $\Sigma$ of $\NN^*$ is a {\em subset} if $\Sigma(x)\in\{0,1\}$ for all $x\in\NN^*$. In this case, we identify $\Sigma$ with the subset of $\NN^*$ which it is the characteristic map and this will not lead to confusion. Explicitly, if $\Sigma$ is a finite subset of $\NN^*$, we can describe it just by making the explicit list of its elements $\Sigma=\{i_1,\ldots,i_s\}$. 
For example, we denote with $\emptyset\in\mathfrak{S}$ both the empty set and the map sending $\NN^*$ to $0$ identically. Note that in this case we are led to the usual notion of cardinality.

Note that for subsets, the addition corresponds to the union $\cup$. If $\Sigma,\Sigma'\in\mathfrak{S}$ are such that $\Sigma+\Sigma'=\Sigma\cdot\Sigma'$, then
we say that they are {\em disjoint} and we write $$\Sigma\sqcup\Sigma':=\Sigma+\Sigma'=\Sigma\cdot\Sigma'.$$
The {\em support} $\operatorname{Spt}(\Sigma)$ of $\Sigma\in\mathfrak{S}$ is the subset of all the elements $x\in\NN^*$ such that $\Sigma(x)\neq0$. If $\Sigma'\in\mathfrak{S}$ and $\operatorname{Spt}(\Sigma')=\Sigma$, we say that $\Sigma'$ is a {\em thickening} of $\Sigma$.
 There is a partial order on $\mathfrak{S}$ which is induced by $\leq$ on $\NN$.
We write $\Sigma\leq\Sigma'$ if $\Sigma(x)\leq\Sigma'(x)$ for all $x\in\NN^*$. On subsets of $\NN^*$, this coincides with the inclusion $\subset$. 

\begin{Definition}
{\em The {\em monoid of degrees} is the set
$$\boldsymbol{S}:=\mathfrak{S}\times\NN^*\cup\left\{\binom{\emptyset}{0}\right\}$$
with the structure of monoid induced by the semigroups $(\mathfrak{S},\cdot)$
and $(\NN^*,+)$. The binary operation of $\boldsymbol{S}$ is denoted multiplicatively, with the symbol $\cdot$. The neutral element for $\cdot$ is $\binom{\emptyset}{0}$ and is more simply denoted by $\boldsymbol{0}$. }\end{Definition}
Note that we have adopted a 'vertical rendering' for the elements of $\boldsymbol{S}$. 
The monoid $\boldsymbol{S}$ allows to define a grading structure on the $\FF_p$-algebra of multiple zeta values in Tate algebras as defined in \S \ref{mzvtate}. Note that $(\NN,\max,+)$ is a unitary semiring with $0$ the neutral element for both the binary operations so that $\boldsymbol{S}$ itself comes equipped with a structure of unitary semiring but we will not use the additive operation in this paper. The elements of $\boldsymbol{S}^{\oplus r}$ with $r>0$ are also called {\em admissible composition arrays} and we adopt the convention that 
$\boldsymbol{0}$ is the unique admissible composition array of depth $0$ and $\boldsymbol{S}^{\oplus0}=\{\boldsymbol{0}\}$;
see \S \ref{composition-arrays} below.

\subsubsection{Composition arrays}\label{composition-arrays}

This subsection only contains definitions, notations and terminology useful for the sequel.
We consider $n\in\ZZ$ and $\Sigma\in\mathfrak{S}$.
A {\em composition array of $\binom{\Sigma}{n}$} is a matrix of the form

\begin{equation}\label{a-sample-of-C}
\mathcal{C}:=\begin{pmatrix}\Sigma_1 & \cdots & \Sigma_r\\ n_1 & \cdots & n_r\end{pmatrix}=(\mathcal{C}_1,\ldots,\mathcal{C}_r),
\end{equation} 
with $\Sigma_1,\ldots,\Sigma_r\in\mathfrak{S}$ and $n_1,\ldots,n_r\in\ZZ$ such that $\Sigma_1\cdot\cdots\cdot\Sigma_r=\Sigma$ and $n=\sum_in_i$. We can more simply speak about composition arrays without mentioning $\Sigma$ and $n$.
We call $r$ the {\em depth} of $\mathcal{C}$, 
$n=\sum_in_i$ its {\em weight} and $\Sigma=\prod_i\Sigma_i\in\mathfrak{S}$ its {\em type}.
If $n_i>0$ for all $i$, or if $\mathcal{C}=\boldsymbol{0}=\binom{\emptyset}{0}$, we say that $\mathcal{C}$ is {\em admissible}. 
In other words, an {\em admissible composition array} $\mathcal{C}$ of $\mathcal{X}\in\boldsymbol{S}$
 of depth $r>0$ is an $r$-tuple $(\mathcal{C}_1,\ldots,\mathcal{C}_r)$ of $(\boldsymbol{S}\setminus\{\boldsymbol{0}\})^{\oplus r}$ 
such that $$\mathcal{C}_1\cdot\cdots\cdot\mathcal{C}_r=\mathcal{X},$$
and, $\boldsymbol{0}=\binom{\emptyset}{0}$ is the unique admissible composition array of depth zero. It has weight $0$ and type $\emptyset$.
If $\mathcal{C}$ is such that $\Sigma$ is a finite subset of $\NN^*$ with the property that $\Sigma=\sqcup_i\Sigma_i$ (disjoint union), we say that $\mathcal{C}$ is {\em basic}. Suppose that $\mathcal{C}$ is a basic composition array of $\binom{\Sigma}{n}$. If $\Sigma'$
is a thickening of $\Sigma$ then we can write $\Sigma'=\Sigma_1'\cdots\Sigma_r'=\Sigma_1'+\cdots+\Sigma_r'$
with $\Sigma'_i$ thickening of $\Sigma_i$ for all $i$. The composition array $$\mathcal{C'}=\begin{pmatrix}\Sigma_1' & \cdots & \Sigma_r'\\ n_1 & \cdots & n_r\end{pmatrix}$$ is then called the {\em thickening of $\mathcal{C}$ along $\Sigma'$}. 

\subsection{Multiple sums}\label{multiple-sums}

We introduce the formalism of our multiple zeta sums, elaborated on the model of Thakur's multiple zeta values. Then, we proceed with the first few elementary properties of our sums.
We now follow the formalism of \cite[\S 8]{PEL3} but we essentially make use of the proofs contained in \cite{PEL1}. Let $\mathcal{B}$ be a Banach $\CC_\infty$-algebra with norm $\|\cdot\|$. We consider:
\begin{itemize}
\item[(i)] Injective maps $\delta_i:A\rightarrow\mathcal{B}$ for $i\in\NN^*$ such that for all $i$ the set $\delta_i(A)$
is bounded for $\|\cdot\|$.
\item[(ii)] An injective map $\gamma:A\rightarrow\mathcal{B}^\times$, where $\mathcal{B}^\times$ is the group of invertible elements, such that $\|\gamma(ab)\|=\|\gamma(a)\|\|\gamma(b)\|$ for all $a,b\in A$ and $\|\gamma(a)\|\rightarrow\infty$ as $|a|\to  \infty$.
\end{itemize}
The map $i\mapsto\delta_i$ for $i\in\NN^*$ extends in a unique way to a map $$\mathfrak{S}\xrightarrow{\delta}\operatorname{Maps}(A,\mathcal{B})$$ by the rule
$$\Sigma\mapsto\left(A\ni a\mapsto \prod_{i\in\NN^*}\delta_i(a)^{\Sigma(i)}\in\mathcal{B}\right).$$
 We denote by $\delta_\Sigma$ the image of $\Sigma$ by $\delta$.
The image $\mathcal{T}_\delta$ of $\delta$ has a structure of semi-ring with unit the map
$\boldsymbol{1}:a\mapsto 1$, for all $a\in A$. Let $\mathcal{C}\in\boldsymbol{S}^{\oplus r}$ be an admissible composition array as in (\ref{a-sample-of-C}). We set:
\begin{equation}\label{prototype}
\zeta_{A}^{\gamma,\delta}(\mathcal{C})=\zeta_{A}^{\gamma,\delta}\begin{pmatrix}\Sigma_1 & \cdots & \Sigma_r\\ n_1 & \cdots & n_r\end{pmatrix}=\sum_{\begin{smallmatrix}a_1,\ldots,a_r\in A^+\\ |a_1|>|a_2|>\cdots>|a_r|\end{smallmatrix}}\frac{\delta_{\Sigma_1}(a_1)\cdots\delta_{\Sigma_r}(a_r)}{\gamma(a_1)^{n_1}\cdots\gamma(a_r)^{n_r}}\in \mathcal{B}.\end{equation}
Note indeed that the conditions (i) and (ii) imply the convergence of the series above. 
This is the {\em multiple zeta series associated to $\mathcal{C}$ and the datum $(\gamma,\delta)$}. 
We denote by $\mathcal{Z}_{n,\Sigma}^{\gamma,\delta}$ the $\FF_p$-subvector space of $\mathcal{B}$ generated by the series $\zeta_{A}^{\gamma,\delta}(\mathcal{C})$ with $\mathcal{C}$ admissible composition array of $\binom{\Sigma}{n}$
and we set $\zeta_A^{\gamma,\delta}\binom{\emptyset}{0}:=1$. We also denote by $\mathcal{Z}^{\gamma,\delta}$ the $\FF_p$-vector space generated by all the sets $\mathcal{Z}_{w,\Sigma}^{\gamma,\delta}$.
Let $\mathcal{C}$ be as in (\ref{a-sample-of-C}) and admissible, and let us consider $d\geq 0$. We set:
\begin{eqnarray}
S_d^{\gamma,\delta}(\mathcal{C})&=&\sum_{\begin{smallmatrix}a_1,\ldots,a_r\in A^+\\ |\theta|^d=|a_1|>|a_2|>\cdots>|a_r|\end{smallmatrix}}\frac{\delta_{\Sigma_1}(a_1)\cdots\delta_{\Sigma_r}(a_r)}{\gamma(a_1)^{n_1}\cdots\gamma(a_r)^{n_r}},\label{powersums-def}\\
S_{<d}^{\gamma,\delta}(\mathcal{C})&=&\sum_{\begin{smallmatrix}a_1,\ldots,a_r\in A^+\\ |\theta|^d>|a_1|>|a_2|>\cdots>|a_r|\end{smallmatrix}}\frac{\delta_{\Sigma_1}(a_1)\cdots\delta_{\Sigma_r}(a_r)}{\gamma(a_1)^{n_1}\cdots\gamma(a_r)^{n_r}}=\sum_{0\leq j <d}S_j^{\gamma,\delta}(\mathcal{C})\label{powersums-def2}.
\end{eqnarray}
Additionally we set
$S_{0}^{\delta}\binom{\emptyset}{0}:=1$, $S_{k}^{\delta}\binom{\emptyset}{0}:=0$ for $k>0$, $ S_{<0}^{\delta}\binom{\emptyset}{0}:=0$ and $S_{<k}^{\delta}\binom{\emptyset}{0}:=1$ for $k>0$.
Note that
\begin{equation}\label{E:expres}
\zeta_{A}^{\gamma,\delta}(\mathcal{C})=\sum_{d\geq 0}S_d^{\gamma,\delta}(\mathcal{C})=\lim_{d\rightarrow\infty}S_{<d}^{\gamma,\delta}(\mathcal{C}).
\end{equation}

\begin{Proposition}\label{elementary-facts} Let $\mathcal{C}_1$ and $\mathcal{C}_2$ be two admissible composition arrays of $\binom{U_1}{n_1}$ and $\binom{U_2}{n_2}$.
The following properties hold. 
\begin{enumerate}
\item There exists a finite set $I$ and elements $f_i\in\FF_p$, $i\in I$, such that $S_d^{\gamma,\delta}(\mathcal{C}_1)S_d^{\gamma,\delta}(\mathcal{C}_2)=\sum_{i\in I}f_iS_d^{\gamma,\delta}(\mathcal{D}_i)$ for all $d\geq 0$, 
where for $i\in I$, $\mathcal{D}_i$ are admissible composition arrays of
$\binom{U_1\cdot U_2}{n_1+n_2}$, independent of $d$.
\item There exists a finite set $J$ and elements $g_i\in\FF_p$, $i\in J$, such that $S_{<d}^{\gamma,\delta}(\mathcal{C}_1)S_{<d}^{\gamma,\delta}(\mathcal{C}_2)=\sum_{i\in J}g_iS_{<d}^{\gamma,\delta}(\mathcal{E}_i)$ for all $d\geq 0$, 
where for $i\in J$, $\mathcal{E}_i$ are admissible composition arrays of
$\binom{U_1\cdot U_2}{n_1+n_2}$, independent of $d$.
\item There exists a finite set $L$ and elements $h_i\in\FF_p$, $i\in L$, such that
for all $d\geq 0$ we have $S_d^{\gamma,\delta}(\mathcal{C}_1)S_{<d}^{\gamma,\delta}(\mathcal{C}_2)=\sum_{i\in L}h_iS_d^{\gamma,\delta}(\mathcal{F}_i)$ where $\mathcal{F}_i$
are admissible composition arrays of
$\binom{U_1\cdot U_2}{n_1+n_2}$, independent of $d$.
\end{enumerate}
Moreover, all the coefficients $f_i,g_i,h_i$ do not depend on $\gamma$, $\delta$.
\end{Proposition}

\begin{proof}
This follows easily from \cite[Theorem 8.2]{PEL3} (see also \cite[Theorem 2.5]{PEL1}, where a particular case is proved and note that the latter reference contains all the tools to prove our result). The methods are also very similar to those involved in Thakur's \cite{THA2} and \cite[Theorem 5.1]{ThJTNB}.
\end{proof}

For $w\in\NN$ and $\Sigma\in\mathfrak{S}$ let $\boldsymbol{C}_{w,\Sigma}$ be the free $\FF_p$-vector space generated by indeterminates denoted by $[\mathcal{C}]$ where $\mathcal{C}$ is an admissible composition array as in (\ref{a-sample-of-C}). We thus have the $\FF_p$-vector space $\boldsymbol{C}:=\oplus_{w,\Sigma}\boldsymbol{C}_{w,\Sigma}$ and by using (\ref{E:expres}) and Proposition \ref{elementary-facts} we can endow $\boldsymbol{C}$ with a natural structure of $\FF_p$-algebra with product $\odot_\zeta$ and we have noticed that it does not depend on $\gamma,\delta$. This algebra can be realised as the quotient of the free $\FF_p$-algebra generated by the 
indeterminates $[\mathcal{C}]$, by the ideal generated by the quadratic polynomials underlying the relations produced by Proposition \ref{elementary-facts}.
We deduce that $\mathcal{Z}^{\gamma,\delta}$ has a structure of $\FF_p$-algebra and:
\begin{Theorem}\label{mytheoremshuffle} For any $\gamma,\delta$ the assignment defined by  
$[\mathcal{C}]\mapsto \zeta_A^{\gamma,\delta}(\mathcal{C})$ defines an $\FF_p$-algebra morphism $$(\boldsymbol{C},+,\odot_\zeta)\xrightarrow{\zeta_A^{\gamma,\delta}}\mathcal{Z}^{\gamma,\delta}$$
which induces $\FF_p$-linear maps $\boldsymbol{C}_{w,\Sigma}\rightarrow\mathcal{Z}^{\gamma,\delta}_{w,\Sigma}$.
\end{Theorem}
We call the product structures induced on $\boldsymbol{C}$ and $\mathcal{Z}^{\gamma,\delta}$ the {\em harmonic product.}  
In the present work, we only study the case in which $\gamma(a)=a\in K\subset\mathcal{B}$ (note that in \cite{PEL3} more general maps $\gamma$ are considered). 
In our settings, as presented so far, we have stressed the dependence on two classes of maps: $\gamma,\delta$. 
From now on, we set $\gamma(a)=a$ for all $a$ (i.e. the image of $a\in A\subset\CC_\infty$ in $\mathcal{B}$) for the $\CC_\infty$-algebra structure of $\mathcal{B}$ and we simplify our notation by writing 
$\zeta_{A}^{\delta}(\mathcal{C})$ for $\zeta_{A}^{\text{Id},\delta}(\mathcal{C})$ and 
$\mathcal{Z}^{\delta}$ for $\mathcal{Z}^{\text{Id},\delta}$ etc.

\section{Multiple zeta values in Tate algebras}\label{mzvtate}

We recall the foundations of the theory of multiple zeta values in Tate algebras of \cite{PEL1} and we collect some basic properties. The main result is Corollary \ref{resultofchang} which states 
that the $K$-algebra of multiple zeta values in Tate algebras is graded by the monoid $\boldsymbol{S}$.

If $R$ is a commutative ring with unit, we denote by $R[\underline{t}]$ the commutative $R$-algebra of polynomials with coefficients in $R$ in the indeterminates $t_i$ with $i\in\NN^*$. If $\Sigma\in\mathfrak{S}$, we denote by $R[\underline{t}_\Sigma]$ the sub-algebra of polynomials
in the indeterminates $t_i$ with $i\in\operatorname{Spt}(\Sigma)$. Now we restrict our attention to the case of $R$ an $\FF_q$-algebra. More precisely, we consider the case $R=\FF_q[\underline{t}]$.
A {\em multiplicative semi-character} $\chi:A\rightarrow\FF_q[\underline{t}]$ is a map defined by 
$$\chi(a)=\prod_{i\in\NN^*}\chi_{t_i}(a)^{j_i},$$
for integers $j_i\in\NN$ such that $j_i=0$ for all but finitely many $i\in\NN^*$. Note that any such map is totally multiplicative.
We denote by $\mathcal{T}_{\text{sc}}$ the set of all the multiplicative semi-characters. There is a natural bijection 
$$\Sigma=\Sigma_\chi\leftrightarrow\chi_\Sigma=\chi$$
between weighted subsets of $\NN^*$ and multiplicative semi-characters inducing a structure of semi-ring on
$\mathcal{T}_{\text{sc}}$ so that the multiplicative unit is $\boldsymbol{1}$, the trivial semi-character which maps 
every $a\in A$ to $1\in\FF_q[\underline{t}]$. The bijection is induced by
$$\Sigma\mapsto \chi_\Sigma=\left(a\mapsto\prod_{i\in\NN^*}\chi_{t_i}(a)^{\Sigma(i)}\right).$$

This, together with the choice $\gamma(a)=a$ for all $a$, leads to the definition of the $\FF_p$-algebra of {\em multiple zeta values $\mathcal{Z}$}.
Explicitly, the multiple zeta value associated to an admissible $\mathcal{C}$ as in (\ref{a-sample-of-C}) and the choice $(\gamma,\delta)=(\operatorname{id},\chi)$ is the element of the standard Tate algebra 
$$\TT_\Sigma=\widehat{\CC_\infty[\underline{t}_{\operatorname{Spt}(\Sigma)}]}_{\|\cdot\|}=\widehat{\CC_\infty[t_i:i\in\operatorname{Spt}(\Sigma)]}_{\|\cdot\|}$$ (the completion of the polynomial algebra $\CC_\infty[\underline{t}_{\operatorname{Spt}(\Sigma)}]$ for the Gauss valuation $\|\cdot\|$):

\begin{equation}\label{mzvchit}
\zeta_A(\mathcal{C})=\zeta_A\begin{pmatrix}\Sigma_1 & \cdots & \Sigma_r\\ n_1 & \cdots & n_r\end{pmatrix}=\sum_{d\geq 0}\sum_{\begin{smallmatrix}d> d_1>\cdots>d_r\geq 0\\ |a_i|=|\theta|^{d_i}\forall i=1,\ldots,r\end{smallmatrix}}\frac{\chi_{\Sigma_1}(a_1)\cdots\chi_{\Sigma_r}(a_r)}{a_1^{n_1}\cdots a_r^{n_r}}\in\TT_\Sigma.
\end{equation}

\subsection{Basic properties}\label{basic-properties}

A simple Galois invariance argument allows to show that any admissible multiple zeta value as in (\ref{mzvchit})
belongs to $\FF_p[\underline{t}_\Sigma][[\frac{1}{\theta}]]$.
It has been further proved (see \cite[Proposition 4]{PEL}) that the above series converges in fact to an 
entire function $\CC_\infty^\Sigma\rightarrow\CC_\infty$. We can thus write $\zeta_A(\mathcal{C})\in\EE_{\Sigma}$
where $\EE_\Sigma$ is the $\CC_\infty$-subalgebra of $\TT_\Sigma$ whose elements are the entire functions in the variables $t_i$ with $i\in\operatorname{Spt}(\Sigma)$.

Let $n$ be in $\NN^*$ and $\Sigma\in\mathfrak{S}$. We consider the $\FF_p$-vector subspace $\mathcal{Z}_{n,\Sigma}$ of the $\CC_\infty$-algebra $$\EE_{\NN^*}:=\bigcup_{\Sigma\in\mathfrak{S}}\EE_\Sigma$$ generated by the multiple zeta values
(\ref{mzvchit}) with admissible composition array (\ref{a-sample-of-C}) of weight $n$
and type $\Sigma$.
For such an {\em admissible} multiple zeta value, we say that $r$ is the {\em depth}, $\sum_in_i$ is the {\em weight},
and $\Sigma$ is the {\em type}. We also set $\zeta_A\binom{\emptyset}{0}:=1$ (type $\emptyset$, weight $0$) and $\mathcal{Z}_{0,\emptyset}:=\FF_p$. We consider the $\FF_p$-subvector space of $\EE_{\NN^*}$:
$$\mathcal{Z}=\sum_{n\geq 0\atop \Sigma\in\mathfrak{S}}\mathcal{Z}_{n,\Sigma}.$$ By Theorem \ref{mytheoremshuffle}, $\mathcal{Z}$ is an $\FF_p$-algebra and we have the 
$\FF_p$-algebra morphism $\zeta_A:\boldsymbol{C}\rightarrow\mathcal{Z}$. Let $\boldsymbol{C}_\emptyset$ be the subalgebra of $\boldsymbol{C}$ generated by $[\mathcal{C}]\in \boldsymbol{C} $
where $\mathcal{C}$ are admissible composition arrays of trivial type. Then $$\zeta_A(\boldsymbol{C}_\emptyset)=\mathcal{Z}_\emptyset:=\sum_{n\geq0}\mathcal{Z}_{n,\emptyset}$$ is the $\FF_p$-algebra of {\em Thakur's multiple zeta values}. More explicitly,
let $n$ be in $\NN^*$. Then, every admissible composition array of $\binom{\emptyset}{n}$ is of the form
$$\begin{pmatrix}\emptyset & \cdots & \emptyset\\ n_1 & \cdots & n_r\end{pmatrix}$$ and the associated multiple zeta value is given in (\ref{thakur-mzv}) (see \cite[Definition 5.10.1]{THA0} (\footnote{We adopt in this case a specific notation which is closer to the notation introduced by Thakur.})). 

If $B$ is a subalgebra of the $K$-algebra $K[\underline{t}_\Sigma]$, then 
we also consider the $B$-submodules $\mathcal{Z}_{n,\Sigma}(B)$ of $\EE_{\NN^*}$ generated by all the $B$-linear combinations of the $\zeta_A(\mathcal{C})$'s with $\mathcal{C}$ admissible of type $\Sigma$ and weight $n$, and along with this,
the $B$-subalgebra $\mathcal{Z}(B)$ of $\EE_{\NN^*}$. We have the next result with $B=K[\underline{t}]$.

\begin{Proposition}\label{proposition-general properties}
Let us fix a composition array $\mathcal{C}$ as in (\ref{a-sample-of-C}), not necessarily admissible. 
Then $\zeta_A(\mathcal{C})$ converges to an element of $\mathcal{Z}(K[\underline{t}])$.
Moreover, we have the following properties.
\begin{enumerate}
\item For every $i\in\NN^*$ and $k\in\NN$ the evaluation map $\CC_\infty[t_i]\rightarrow\CC_\infty$ defined by $t_i\mapsto\theta^{q^k}$
induces a $K$-algebra endomorphism $\mathcal{Z}(K[\underline{t}])\rightarrow\mathcal{Z}(K[\underline{t}])$. 
\item The $\FF_p[\underline{t}]$-linear automorphism $\mu$ of $\CC_\infty[\underline{t}]$
defined by $c\mapsto c^p$ for $c\in\CC_\infty$ induces an $\FF_p[\underline{t}]$-linear endomorphism of $\mathcal{Z}(K[\underline{t}])$.
\end{enumerate}
\end{Proposition}

\begin{proof}
Observe that
for all $n\leq 0$ and $\Sigma\in\mathfrak{S}$, $S_d\binom{\Sigma}{n}=0$
for all but finitely many $d\in\NN$. This is elementary and easy to check.
If $\mathcal{C}$ is a composition array as in (\ref{a-sample-of-C}) and if $n_1\leq 0$, there exists $M\geq 0$ such that for all $d\geq M$, $S_d\binom{\Sigma_1}{n_1}=0$. 
Then
$$\zeta_A(
\mathcal{C})=S_{<k}(
\mathcal{C}),\quad \forall k\geq M$$
and $\zeta_A(
\mathcal{C})\in K[\underline{t}]\subset\mathcal{Z}(K[\underline{t}])$.
We now consider a non-admissible composition array $\mathcal{C}$ as in   (\ref{a-sample-of-C})
with $n_1,\ldots,n_{j-1}>0$ and $n_j\leq 0$ so that $(\mathcal{C}_1,\ldots,\mathcal{C}_{j-1})$ is admissible, and we study the sequence $(S_{<k}(\mathcal{C}))_{k\geq 0}$. There exists $M$ such that for all $d\geq M$, $S_d\binom{\Sigma_j}{n_j}=0$.
Then,
\begin{multline*}
\zeta_A(\mathcal{C})=\sum_{M> d_1>\cdots>d_r\geq 0}
S_{d_1}\binom{\Sigma_1}{n_1}\cdots S_{d_r}\binom{\Sigma_r}{n_r}+\\
+\sum_{d_1\geq M> d_2>\cdots>d_r\geq 0}
S_{d_1}\binom{\Sigma_1}{n_1}\cdots S_{d_r}\binom{\Sigma_r}{n_r}+\cdots\\
+\sum_{d_1>\cdots>d_{j-1}\geq M>d_{j}>\cdots>d_r\geq 0}
S_{d_1}\binom{\Sigma_1}{n_1}\cdots S_{d_r}\binom{\Sigma_r}{n_r}.
\end{multline*}
If $1\leq k\leq j-1$ then we have
\begin{multline*}
\sum_{d_1>\cdots>d_{k}\geq M>d_{k+1}>\cdots>d_r\geq 0}
S_{d_1}\binom{\Sigma_1}{n_1}\cdots S_{d_r}\binom{\Sigma_r}{n_r}=\\ =S_{<M}\begin{pmatrix}\Sigma_{k+1} & \cdots & \Sigma_{r} \\ n_{k+1} & \cdots & n_r\end{pmatrix}\sum_{d_1>\cdots>d_k\geq M}S_{d_1}(\mathcal{C}_1)\cdots S_{d_k}(\mathcal{C}_k).
\end{multline*}
We note that 
\begin{multline*}
\sum_{d_1>\cdots>d_k\geq M}S_{d_1}(\mathcal{C}_1)\cdots S_{d_k}(\mathcal{C}_k)=
\zeta_A(\mathcal{C}_1,\ldots,\mathcal{C}_k)-\\ -\sum_{d_1>\cdots>d_{k-1}\geq M>d_k}S_{d_1}(\mathcal{C}_1)\cdots S_{d_k}(\mathcal{C}_k)-\cdots-\sum_{M>d_1>d_2>\cdots>d_{k}}S_{d_1}(\mathcal{C}_1)\cdots S_{d_k}(\mathcal{C}_k).\end{multline*}
Inductively on $k$, we get
$$\sum_{d_1>\cdots>d_k\geq M}S_{d_1}(\mathcal{C}_1)\cdots S_{d_k}(\mathcal{C}_k)\in\mathcal{Z}(K[\underline{t}]).$$ Therefore $\zeta_A(C)\in \mathcal{Z}(K[\underline{t}])$ and this proves the first property described in the proposition. The properties (1) and (2) are easy and left to the reader.\end{proof}

\subsubsection{Grading structures}
In \cite{CHA}, Chang proves the following result (analogue of a conjecture of Goncharov):
\begin{Theorem}[Chang]\label{resultofchang}
The $K$-algebra $\mathcal{Z}_\emptyset(K)$ is graded by the weights.
\end{Theorem}
This result relying on a transcendence proof can be applied to deduce several statements about our multiple zeta values in Tate algebras. One of them is the next:

\begin{Corollary}\label{cor-grading}
The $K[\underline{t}]$-algebra $\mathcal{Z}(K[\underline{t}])$ is graded by weights and types, that is, by the monoid $\boldsymbol{S}$.
\end{Corollary}

\begin{proof}
First of all we focus on a property of entire functions in $\EE_{\NN^*}$. Let $U$ be a non-empty finite subset of $\NN^*$ and $i$ an element of $U$. Let us write $U'=U\setminus\{i\}$. Let $f$ be in $\EE_U$. Then $f$ can be seen as
an entire function
$$\CC_\infty\rightarrow\EE_{U'}.$$
In other words, we can expand $f$ in a series
$$f(t_i)=\sum_{k\geq 0}f_kt_i^k,$$
where $f_k$ is an element of $\EE_{U'}$ for all $k\geq 0$, with the property that for all valuations $\|\cdot\|_v$
of $\CC_\infty[\underline{t}_{U'}]$ (trivial over $\FF_q[\underline{t}]$, which extend in a unique way to $\EE_{U'}$), $\|f_k\|_vR^k\rightarrow0$ as $k$ tends to $\infty$. We claim that if $\tau^N(f)_{t_i=\theta}=0$ for infinitely many $N\in\NN$, then $f=0$.
Indeed, note that
$$\tau^N(f)_{t_i=\theta}=f\Big(t_i=\theta^{\frac{1}{q^N}},t_l\mapsto t_l^{\frac{1}{q^N}}:l\in U'\Big)^{q^N},\quad N\geq0.$$
Hence, $\tau^N(f)_{t_i=\theta}=0$ if and only if $f_{t_i=\theta^{\frac{1}{q^N}}}=0$ in $\EE_{U'}$ and the claim follows from the fact that the set $\{\theta^{\frac{1}{q^N}}:N\geq 0\}$ is infinite.

To prove our result it suffices to show that if $Z_j\in\mathcal{Z}_{n_j,\Sigma_j}(K[\underline{t}_{U}])\setminus\{0\}$ ($j=1,\ldots,m$)
are such that the set $\{\binom{\Sigma_j}{n_j}:j=1,\ldots,m\}$ has $m$ elements, 
then the $Z_j$'s are linearly independent over $K[\underline{t}_{U}]$, where $U\in\mathfrak{S}$
is a subset such that $\operatorname{Spt}(\Sigma_j)\subset U$ for all $j$. We prove this property by induction over $|U|\in\NN$. If $|U|=0$, then this is just Theorem \ref{resultofchang}. Now let us suppose that 
$|U|>0$, let $i$ be in $U$ as before, let us write $U'=U\setminus\{i\}$. For all $j=1,\ldots,m$ there exists $\beta_j\in\NN$ such that $$\chi_{\Sigma_j}=\chi_{t_i}^{\beta_j}\chi_{\Sigma'_j},$$
where $\Sigma_j'\leq \Sigma_j$.
Assume by contradiction that there exist $c_1,\ldots,c_m\in K[\underline{t}_U]$ such that $f=\sum_jc_jZ_j=0$.
Then $\tau^N(f)_{t_i=\theta}=0$ for all $N\geq 0$. 
Now, for all but finitely many $N$, we have $\tau^N(c_j)_{t_i=\theta}\neq0$, $q^Nn_j-\beta_j>0$ and
$$\tau^N(c_jZ_j)_{t_i=\theta}\in\mathcal{Z}_{q^Nn_j-\beta_j,\Sigma_j'}(K[\underline{t}_{U'}])\setminus\{0\}.$$
If $j_1,j_2$ are distinct elements of $\{1,\ldots,m\}$ such that $\binom{\beta_{j_1}}{n_1}=\binom{\beta_{j_2}}{n_2}$, then $\Sigma_{j_1}'\neq\Sigma_{j_2}'$ by our hypothesis. If $\binom{\beta_{j_1}}{n_1}\neq\binom{\beta_{j_2}}{n_2}$, the sets $\{q^Nn_1-\beta_1,q^Nn_2-\beta_2\}$ are not singletons for all 
but finitely many $N$. Hence for all but finitely many $N$
the set $$\left\{\binom{\Sigma'_j}{q^Nn_j-\beta_j}:j=1,\ldots,m\right\}\subset\NN^*\times\mathfrak{S},$$ well defined, has exactly $m$ elements and by the induction hypothesis we can find $N$ such that the non-zero elements
$\tau^N(Z_j)_{t_i=\theta}\in\mathcal{Z}_{q^Nn_j-\beta_j,\Sigma'_j}(K[\underline{t}_{U'}])\setminus\{0\}$
are linearly independent over $K[\underline{t}_{U'}]$ which yields a contradiction with the non vanishing of the evaluations of $\tau^N(c_j)$.
\end{proof}
We end this section with several remarks and questions.

\subsubsection{Remark: a question on topological groups}

Theorem \ref{resultofchang} implies that $\mathcal{Z}_\emptyset(K)$ is graded by the monoid $\NN$ while Corollary \ref{Fp-trivial} implies, more generally, that $\mathcal{Z}_\emptyset(K[\underline{t}])$ is graded by the monoid $\boldsymbol{S}$. In his book \cite[\S 8]{Goss}, Goss describes the construction of a complete topological group $\mathbb{S}_\infty$ (determined by the choice of a uniformiser of $K_\infty$) together 
with embeddings $\NN\rightarrow\mathbb{S}_\infty$ so that the image is discrete, hence allowing to extend the map 
$\NN^*\rightarrow K_\infty$ defined by $n\mapsto\zeta_A(n)$ to a partial analogue of Riemann's zeta function, the {\em zeta function of Goss}, $\zeta_A:\mathbb{S}_\infty\rightarrow\CC_\infty$. Similar constructions lie above families of multiple zeta values of Thakur of $\mathcal{Z}_\emptyset$ that we will not discuss in the present paper. We do not know how to construct a complete topological group $\widehat{\boldsymbol{S}}$ together with embeddings $\boldsymbol{S}\rightarrow \widehat{\boldsymbol{S}}$ with discrete images in such a way that the zeta values $\zeta_A(\mathcal{C})$ with $\mathcal{C}\in \boldsymbol{S}$ or more generally, families of multiple zeta values $\zeta_A(\mathcal{C})$, can be extended analogously.

\subsubsection{Remark: a problem of homogeneity}

Note that the map $\mu$ of Proposition \ref{proposition-general properties} induces $\FF_p[\underline{t}_\Sigma]$-linear maps $$\mathcal{Z}_{n,\Sigma}(K[\underline{t}_\Sigma])\rightarrow\mathcal{Z}_{pn,\Sigma}(K[\underline{t}_\Sigma]).$$
If $\mathcal{C}$ is not admissible $\zeta_A(\mathcal{C})$ needs not to be homogeneous for the grading
$$\mathcal{Z}(K[\underline{t}])=\bigoplus_{(w,\Sigma)\in\NN^*\times\mathfrak{S}}\mathcal{Z}_{w,\Sigma}(K[\underline{t}]).$$ Therefore, the evaluation map $t_i\mapsto \theta^{q^k}$ (for $i\in\NN^*$ and $k\in\NN$) does not respect the grading by weights. Luckily this problem does not occur for the {\em trivial elements} defined in \S \ref{Trivial-elements}, see Corollary \ref{corollary-doesnt-exceed}.

\subsubsection{Remark: the role of the Carlitz period} The so-called Euler-Carlitz relations imply that if $n\in\NN^*$ is such that $q-1\mid n$, then $\zeta_A(n)\in\widetilde{\pi}^nK^\times$ with 
$$\widetilde{\pi}:=\theta\lambda_\theta\prod_{i>0}\left(1-\frac{\theta}{\theta^{q^i}}\right)^{-1}\in \lambda_\theta K_\infty\subset\CC_\infty$$
(for a choice $\lambda_\theta=(-\theta)^{\frac{1}{q-1}}$ of root of $-\theta$). Then, $\widetilde{\pi}^n\in\mathcal{Z}_{n,\emptyset}(K)$
for all $n\geq 0$ such that $q-1\mid n$. Note however that in general, we do not expect the relation $\widetilde{\pi}^n\in\mathcal{Z}_{n,\emptyset}$. 

\section{Multiple polylogarithms}\label{Multiple-polylogarithms}

Let $\mathcal{C}$ be an admissible composition array of $\binom{\Sigma}{n}$ as in (\ref{a-sample-of-C}). We define:
$$\lambda_A(\mathcal{C}):=\sum_{d_1>\cdots>d_r\geq 0}S_{d_1}(n_1)\cdots S_{d_r}(n_r)\prod_{i\in \NN^*}X_i^{\sum_{j=1}^r\Sigma_j(i)q^{d_j}},$$ where $$S_d(n):=S_d\binom{\emptyset}{n}=\sum_{\begin{smallmatrix} a\in A\\ 
\text{monic}\\ |a|=|\theta|^d\end{smallmatrix}}\frac{1}{a^n}.$$ It is easy to see that $\lambda_A(\mathcal{C})$ converges in the Tate algebra $\widehat{K[\underline{X}_\Sigma]}_{\|\cdot\|}$.

For $n\in\NN^*$ and $\Sigma\in\mathfrak{S}$, let $\mathcal{L}_{n,\Sigma}$ be the $\FF_p$-vector space generated by $\lambda_A(\mathcal{C})$ where $\mathcal{C}$ is a composition array of $\binom{\Sigma}{n}$,
with the convention $\mathcal{L}_{0,\emptyset}=\FF_p\lambda_A\binom{\emptyset}{0}=\FF_p$. 
Note that $\mathcal{L}_{n,\emptyset}=\mathcal{Z}_{n,\emptyset}$.
The series $\lambda_A(\mathcal{C})$ can be interpreted as the 
case of the datum $(\gamma,\delta)$ in \S \ref{multiple-sums} determined by 
\begin{equation}\label{delta-sigma}
\delta_\Sigma(a)=\prod_{i\in\NN^*}X_i^{\Sigma(i)q^{\deg_\theta(a)}}
\end{equation} and $\gamma$ the identity map. 
By the harmonic product formulas of Proposition \ref{elementary-facts} in the case of trivial type, the multiple polylogarithms carry a variant of the harmonic product as well. 
In particular, there exists an $\FF_p$-linear map
$$\lambda_A:\boldsymbol{C}\rightarrow\mathcal{L}=\sum_{(n,\Sigma)\in\boldsymbol{S}}\mathcal{L}_{n,\Sigma}$$ and a structure of $\FF_p$-algebra over $\boldsymbol{C}$ with product 
$\odot_\lambda$ such that $\lambda_A$ induces an $\FF_p$-algebra morphism; to see this we use Proposition \ref{elementary-facts} with $\Sigma=\emptyset$.

The products $\odot_\lambda$ and $\odot_\zeta$ (the latter introduced in \S \ref{basic-properties}) agree on $\mathcal{Z}_\emptyset=\mathcal{L}_\emptyset=\sum_n\mathcal{L}_{n,\Sigma}$ but are different otherwise. For instance, it is easy to see that, if
$\mathcal{C}=\binom{\{i\}}{1}$, $\mathcal{D}=\binom{\{j\}}{1}$ with $i\neq j$ and $q>2$ then:
$$[\mathcal{C}]\odot_\zeta[\mathcal{D}]=\left[\binom{\{i,j\}}{2}\right],\quad [\mathcal{C}]\odot_\lambda[\mathcal{D}]=\left[\binom{\{i,j\}}{2}\right]+\left[\begin{pmatrix}\{i\} & \{j\} \\ 1 & 1\end{pmatrix}\right]+\left[\begin{pmatrix}\{j\} & \{i\} \\ 1 & 1\end{pmatrix}\right].$$

\begin{Proposition}
For $\Sigma,\Sigma'\in\mathfrak{S}$ and $n,n'\in\NN^*$ we have $\mathcal{L}_{n,\Sigma}\mathcal{L}_{n',\Sigma'}\subset\mathcal{L}_{n+n',\Sigma+\Sigma'}$.
\end{Proposition}
\begin{proof}
It follows easily from the existence of the harmonic product of the powers sums $S_d(n)$ as first noticed by Thakur in \cite{THA2} and the next result which is easy to check. 

\begin{Lemma}\label{larger-lambda}
If $\mathcal{C}_1,\ldots,\mathcal{C}_r$ are composition arrays of $\binom{\emptyset}{n_1},\ldots,\binom{\emptyset}{n_r}$ and if $(\Sigma_1,\ldots,\Sigma_r)\in\mathfrak{S}^{\oplus r}$, then
the sum
$$\sum_{d_1>\cdots>d_r}S_{d_1}(\mathcal{C}_1)\cdots S_{d_r}(\mathcal{C}_r)\prod_{i\in\NN^*}X_i^{\sum_{j=1}^r\Sigma_j(i)q^{d_i}}$$
belongs to $\mathcal{L}_{n,\Sigma}$, where $n=\sum_in_i$ and $\Sigma=\Sigma_1\cdots\Sigma_r$.\end{Lemma}
\end{proof}

An important feature is that since $\mathcal{L}$ embeds in a filtered union of Tate algebras, the evaluation map $\ev$ defined by sending $X_i$ to $1$ for all $i\in\NN^*$ is a well defined $\FF_p$-algebra map $\mathcal{L}\xrightarrow{\ev}K_\infty.$ More precisely:

\begin{Lemma}\label{lemma-map-ev}
The map $\ev$ defines an $\FF_p$-algebra map $\mathcal{L}\rightarrow\mathcal{Z}_\emptyset$
and surjective $\FF_p$-linear maps $\mathcal{L}_{n,\Sigma}\rightarrow\mathcal{Z}_{n,\emptyset}$. 
\end{Lemma}

\begin{proof}
This follows immediately from the definition of multiple zeta values of Thakur as series in multiple power sums $S_d(n_1,\ldots,n_r)$ with $n_1,\ldots,n_r\in\NN^*$ (we simplify the notations when the type is trivial).
\end{proof}

We can define the $K$-algebra $\mathcal{L}(K)$ in a way completely similar to that of $\mathcal{Z}(K)$. 
We propose a conjecture:

\begin{Conjecture}\label{Fp-trivial}
For any $n\in\NN^*$ and $\Sigma\in\mathfrak{S}$, the $\FF_p$-linear maps $\zeta_A$ and $\lambda_A$
define isomorphisms of vector spaces
$$\mathcal{Z}_{n,\Sigma}\xleftarrow{\zeta_A}\boldsymbol{C}_{n,\Sigma}\xrightarrow{\lambda_A}\mathcal{L}_{n,\Sigma}.$$
\end{Conjecture}

Looking at the case of trivial type $\Sigma=\emptyset$ the conjecture above reduces to a similar conjecture by Thakur \cite[\S 6.3]{ThJTNB}, stipulating that the only $\FF_p$-linear relations among Thakur's multiple zeta values are the trivial ones. The map $\lambda_A$ does not seem to extend to an $\FF_p$-algebra homomorphism. However, this does not rule out that the $\FF_p$-algebras $(\boldsymbol{C},+,\odot_\zeta)$ and $(\boldsymbol{C},+,\odot_\lambda)$ are possibly isomorphic.

The $K$-algebra $\mathcal{L}(K)$ is not graded by $\boldsymbol{S}$. For example, the $K$-linear relations (\ref{non-homogeneous-relations}) are not homogeneous. In the next subsection, we analyse what structures are respected by the product $\odot_\lambda$.

\subsection{Grading structures on $\mathcal{L}(K)$} If $\mathcal{C}\in\boldsymbol{S}^{\oplus r}$ is a composition array, we denote by $\mathcal{C}^\flat$ the element of $\mathfrak{S}^{\oplus s}$ the $s$-tuple (for some $s\leq r$)  resulting from the first line of $\mathcal{C}$ (in $\mathfrak{S}^{\oplus r}$) by removing the occurrences of $\emptyset$. If $\mathcal{C}$ is of trivial type, we set $\mathcal{C}^\flat=(\emptyset)$. The aim of this subsection is to prove the next result. 

\begin{Theorem}\label{grading-structures}
Let us consider admissible composition arrays $\mathcal{C}_1,\ldots,\mathcal{C}_k$ and let us suppose that the functions $\lambda_A(\mathcal{C}_1),\ldots,\lambda_A(\mathcal{C}_k)$ are $K$-linearly dependent. Then there exist two distinct integers $i,j\in\{1,\ldots,k\}$ such that
$\mathcal{C}_i$ and $\mathcal{C}_j$ have the same weight and $$\mathcal{C}_i^\flat=(U_1,\ldots,U_r),\quad \mathcal{C}_j^\flat=(V_1,\ldots,V_r)\in \mathfrak{S}^{\oplus r}$$ for some $r$. Moreover, there exist integers $\kappa_1,\ldots,\kappa_r\in\ZZ$ such that $U_i=q^{\kappa_i}V_i$ for $i=1,\ldots,r$.
\end{Theorem}

In particular, if the types of $\mathcal{C}_1,\ldots,\mathcal{C}_k$ are subsets of $\NN^*$, then 
$\kappa_1=\cdots=\kappa_r=0$ and $\mathcal{C}_i^\flat=\mathcal{C}_j^\flat$.
To prove this result we need some definitions and intermediary results.

\begin{Definition}\label{cluster}{\em A {\em cluster of depth $r>0$} is a subset of $\NN^{\NN^*}$ of the form $[\alpha_1,\ldots,\alpha_r]:=\{\alpha_1q^{d_1}+\cdots+\alpha_rq^{d_r}\}$, where $\alpha_1,\ldots,\alpha_r$ are maps $\NN^*\rightarrow\NN$ with finite support but not constant equal to zero (almost everywhere zero but non-zero) and 
the $r$-tuples $(d_1,\ldots,d_r)\in\NN^r$ are such that $d_1>\cdots>d_r$.
We define a partial ordering $\prec$ on clusters by saying that $$\underline{\alpha}:=[\alpha_1,\ldots,\alpha_r]\prec\underline{\beta}$$
if there exists a sequence $(d_{1,m},\ldots,d_{r,m})_{m\geq 0}$ which is {\em good} in the sense that 
$d_{r,m}\rightarrow\infty$ as $m\rightarrow\infty$ and $d_{i,m}-d_{i+1,m}\rightarrow\infty$ as $m\rightarrow\infty$, for all $i<r$, such that $\alpha_1q^{d_{1,m}}+\cdots+\alpha_rq^{d_{r,m}}\in\underline{\beta}$ for all $m\geq 0$.
We define a relation of equivalence $\sim$ between clusters by saying that $\underline{\alpha}=[\alpha_1,\ldots,\alpha_r]$ and $\underline{\beta}=[\beta_1,\ldots,\beta_l]$ are {\em equivalent}, written as $\underline{\alpha}\sim\underline{\beta}$, if $r=l$ and for all $i=1,\ldots,r$ there exists $\kappa_i\in\ZZ$ such that $\alpha_i=q^{\kappa_i}\beta_i$.
}\end{Definition}

\begin{Lemma}
If $\underline{\alpha}=[\alpha_1,\ldots,\alpha_r]$ and $\underline{\beta}=[\beta_1,\ldots,\beta_l]$ are two clusters of rank $r,l$ and $\underline{\alpha}\prec\underline{\beta}$, then $r\leq l$ and we can find an increasing sequence $0=h_0<\cdots<h_r=l$ and integers $\kappa_{i,j}\in\ZZ$
such that
\begin{equation}\label{alpha-identity}
\alpha_i=\beta_{h_{i-1}+1}q^{\kappa_{i,h_{i-1}+1}}+\cdots+\beta_{h_{i}}q^{\kappa_{i,h_i}},\quad \forall i=1,\ldots,r.\end{equation}
\end{Lemma}

\begin{proof}
We consider a sequence $(d_{1,m},\ldots,d_{r,m})_{m\geq 0}$ which is good as in Definition \ref{cluster}, such that 
$$\sum_{i=1}^r\alpha_iq^{d_{i,m}}=\sum_{j=1}^l\beta_jq^{e_{j,m}},\quad \forall m\geq0,$$
where the sequence $(e_{1,m},\ldots,e_{l,m})_{m\geq 0}$ satisfies $e_{1,m}>\cdots>e_{l,m}$ for all $m$. Dividing by $q^{d_{r,m}}$ we obtain
$$\alpha_r+\sum_{i=1}^{r-1}\alpha_iq^{d_{i,m}-d_{r,m}}=\sum_{j=1}^l\beta_jq^{e_{j,m}-d_{r,m}},\quad m\geq 0.$$ the sum $\sum_{i=1}^{r-1}\alpha_iq^{d_{m,i}-d_{m,r}}$ tends to zero in $\ZZ_p$
and setting $f_{j,m}:=e_{j,m}-d_{r,m}$ for all $j,m$ we can write, in $\ZZ_p$:
$$\alpha_r=\lim_{m\rightarrow\infty}\sum_{j=1}^l\beta_jq^{f_{j,m}}.$$
We deduce that $\{f_{l,m},m\geq 0\}$ is bounded. Any infinite subsequence of a good sequence $(d_{1,m},\ldots,d_{r,m})_{m\geq 0}$ is good so that, upon extraction, we can suppose that $f_{l,m}=\kappa_{r,l}$ for all $m$. If the sequence $(f_{l-1,m})_{m\geq 0}$ is unbounded it contains an infinite subsequence growing to $\infty$ and we deduce that $\alpha_r=\beta_lq^{\kappa_{r,l}}$, which is an identity of the type (\ref{alpha-identity}). Otherwise it is bounded and again upon restriction to an infinite subsequence, we can suppose that it is constant. 
Since the sequence $(d_{1,m},\ldots,d_{r,m})_{m\geq 0}$ is good there exists $h_{r-1}>0$ such that $(f_{h_{r-1}+1,m})_{m\geq 0}$ is bounded and 
$(f_{h_{r-1},m})_{m\geq 0}$ is unbounded. This implies, by the same arguments as above, that $\alpha_r$ satisfies an identity like (\ref{alpha-identity}). At this point, arguing with clusters of depth $r-1$ and $l-1$ and using induction, we conclude that $r\leq l$ and that identities like (\ref{alpha-identity}) hold for all $i$ hence completing the proof of the lemma.
\end{proof}
We immediately deduce the next result.
\begin{Corollary}\label{corollary-ordering}
The relation $\prec$ is transitive. Moreover, 
If $\underline{\alpha}$ and $\underline{\beta}$ are clusters such that $\underline{\alpha}\prec\underline{\beta}$ and $\underline{\beta}\prec\underline{\alpha}$, then $\underline{\alpha}\sim\underline{\beta}$.
\end{Corollary}

\begin{proof}[Proof of Theorem \ref{grading-structures}]
For a formal power series $f\in\CC_\infty[[\underline{X}_\Sigma]]$ we define the {\em support}
$\operatorname{Spt}(f)$ to be the subset of $\NN^\Sigma$ whose elements are the $\underline{i}=(i_j:j\in\Sigma)$ such that the monomial $\prod_{j\in\Sigma}X_j^{i_j}$ occurs in the series defining $f$ with non-zero coefficient. Let us consider a non-trivial linear dependence relation
$$\sum_{i=1}^kc_k\lambda_A(\mathcal{C}_i)=0,\quad c_1,\ldots,c_k\in K^\times,$$ where we can assume that all the coefficients $c_i$ are non-zero. Applying the map $\ev$ we see that 
$\sum_{i=1}^kc_i\zeta_A(n_i)=0$ where $n_i$ is the weight of $\mathcal{C}_i$ for all $i$. By Chang's Theorem \ref{resultofchang}, we can therefore assume that $\mathcal{C}_1,\ldots,\mathcal{C}_k$ all have the same weight. Also, there is no loss of generality to suppose that no one has trivial type. We note that for all $i\in\{1,\ldots,k\}$, 
$$\operatorname{Spt}(\lambda_A(\mathcal{C}_i))\subset\bigcup_{j\neq i}\operatorname{Spt}(\lambda_A(\mathcal{C}_j)).$$ For all $i$, $\operatorname{Spt}(\lambda_A(\mathcal{C}_i))$ contains the cluster $\underline{\alpha}_i=[\alpha_{i,1},\ldots,\alpha_{i,r_i}]$ where
$\alpha_{i,1},\ldots,\alpha_{i,r_i}$ are the entries of the first row of $\mathcal{C}_i^\flat$. We deduce that for all $i$,
$$\underline{\alpha}_i\subset\bigcup_{j\neq i}\underline{\alpha}_j.$$
By the box principle, there is a map $g:\{1,\ldots,k\}\rightarrow\{1,\ldots,k\}$ such that
for all $i$, $g(i)\neq i$, with $\underline{\alpha}_{i}\prec\underline{\alpha}_{g(i)}$ for all $i$. Applying Corollary \ref{corollary-ordering}  
there exist $i\neq j$ such that $\underline{\alpha}_{i}\prec \underline{\alpha}_{j}$ and 
$\underline{\alpha}_{j}\prec \underline{\alpha}_{i}$ which implies $\underline{\alpha}_{i}\sim \underline{\alpha}_{j}$. This implies the theorem.
\end{proof}

\subsubsection{$\lambda$-degrees and the Hadamard product}\label{lambda-degrees}

We come back to the $\FF_p$-vector space $\boldsymbol{C}$ and consider the $K$-vector space $\boldsymbol{C}(K)$ it generates by extension of the scalars. We have the algebras $(\boldsymbol{C},+,\odot_\zeta)$ and $(\boldsymbol{C}(K),+,\odot_\zeta)$ that are both graded by the 
monoid $\boldsymbol{S}$, where the degree of $\mathcal{C}$ is $\binom{\Sigma}{n}$
with $\Sigma$ the type of $\mathcal{C}$ and $n$ its weight. We recall that the product $\odot_\lambda$ over the above vector spaces is not compatible with this grading. In order to present the following, we are now going to speak about {\em $\zeta$-degree} and {\em  $\zeta$-type} for our usual 
degree and type of a composition array.

We describe another grading structure over $\boldsymbol{C}$ and $\boldsymbol{C}(K)$, again by using the monoid $\boldsymbol{S}$. Let $\mathcal{C}$ be an admissible composition array as in 
(\ref{a-sample-of-C}). Then we can write:
$$\mathcal{C}=\begin{pmatrix}\emptyset & \cdots & \emptyset & \Sigma_1 & \emptyset & \cdots & \emptyset & \cdots\cdots & \Sigma_k & \emptyset & \cdots & \emptyset \\
n_{0,0} & \cdots & n_{0,r_0} & n_{1,0} & n_{1,1} & \cdots & n_{1,r_1} & \cdots\cdots & n_{k,0} & n_{k,1} & \cdots & n_{k,r_k} 
\end{pmatrix}.$$
The way this matrix is written highlights the presence of the neutral element $\emptyset$ of $\mathfrak{S}$ intertwined with elements of $\mathfrak{S}\setminus\{\emptyset\}$ in the first row. We define the {\em $\lambda$-type} of $\mathcal{C}$ as the 
element 
$$\Sigma_1^{q^{r_1}}\cdots\Sigma_k^{q^{r_k}}\in\mathfrak{S}.$$
Here, $\Sigma^{n}$ stands for $\Sigma\cdots\Sigma$ with $n$ factors. Note that the relations
(\ref{non-homogeneous-relations}) are $K$-linear dependence relations among elements of 
$\mathcal{L}(K)$ of the form $\lambda_A(\mathcal{C})$ with $\mathcal{C}$ of the same $\lambda$-type $\sg^{q^k}$ for $k\geq 1$, although different $\zeta$-types, as previously noticed. The $\lambda$-type naturally induces a $\lambda$-degree by just using the usual weight of a composition array.
We may conjecture the following property that will be studied in another work.

\begin{Conjecture}
The $\lambda$-degree defines a grading of the algebra $\mathcal{L}(K)$ by the monoid $\boldsymbol{S}$.
\end{Conjecture}

The interest in introducing the $\lambda$-type relies in the existence of another operation on $\mathcal{L}(K)$, throughout the {\em Hadamard product} of formal series. Let $F$ be any field.
The {\em diagonal} of a formal series
$$f\in F[[X_1,\ldots,X_n,Y_1,\ldots,Y_n]]$$ is defined in the following way. Write 
$$f=\sum_{i_1,\ldots,i_n,j_1,\ldots,j_n}f_{i_1,\ldots,i_n,j_1,\ldots,j_n}X_1^{i_1}\cdots X_n^{i_n}Y_1^{j_1}\cdots Y_n^{j_n},\quad f_{i_1,\ldots,i_n,j_1,\ldots,j_n}\in F.$$ We introduce the
operator 
$$\Delta_{1/2}(f):=\sum_{i_1,\ldots,i_n}f_{i_1,\ldots,i_n,i_1,\ldots,i_n}X_1^{i_1}\cdots X_n^{i_n}\in F[[X_1,\ldots,X_n]].$$ This is the {\em diagonal operator} considered in \cite{FUR}. 
Let $f,g$ be formal series of $F[[X_1,\ldots,X_n]]$ with $F$. Their {\em Hadamard product} 
is the formal series 
$$f\odot g:=\Delta_{1/2}\big(f(X_1,\ldots,X_n)g(Y_1,\ldots,Y_n)\big)\in F[[X_1,\ldots,X_n]].$$ 
We now suppose $F=\CC_\infty$. Let $\mathcal{L}^\lambda_{n,\Sigma}(K)$ be the $K$-vector space generated by the 
functions $\lambda_A(\mathcal{C})$ with $\mathcal{C}$ of $\lambda$-degree $\binom{\Sigma}{n}\in\boldsymbol{S}$. We have:

\begin{Lemma}\label{hadamard-product}
The Hadamard product defines a bilinear map
$$\mathcal{L}^\lambda_{m,\Sigma}(K)\times\mathcal{L}^\lambda_{n,\Sigma}(K)\xrightarrow{\odot_\Sigma}\mathcal{L}^\lambda_{m+n,\Sigma}(K).$$
\end{Lemma}
\begin{proof}
This follows directly from Lemma \ref{larger-lambda} and Part 1 of Proposition \ref{elementary-facts}.
\end{proof}
For example, we have 
$\lambda_A\binom{\sg}{1}\in\mathcal{L}^\lambda_{1,\sg}=\mathcal{L}_{1,\sg}$ and 
$\lambda_A\binom{\sg}{1}\odot_\sg\lambda_A\binom{\sg}{1}=\lambda_A\binom{\sg}{2}$.

\subsubsection{Remark}\label{formalism} The formalism of the algebras $\mathcal{Z}$ and $\mathcal{L}$ can be unified. Indeed one can construct, in the same manner, multiple polylogarithms having in the coefficients of the monomials (\ref{delta-sigma}), twisted power sums as in (\ref{powersums-def}).

\subsubsection{Remark}
	Assume that the composition array $\mathcal{C}$ has depth $r$. We note that to give an  interpretation of $\zeta_A(\mathcal{C})$ as entries of logarithms of special points, the first author \cite{Gezmis} used a different type of multiple polylogarithms defined as functions on a certain subset of $\mathbb{T}_{\Sigma}^r$ which can be seen as a higher depth generalization of the polylogarithms defined in \cite{ANG&PEL&TAV}. In the present paper, we instead focus on elements of the form $\lambda_A(\mathcal{C})\in \widehat{K[\underline{X}_\Sigma]}_{\|\cdot\|}$, a variant of Chang's multiple polylogarithms in \cite{CHA}.
	
\subsubsection{Remark} We end this section about multiple polylogarithms by pointing out that 
one can also consider partial higher divided derivatives on the algebra $\mathcal{L}$. In certain cases, one can observe properties in a certain sense dual to part of the statement of Proposition \ref{proposition-general properties} (twisting and evaluating at $t_i=\theta^{q^{k_i}}$ being assimilated to an integral operator on $\mathcal{Z}$ versus partial higher divided derivatives). Since these properties will not be used in this paper, we will omit the details. 

\section{Isomorphisms between the spaces ${\mathcal{Z}_{n,\Sigma}}$ and ${\mathcal{L}_{n,\Sigma}}$}
\label{decomposition-theorem} 

Let $\Sigma$ be a set in $\mathfrak{S}$. We recall that $\EE_\Sigma$ denotes the $\CC_\infty$-algebra of entire functions $\CC_\infty^\Sigma\rightarrow\CC_\infty$. We now construct a $\CC_\infty$-linear map
$$\mathcal{F}_\Sigma:\EE_\Sigma\rightarrow\CC_\infty[[\underline{X}_\Sigma]]^{\operatorname{lin}}.$$
If $t$ is a variable, we set $b_{0}(t):=1$ and for any $i\geq 1$ we define 
\[
b_i(t):=(t-\theta)\dots (t-\theta^{q^{i-1}})\in A[t].
\] 
We also set for further use, for $i>0$, $$b_{-i}(t):=\prod_{m=1}^{i}(t-\theta^{q^{-m}})^{-1}\in K^{\operatorname{perf}}(t)\cap\TT,$$ where $K^{\operatorname{perf}}$ denotes the perfect closure of $K$ in $\CC_\infty$.
We write $B_{\underline{i}}(\Sigma)=\prod_{j\in\Sigma}b_{i_j}(t_j)$ with $\underline{i}\in\ZZ^\Sigma$, $\underline{i}=(i_j:j\in\Sigma)$.
Noting that $b_i(t)\mid b_j(t)$ if $j\geq i$ we have the well defined inverse limit and the canonical $\CC_\infty$-algebra map
$$\CC_\infty[\underline{t}_\Sigma]\rightarrow\varprojlim_{\underline{i}}\frac{\CC_\infty[\underline{t}_\Sigma]}{B_{\underline{i}}(\Sigma)\CC_\infty[\underline{t}_\Sigma]},$$
with $\underline{i}=(i_j:j\in\Sigma)$. We also set $\underline{\theta}^{q^{\underline{i}}}:=(\theta^{q^{i_j}}:j\in\Sigma)$.
The evaluation map $\CC_\infty[\underline{t}_\Sigma]\ni f\mapsto f(\underline{\theta}^{q^{\underline{i}}})\in\CC_\infty$ extends in a unique way to the $\CC_\infty$-algebra $\varprojlim_{\underline{i}}\CC_\infty[\underline{t}_\Sigma]/B_{\underline{i}}(\Sigma)\CC_\infty[\underline{t}_\Sigma]$. 
\begin{Lemma}
For every $f\in\EE_\Sigma$ there exists a unique element 
\begin{equation}\label{Definition-Phi}
\Phi(f)=\sum_{\underline{i}}f_{\underline{i}}B_{\underline{i}}(\Sigma)\in\varprojlim_{\underline{i}}\frac{\CC_\infty[\underline{t}_\Sigma]}{B_{\underline{i}}(\Sigma)\CC_\infty[\underline{t}_\Sigma]},
\end{equation}
with $f_{\underline{i}}\in\CC_\infty$ for all $\underline{i}\in\NN^\Sigma$,
such that $\Phi(f)(\underline{\theta}^{q^{\underline{i}}})=f(\underline{\theta}^{q^{\underline{i}}})$ for all $\underline{i}\in\NN^\Sigma$. 
\end{Lemma}
\begin{proof}
It suffices to set
$$f_{\underline{i}}:=\sum_{\underline{j}\leq \underline{i}}\frac{f(\underline{\theta}^{q^{\underline{j}}})}{D_{\underline{j}}l^{q^{\underline{j}}}_{\underline{i}-\underline{j}}},$$
where with $\underline{j}=(j_k:k\in\Sigma)$ we write $D_{\underline{j}}=\prod_{k\in\Sigma}D_{j_k}$ and the other multi-index notations are defined similarly, including $\underline{j}\leq \underline{i}$
which means that $\underline{i}-\underline{j}\in\NN^\Sigma$.
\end{proof}

We set
\begin{equation}\label{definition-F}
\mathcal{F}_\Sigma(f):=\sum_{\underline{i}}f_{\underline{i}}\underline{X}_\Sigma^{q^{\underline{i}}}\in\CC_\infty[[\underline{X}_\Sigma]]^{\operatorname{lin}},\end{equation}
where $\underline{X}_\Sigma^{q^{\underline{i}}}=\prod_{j\in\Sigma}X_{j}^{q^{i_j}}$.

The $\FF_p$-vector space
\begin{equation}\label{Sigma-space}
\mathcal{Z}_{\Sigma }:=\sum_{n\geq 0}\mathcal{Z}_{n,\Sigma }\subset\mathcal{Z}(K[\underline{t}_\Sigma])\subset\EE_\Sigma
\end{equation} also has a structure of graded $\mathcal{Z}_\emptyset$-module, that is, a module over Thakur's multiple zeta values where the multiplication respects the grading induced by weights and types. A similar construction holds for the spaces $\mathcal{L}_{N,\Sigma}$ and gives rise to the graded $\mathcal{Z}_\emptyset$-module $\mathcal{L}_\Sigma$; to see this, use the remark after the statement of Theorem \ref{grading-structures}. The aim of this section is to prove:
\begin{Theorem}\label{Z-L}
If $\Sigma\in\mathfrak{S}$ is a set such that $|\Sigma|<q$ then the map $\mathcal{F}_\Sigma$
induces an isomorphism of graded $\mathcal{Z}_\emptyset$-modules
$$\mathcal{Z}_{\Sigma}\xrightarrow{\mathcal{F}_\Sigma}\mathcal{L}_{\Sigma}.$$ 
\end{Theorem}

Throughout this section, we fix a subset $\Sigma$ of $\NN^{*}$ with $|\Sigma|<q$. We recall that by convention, empty sums are considered to be zero and empty products are equal to one.
For disjoint subsets $U_1,\dots,U_r \subset \Sigma$ such that $\sqcup_iU_i=\Sigma$, we consider a basic admissible composition array $\mathcal{C}$ of $\binom{\Sigma}{n}$. We can write:
\begin{equation}\label{eq-signature}
\mathcal{C}=\begin{pmatrix}U_1 & \cdots & U_r & \emptyset & \cdots & \emptyset \\ n_1 & \cdots & n_r &  n_1' & \cdots & n_{r'}'\end{pmatrix}.
\end{equation}
 We note that some of the $U_i$'s are allowed to be empty. We say that $\mathcal{C}$ is of {\em signature} $\sign(\mathcal{C})=N>0$ if $$\sum_{\begin{smallmatrix}
 	i\in \{1,\dots,r\}\\ U_i\neq \emptyset
 	\end{smallmatrix}}n_i=N.$$ 

	We define
$$S_d^\dag(\mathcal{C})=\sum_{d=d_1>\cdots>d_{r'}\geq0}S_{d_1}(n_1)\cdots S_{d_{r'}}(n_{r'}')\prod_{i=1}^rB_{d_i}(U_i)\in K[\underline{t}_\Sigma],\quad d\geq 0$$ where $B_d(U)=\prod_{j\in U}b_d(t_j)$. Since $|\Sigma|<q$, the series 
$$\zeta_A^\dag(\mathcal{C}):=\sum_{d\geq 0}S_d^\dag(\mathcal{C})$$ converges for the Gauss norm to an entire function of $\EE_\Sigma$ in the variables $\underline{t}_\Sigma$ (this property is no longer guaranteed if $|\Sigma|\geq q$). The proof of Theorem \ref{Z-L} is based on the following.
\begin{Theorem}\label{signature-theorem} 
Assuming $|\Sigma|<q$, we can decompose 
\begin{equation}\label{identity-ot-th}
\zeta_A(\mathcal{C})=\zeta_A^\dag(\mathcal{C})+\sum_{i\in I}\alpha_i\zeta_A(\mathcal{C}_i)
\end{equation} where $I$ is a finite set,
the coefficients $\alpha_i$ are in $\FF_p$ and where $(\mathcal{C}_i)_{i\in I}$ is a family of admissible composition arrays of $\binom{\Sigma}{n}$ such that $\sign(\mathcal{C}_i)<N$.
\end{Theorem}
Let $\mathcal{Z}^{\dag}_{w,\Sigma}$ be the subvector space (over $\FF_p$) of $\EE_{\Sigma}$  generated by the elements $\zeta_A^\dag(\mathcal{C})$. Theorem \ref{signature-theorem} directly implies, by means of the property describing the signature:

\begin{Corollary}\label{coro-dag}
	Assume that $|\Sigma|<q$. Then for all $w>0$ we have $\mathcal{Z}^\dag_{w,\Sigma}=\mathcal{Z}_{w,\Sigma}$.
\end{Corollary}
In other words, with $\Phi$ being defined as in (\ref{Definition-Phi}), we have 
$\Phi(\mathcal{Z}_\Sigma)\subset\EE_\Sigma$, namely, for any $f\in\mathcal{Z}_\Sigma$ the series defining $\Phi(f)$ converges to $f$ for any valuation of $\CC_\infty[\underline{t}_\Sigma]$ extending to a valuation of $\EE_\Sigma$ and for all $w$,
$\mathcal{Z}^\dag_{w,\Sigma}=\Phi(\mathcal{Z}_{w,\Sigma})$ in agreement with Corollary \ref{coro-dag}.
In particular, we can define the graded $\mathcal{Z}$-module $\mathcal{Z}_{\Sigma}^\dag=\oplus_w\mathcal{Z}^{\dag}_{w,\Sigma}$ and we have the equality $\mathcal{Z}^\dag_{\Sigma}=\mathcal{Z}_{\Sigma}$.
The proof of Theorem \ref{signature-theorem} will be given in \S \ref{part-1} and \ref{part-2}.

\subsection{A formula for power sums}\label{part-1}

We start with a formula by Perkins (see \cite[Proposition 2.17]{PER}): For any subset $U\subset \NN$ such that $|U|<q$, we have

\begin{equation}\label{eq1}
\sum_{|a|<|\theta|^d}\Big(\prod_{i\in U}\chi_{t_i}(a)\Big)\frac{l_dE_d(z-a)}{z-a}=\prod_{i\in U}\sum_{j=0}^{d-1}b_j(t_i)E_j(z),
\end{equation}

where $E_d(z)=D_d^{-1}\prod_{|a|<|\theta|^d}(z-a)$ is the $d$-th Carlitz polynomial satisfying:

\begin{equation}\label{eq2}
C_a(X)=\sum_{i=0}^{\deg(a)}E_i(a)X^{q^i},\quad a\in A,\end{equation}
where $C_a(X)$ is the multiplication of $X$  by $a\in A$ for the Carlitz's $A$-module structure.

The formula (\ref{eq1}) holds in $K[\underline{t}_{U},z]$. Note that $E_d$ is $\FF_q$-linear
and $E_d(a)=0$ if $|a|<|\theta|^d$ by (\ref{eq2}) so that we can pull the factor $E_d(z-a)$ out of the sum in the left-hand side of (\ref{eq1}) as a factor in the left. 

Additionally, with $\mathcal{D}_n$ the partial divided derivative of order $n$ in the variable $z$, 
we have 
\begin{equation}\label{eq3}
\mathcal{D}_{n}(E_d)=D_i^{-1}l_{d-i}^{-q^i},\quad n=q^i\leq d,\quad \exists i\geq 0, 
\end{equation}
and $\mathcal{D}_n(E_d)=0$ otherwise ($D_i$ denotes the $i$-th coefficient of the $\tau$-expansion of $\exp_C\in K[[\tau]]$). In particular, $\mathcal{D}_1(E_d)=\frac{1}{l_d}$. We therefore get
$$\frac{1}{l_dE_d}=\frac{\mathcal{D}_1(E_d)}{E_d}=\sum_{|a|<|\theta|^d}\frac{1}{z-a}.$$
Hence, we can rewrite (\ref{eq1}) as follows, where $\sigma_U(a)=\prod_{i\in U}\chi_{t_i}(a)$:
\begin{equation}\label{eq1-bis}
\sum_{|a|<|\theta|^d}\frac{\sigma_U(a)}{z-a}=\left(\sum_{|a|<|\theta|^d}\frac{1}{z-a}\right)\prod_{i\in U}\sum_{j=0}^{d-1}b_j(t_i)E_j(z).
\end{equation}
We also set
$$H_{m}^{(d)}(t):=D_m^{-1}\sum_{j=m}^{d-1}b_j(t)l_{j-m}^{-q^m}\in K[t].$$
Then we have the expansion
\begin{equation}\label{expansion}\sum_{j=0}^{d-1}b_j(t)E_j(z)=\sum_{m=0}^{d-1}H_m^{(d)}(t)z^{q^m}.\end{equation}
Since $|U|<q$, we have 
$$\prod_{i\in U}\left(\sum_{j=0}^{d-1}b_j(t_i)E_j(z)\right)=\sum_{\begin{smallmatrix}\underline{m}=(m_i)_i\in\NN^U\\ m_i\leq q^{d-1},\forall i\in U\end{smallmatrix}}H_{\underline{m}}^{(d)}(U)z^{\sum_{i\in U}q^{m_i}},$$
with $H_{\underline{m}}^{(d)}(U):=\prod_{i\in U}H_{m_i}^{(d)}(t_i)$.
We use these formulas to rewrite the product
$$\Pi_d(U):=\left(\sum_{|a|<|\theta|^d}\frac{1}{z-a}\right)\prod_{i\in U}\left(T_i-\sum_{j=0}^{d-1}b_j(t_i)E_j(z)\right),$$ where the $T_i$'s are elements of $K[\underline{t}_\Sigma]$.
We get, expanding:
\begin{eqnarray}
\lefteqn{\Pi_d(U)=}\nonumber\\&=&\Big(\sum_{|a|<|\theta|^d}\frac{1}{z-a}\Big)\left(\prod_{i\in U}T_i+\sum_{\begin{smallmatrix}V_1\sqcup V_2=U\\ V_2\neq \emptyset\end{smallmatrix}}(-1)^{|V_2|}\Big(\prod_{j_1\in V_1}T_{j_1}\Big)\Big(\prod_{j_2\in V_2}\sum_{j=0}^{d-1}b_j(t_{j_2})E_j(z)\Big)\right)\nonumber\\
&=&\Big(\prod_{i\in U}T_i\Big)\Big(\sum_{|a|<|\theta|^d}\frac{1}{z-a}\Big)+
\sum_{\begin{smallmatrix}V_1\sqcup V_2=U\\ V_2\neq \emptyset\end{smallmatrix}}(-1)^{|V_2|}\Big(\prod_{j_1\in V_1}T_{j_1}\Big)\sum_{|a|<|\theta|^d}\frac{\sigma_{V_2}(a)}{z-a}\label{three-steps}.\end{eqnarray}
We have labelled the identity between $\Pi_d(U)$ with (\ref{three-steps}). The last step is done 
using (\ref{eq1-bis}).
Let $N$ be a non-negative integer. We apply $(-1)^N\mathcal{D}_N$ to both the left- and the right-hand sides of the above identity. In order to render our expressions in a more compact way we introduce a few additional notation.
For any natural number $N>0$ and a non-empty finite set $V\subseteq \NN^*$ we define the set $M_V(N)$:
\[
M_{V}(N):=\left\{\underline{m}=(m_j)_{j\in V}\in \NN^{V}:N-\sum_{j\in V}q^{m_j}\geq 0 \right\}.
\]
Furthermore, we set $M_{\emptyset}(N)=\emptyset$. We also write $q^{\underline{m}}$ at the place of $\sum_{j\in V}q^{m_j}$.
The left-hand side of (\ref{three-steps}) becomes, after application of $(-1)^N\mathcal{D}_N$ by using 
$$(-1)^N\mathcal{D}_N\left(\frac{1}{z-a}\right)=\frac{1}{(z-a)^{1+N}},$$
\begin{multline*}
\Big(\prod_{i\in U}T_i-\sum_{j=0}^{d-1}b_j(t_i)E_j(z)\Big)\sum_{|a|<|\theta|^d}\frac{1}{(z-a)^{N+1}}+\\ \sum_{V\subsetneq U}(-1)^{|V|}\sum_{\underline{m}\in M_V(N)}H_{\underline{m}}^{(d)}(V)\sum_{|a|<|\theta|^d}\frac{1}{(z-a)^{N-q^{\underline{m}}+1}},\end{multline*}
where we recall that $H_{\underline{m}}^{(d)}(V)=\prod_{j\in V}H_{m_j}^{(d)}(t_j)$ (and we extend the notation setting it to $0$ when $\underline{m}$ is not an element of $M_V(q^{(d-1)|V|})$). 

The right-hand side of  (\ref{three-steps}) becomes
$$\sum_{|a|<|\theta|^d}\frac{1}{(z-a)^{N+1}}\left(\prod_{i\in U}T_i+\sum_{\begin{smallmatrix}V_1\sqcup V_2=U\\ V_2\neq\emptyset\end{smallmatrix}}(-1)^{|V_2|}\Big(\prod_{j_1\in V_1}T_{j_1}\Big)\sigma_{V_2}(a)\right).$$
We replace
$T_i=t_i^d$ for all $i\in U$ and $z=\theta^d$.  Note that for any $i\in U$, we have $$t_i^d-\sum_{j=0}^{d-1}b_j(t_i)E_j(\theta^d)=b_d(t_i).$$ One way to prove this is to
observe that $C_{\theta^d}(\omega(t_i))=t_i^d\omega(t_i)=\sum_{i=0}^dE_i(\theta^d)b_i(t_i)\omega(t_i)$ and $\tau^d(\omega(t_i))=b_d(t_i)\omega(t_i)$. On the other hand, we have the identity
$$\prod_{i\in U}T_i+\sum_{\begin{smallmatrix}V_1\sqcup V_2=U\\ V_2\neq\emptyset\end{smallmatrix}}(-1)^{|V_2|}\Big(\prod_{j_1\in V_1}T_{j_1}\Big)\sigma_{V_2}(a)=\sigma_U(\theta^d-a).$$
Therefore the identity (\ref{three-steps}) becomes, after application of $(-1)^N\mathcal{D}_N$:
\begin{multline}\label{results}
S_d\binom{U}{N+1}=S_d(N+1)B_d(U)+\\ \sum_{V\subsetneq U}(-1)^{|V|}\sum_{\underline{m}\in M_V(N)}S_d(N-q^{\underline{m}}+1)H_{\underline{m}}^{(d)}(V),\quad d\geq 0,|U|<q.
\end{multline}

\subsubsection{Examples} (a) If we set $N=0$ in (\ref{results}) the inner sum is empty and therefore \begin{equation}\label{id-simplest}
S_d\binom{U}{1}=S_d^\dag\binom{U}{1},
\end{equation}
yielding that $\zeta_A\binom{U}{1}=\zeta_A^\dag\binom{U}{1}$.
This is well known since the formulas 
\begin{equation}\label{simpleformulas}
S_d\binom{\Sigma}{1}=\frac{\prod_{i\in \Sigma}b_d(t_i)}{l_d},\quad |\Sigma|<q
\end{equation} and $S_d(1)=\frac{1}{l_d}$ are themselves well known (see also \cite[Rem. 5.13]{APTR16}). However, the advantage of our approach is that we did not need to make any use of them. This is the crucial initial remark.

\noindent (b) Suppose now $|U|\leq N<q$. We get the formula
\begin{equation}\label{formula-N-smaller-q}
S_d\binom{U}{N+1}=S_d^{\dag}\binom{U}{N+1}+\sum_{V\subsetneq U}S_d(N-|V|+1)\prod_{ j\in V}H_{1}^{(d)}(t_j).
\end{equation}
\noindent (c) Finally suppose $U=\{i\}:=\sg$ and 
$N=q^k$. Then the formula yields:
\begin{equation}\label{eq6}
S_d\binom{\sg}{q^k}=b_dS_d(1)^{q^k}+\sum_{m<k}S_d(q^{k-m}-1)^{q^m}H_m^{(d)}(t_i).
\end{equation}
If $k=0$ we get $S_d\binom{\sg}{1}=b_d(t_i)S_d(1)$ (the sum over $m$ is empty by convention). 

\subsection{The polynomials $H_{m}^{(d)}$}\label{part-2}

In order to complete the proof of Theorem \ref{signature-theorem} we need to show that the coefficients 
$H_{\underline{m}}^{(d)}$ involved in (\ref{formula-N-smaller-q}) can be expressed as combinations of twisted power sums; this is what we do in this subsection.
To ease the notation, we set $H_m^{(d)}:=H_{m}^{(d)}(t_i)$ and use the variable $t$ instead of $t_i$ any time we are focusing on one variable only. We recall that $H_{\underline{m}}^{(d)}(V)$ is
the coefficient of $z^{q^{\underline{m}}}$ in the polynomial $\prod_{i\in V}\sum_{j=0}^{d-1}b_j(t_i)E_j(z)$. We also recall the assumption that $|V|<q$. In this subsection we prove formulas which allow to interpret it as $\FF_p$-combinations of partial sums of multiple zeta values in $\mathcal{Z}_{V}$. We rewrite the $|V|$-variable case of (\ref{eq1-bis}) as:
\begin{equation}\label{identity-new-form}
\prod_{i\in V}\sum_{j=0}^{d-1}b_j(t_i)E_j(z)=\frac{\sum_{|a|<|\theta|^d}\frac{\sigma_V(a)}{z-a}}{\sum_{|a|<|\theta|^d}\frac{1}{z-a}}\in \FF_q(\theta)(z)[\underline{t}_V].\end{equation}
We want to expand the right-hand side in ascending powers of $z$. We set $F_d:=\sum_{|a|<|\theta|^d}\frac{\sigma_V(a)}{z-a}$. In $\FF_q(\theta)[\underline{t}_V][[z]]$ we have
$$F_d=\sum_{k\geq 0}z^{k(q-1)}S_{<d}\binom{V}{k(q-1)+|V|}.$$ Similarly, setting $G_d=\sum_{|a|<|\theta|^d}\frac{1}{z-a}$,
we compute
$$G_d=\frac{1}{z}+\sum_{k\geq 1}z^{k(q-1)-1}S_{<d}(k(q-1)).$$ Expanding $\frac{F_d}{G_d}$ we therefore find:
\begin{equation}\label{E:division}
\frac{F_d}{G_d}=\sum_{k\geq0}z^{k(q-1)+|V|}\kappa^{(d)}_{k(q-1)+|V|}
\end{equation}
where $\kappa^{(d)}_{k(q-1)+|V|}$ can be expressed in the following way:
$$\kappa^{(d)}_{k(q-1)+|V|}=S_{<d}\binom{V}{k(q-1)+|V|} +\mu_{k,1}S_{<d}\binom{V}{(k-1)(q-1)+|V|}+\cdots+
\mu_{k,k}S_{<d}\binom{V}{1},$$ where $\mu_{k,m}$, independent of $d$, are polynomials with coefficients in $\FF_p$ which are
homogemeous of weight $m$ evaluated in the terms $S_{<d}(h(q-1))$ (assumed to be of weight $h$). Here are some examples in the case $V=\sg$ a singleton:
\begin{eqnarray*}
\kappa^{(d)}_{1}&=&S_{<d}\binom{\sg}{1}\\
\kappa^{(d)}_{q}&=&S_{<d}\binom{\sg}{q}+\underbrace{(-S_{<d}(q-1))}_{=:\mu_{1,1}}S_{<d}\binom{\sg}{1}\\
\kappa^{(d)}_{2q-1}&=&S_{<d}\binom{\sg}{2q-1}+\underbrace{(-S_{<d}(q-1))}_{=:\mu_{2,1}}S_{<d}\binom{\sg}{q}+\underbrace{(-S_{<d}(2(q-1)))}_{=:\mu_{2,2}}S_{<d}\binom{\sg}{1}.
\end{eqnarray*}
Coming back to (\ref{identity-new-form}) we see that $\kappa^{(d)}_{k(q-1)+|V|}=0$ if $k(q-1)+|V|$ is not of the form $q^{\underline{m}}$ (note that this yields a collection of non-trivial identities). By using part (2) of Proposition \ref{elementary-facts} we get:
\begin{Lemma}\label{lemma-Hm}
Assuming that $|V|<q$, for all $\underline{m}\in\NN^V$ we can expand:
$$H_{\underline{m}}^{(d)}(V)=\sum_{j\in J}\nu_jS_{<d}(\mathcal{C}_j),\quad d\geq 0$$
where $J$ is a finite set, $(\nu_j)_{j\in J}$ is a family of elements of $\FF_p$ independent of $d$ and $(\mathcal{C}_j)_{j\in J}$ are admissible composition arrays of $\binom{V}{q^m}$ also independent of $d$.
\end{Lemma}

See Lemma \ref{Hdm-explicit} for a much more refined version of the above statement.

We come back to the identity (\ref{results}). Combining with Lemma \ref{lemma-Hm} and applying Part (3) of Proposition \ref{elementary-facts}, we obtain:

\begin{Lemma}\label{lemma-chi-dag}
If $|U|<q$, for $N\geq 0$, there exist a finite set $\mathcal{J}$ and, for $j\in J$, elements $\epsilon_j\in\FF_p$ and composition arrays $\mathcal{E}_j$ of $\binom{U}{N+1}$ of signature $<N+1$ such that
for all $d\geq0$ 
$$S_d\binom{U}{N+1}=S_d^\dag\binom{U}{N+1}+\sum_{j\in J}\epsilon_jS_d(\mathcal{E}_j).$$
\end{Lemma}

\begin{proof}[Proof of Theorem \ref{signature-theorem}]
By using Lemma \ref{lemma-chi-dag} and repeated application of Proposition \ref{elementary-facts} we see that
$$S_d(\mathcal{C})=\sum_{d=d_1>\cdots>d_r}S_{d_1}^\dag\binom{U_1}{n_1}\cdots
S_{d_r}^\dag\binom{U_r}{n_r}S_{<d_r}(n'_1,\ldots,n'_{r'})+\sum_k\delta_kS_d(\mathcal{H}_k),$$ 
with the coefficients $\delta_k\in\FF_p$ independent of $d$ and the composition arrays $\mathcal{H}_k$ of
$\binom{U}{n}$ also independent of $d$ and of smaller signature.
\end{proof}

\begin{proof}[Proof of Theorem \ref{Z-L}]
If the sequence $f_{\underline{i}}B_{\underline{i}}(\Sigma)$ converges to the zero function
uniformly on every bounded subset of $\CC_\infty^\Sigma$, then $\mathcal{F}_\Sigma(f)\in\TT_\Sigma^{\operatorname{lin}}$, the subset of the Tate algebra $\TT_\Sigma$ in the variables $\underline{X}_\Sigma$ of the series which are $\FF_q$-linear in each variable. This is precisely what happens if 
$$f=\zeta^\dag_A(\mathcal{C}),$$
where 
 $\mathcal{C}$ is an admissible composition array of $\binom{\Sigma}{n}$ with $|\Sigma|<q$. 
 Furthermore, $$\mathcal{F}_\Sigma(\zeta_A^\dag(\mathcal{C}))=\lambda_A(\mathcal{C}).$$ This defines an isomorphism of $\mathcal{Z}_\emptyset$-modules.
\end{proof}
For example, the formula (\ref{id-simplest}) immediately implies:
\begin{equation}\label{image-basic-zeta}
\mathcal{F}_{\sg}\left(\zeta_A\binom{\sg}{1}\right)=\lambda_A\binom{\sg}{1}.
\end{equation}

\subsubsection{Remark}\label{remark-carlitz-formalism}

The map $\mathcal{F}_\Sigma$ defined in (\ref{definition-F}) and restricted on polynomials agrees with the formalism of the Carlitz action as in \cite[\S 3.4]{ANG&PEL&TAV}. More precisely, if $f\in\CC_\infty[\underline{t}_\Sigma]$, then 
$\mathcal{F}_\Sigma(f)=f*\underline{X}_\Sigma$ with the $*$-action defined in ibid. (3.2). However, 
in contrast with the above reference, we do not work in the completion of $\CC_\infty[\underline{X}_\Sigma]$ with respect to the norm $\|\cdot\|_C$ defined there; this is another path of investigation that will be developed elsewhere, which however forbids evaluation at $X_i=1$, which we examine here, where we rather work in the completion 
with respect to the Gauss norm $\|\cdot\|$.

\section{Analytic properties of multiple zeta values in Tate algebras}

At the time being, we do not know if the functions $\mathcal{F}_\Sigma$ extend to a unique map defined over $\mathcal{Z}$.
Our next task in this paper is to give alternative definitions to the maps $\mathcal{F}_\Sigma$ so that 
they can be extended to a larger space of functions allowing to investigate their analytic properties. We will therefore introduce the maps $\mathcal{E}_\Sigma$, that turn out to coincide with the $\mathcal{F}_\Sigma$'s. Firstly, in \S \ref{perkins-with}, we recall a formula for twisted power sums that 
the second author obtained in collaboration with Perkins. Then in \S\ref{first-step} we introduce the algebra $\EE_\Sigma$ of entire functions in the variables $t_i$ with $i\in\Sigma$ and consider a (non-exhaustive) filtration $\EE_\Sigma^{[0]}\subset \EE_\Sigma^{[1]}\subset\cdots$. Finally in \S \ref{defi-E} we introduce the map 
 $\mathcal{E}_\Sigma$ which turns out to agree with $\mathcal{F}_\Sigma$ but can be defined analytically. Indeed what we have is that, at once, $\mathcal{E}_\Sigma$ is defined over $\EE_\Sigma^{[0]}$, and also, $\mathcal{Z}_\Sigma\subset\EE_\Sigma^{[0]}$ if $|\Sigma|<q$. In \S \ref{Trivial-elements} we introduce our trivial multiple zeta values in Tate algebras and we prove Theorem C of the introduction.

\subsection{A formula for twisted power sums}\label{perkins-with}

Using the results of \cite{PEL&PER2}, we are able to relate the twisted power sums to a family of polynomials introduced by the second author and Perkins which we recall now. Let $\Sigma$ be a finite subset of $\NN^*$ such that $|\Sigma|=m(q-1)+1$, with $m>0$.
In \cite[Thm. 1]{PEL&PER2} it is proved that if $d\geq m$, then
$$
S_{<d}\binom{\Sigma}{1}=l_{d-1}^{-1}\prod_{i\in\Sigma}b_{d-m}(t_i)\HH_{\Sigma}(Y)_{Y=\theta^{q^{d-m}}}.
$$
We recall that in this formula, $\HH_{\Sigma}\in A[Y][\underline{t}_\Sigma]$ is of degree $$\mu:=\frac{q^m-1}{q-1}-m$$ in $Y$ and of degree $m-1$ in $t_i$ for all $i\in\Sigma$. Using  \cite[Thm. 3.3.6]{DemeslayTh} we can prove the following refinement of \cite[Thm. 1]{PEL&PER2} which removes the condition $d\geq m$. 
\begin{Lemma}\label{pel-per-refined} If $\Sigma$ is a subset of $\NN^*$ such that $|\Sigma|=m(q-1)+1$ with $m>0$ then there exists $\HH_\Sigma\in A[Y][\underline{t}_\Sigma]$ of degree $\mu=\frac{q^m-1}{q-1}-m$
	in $Y$ and degree $m-1$ in each of the variables $t_i$, $i\in\Sigma$, such that 
	\begin{equation}\label{formula-per-per}
	S_{<d}\binom{\Sigma}{1}=l_{d-1}^{-1}\prod_{i\in\Sigma}b_{d-m}(t_i)\HH_{\Sigma}(Y)_{Y=\theta^{q^{d-m}}},\quad d\geq 0,
	\end{equation}
	and $\mathbb{H}_\Sigma$ is uniquely determined by these properties.
\end{Lemma}
\begin{proof} For the convenience of the reader we give all the details. The statement of the lemma is true if $d\geq m$ by \cite[Thm. 1]{PEL&PER2}. Hence we know that there exists $\mathbb{H}_\Sigma$ unique, with the above properties, such that the above identity holds for $d\geq m$ and we need to extend the identities to all $d$. We recall that in his Ph. D. Thesis \cite{DemeslayTh}, Demeslay proved, by using a log-algebraic-type result, the following property (his result is more general but we are going to describe only the aspect that is important to us): there exists a unique $Q\in K^{\text{perf}}(\underline{t}_\Sigma)[x]$ such that for all $d\geq 0$,
	\begin{equation}\label{demeslay-identity}
	S_d\binom{\Sigma}{1}=l_d^{-1}\tau^{d}(Q)_{x=\theta}\prod_{i\in\Sigma}b_d(t_i),
	\end{equation} where $x$ is a new indeterminate. This identity holds in fact for any $\Sigma$, regardless of the congruence condition imposed previously. 
	
	Let $d$ be a non-negative integer. Note that for any $d,m\geq 0$, $b_{d-m}(x)=b_d(x)\tau^d(b_{-m}(x))$. Then, 
	\begin{eqnarray*}
		\lefteqn{S_d\binom{\Sigma}{1}=S_{<d+1}\binom{\Sigma}{1}-S_{<d}\binom{\Sigma}{1}=}\\
		&=&l_d^{-1}\prod_{i\in\Sigma}b_d(t_i)\left(\prod_{i\in\Sigma}\tau^d(b_{-m}(t_i))\right)\left[\prod_{i\in\Sigma}(t_i-\theta^{q^{d-m}})\mathbb{H}_\Sigma(\theta^{q^{d-m+1}})-(\theta-\theta^{q^d})\mathbb{H}_\Sigma(\theta^{q^{d-m}})\right],
	\end{eqnarray*}
	the second identity being valid if $d\geq m$. Note that 
	$b_{-m}(t)\in\FF_q(1/\theta^{q^{-m}},t)$ and we can write $\prod_{i\in\Sigma}\tau^d(b_{-m}(t_i))=\tau^d(f(\theta^{q^{-m}},\underline{t}_\Sigma))\in\FF_q(\theta^{q^{-m}},\underline{t}_\Sigma)$, for all $d\geq 0$, for some $f\in\FF_q(\theta,\underline{t}_\Sigma)$. Similarly, we can write
	$$\left(\prod_{i\in\Sigma}(t_i-\theta^{q^{d-m}})\right)\mathbb{H}_\Sigma(\theta^{q^{d-m+1}})-(\theta-\theta^{q^d})\mathbb{H}_\Sigma(\theta^{q^{d-m}})=\tau^d(g_\Sigma(x,\underline{t}_\Sigma))_{x=\theta},$$ with $g_\Sigma\in\FF_q[\theta^{q^{-m}}][x,\underline{t}_\Sigma]$, and where $\tau$ is extended $\FF_q[x,\underline{t}_\Sigma]$-linearly. 
	By uniqueness, for all $d\in\ZZ$,
	$$\tau^d(Q)_{x=\theta}=\prod_{i\in\Sigma}(t_i-\theta^{q^{d-m}})\mathbb{H}_\Sigma(\theta^{q^{d-m+1}})-(\theta-\theta^{q^d})\mathbb{H}_\Sigma(\theta^{q^{d-m}})$$
	and therefore, $Q=fg_\Sigma$. Replacing in (\ref{demeslay-identity}), summing and telescoping, we complete the proof of the lemma. 
\end{proof}

\subsubsection{Remark: link with log-algebraic theorems} Lemma \ref{pel-per-refined} is, in a certain sense, stronger than certain instances of Anderson's log-algebraicity theorem \cite[Theorem 3]{AndLog}. To illustrate this, we will use the formalism of {\em Stark units}
as in the work of Angl\`es, Ngo Dac, Tavares Ribeiro and the second author \cite{ANG&NGO&TAV,ANG&PEL&TAV}.

Let $\delta$ be the degree in $\theta$ of the polynomial $\mathbb{H}_\Sigma$ and denote by $\tau$ the $\FF_q[\underline{t}_\Sigma][Y]$-linear extension of the map $c\mapsto c^q$, $c\in A$,
to $A[\underline{t}_\Sigma][Y]$. We recall that $E_k(a)=0$ for all $a\in A$ such that $k>\deg_\theta(a)$, where $E_k$ is the $k$-th Carlitz's polynomial. In particular, writing
$E_k=\sum_{j=0}^kD_j^{-1}l_{k-j}^{-q^j}\tau^j$, we have $E_k(\mathbb{H}_\Sigma)=0$ for all $k>\delta$. We compute:
$$\mathcal{A}_k:=\sum_{i+j=k}D_i^{-1}B_i(\Sigma)\tau^i\Big(S_j\binom{\Sigma}{1}\Big).$$ By using 
(\ref{formula-per-per}) we note that
\begin{eqnarray*}
\mathcal{A}_k&=&\sum_{i+j=k}D_i^{-1}l_j^{-q^i}B_i(\Sigma)\tau^i(B_{j-m+1})\tau^i(\mathbb{H}_\Sigma)_{Y=\theta^{q^{k-m+1}}}\\
&=&B_{k-m+1}(\Sigma)E_k(\mathbb{H}_\Sigma)_{Y=\theta^{q^{k-m+1}}}.
\end{eqnarray*}
We see that $\mathcal{A}_k\in A[\underline{t}_\Sigma]$ for all $k\geq 0$. For $k\geq m$ this is obvious. For $0\leq k\leq m-1$, $\mathcal{A}_k$ belongs to $A^{\operatorname{perf}}[\underline{t}_\Sigma]\cap K_\infty[\underline{t}_\Sigma]$ where $A^{\operatorname{perf}}=\cup_{i\geq 0}A^{1/p^i}$. This intersection is easily seen to be equal to $A[\underline{t}_\Sigma]$. Additionally, note that 
$\mathcal{A}_k=0$ for all $k>\delta$.

Now consider the formal series $$\zeta_A^\diamondsuit(1;\Sigma):=\sum_{i\geq 0}z^iS_i\binom{\Sigma}{1}\in K[\underline{t}_\Sigma][[z]].$$
The exponential function associated with the (non-uniformisable) Drinfeld module of rank one 
$\phi^\diamondsuit$ defined by 
$$\phi_\theta^\diamondsuit:=\theta+zB_1(\Sigma)\tau$$
where $z$ is a central variable, is
$$\exp_{\phi^\diamondsuit}=\sum_{i\geq 0}z^iD_i^{-1}B_i(\Sigma)\tau^i.$$
We get 
$$\exp_{\phi^\diamondsuit}(\zeta_A^\diamondsuit(1;\Sigma))=\sum_{k\geq 0}z^k\sum_{i+j=k}D_i^{-1}B_i(\Sigma)\tau^i\Big(S_j\binom{\Sigma}{1}\Big)=\sum_{k\geq 0}z^k\mathcal{A}_k.$$ Hence
$\exp_{\phi\diamondsuit}(\zeta_A^\diamondsuit(1;\Sigma))\in A[\underline{t}_\Sigma][z]$ is a polynomial of degree $\leq \delta$ in $z$ and this is known to imply a particular statement of log-algebraicity (see \cite{APTR16}).

\subsection{The spaces $\mathbb{E}_{\Sigma}^{[N]}$}\label{first-step} 
Let $\Sigma$ be a finite subset of $\NN^*$ of cardinality $s$. Recall that if $\underline{k}=(k_j:j\in\Sigma)\in\NN^{\Sigma}$, we set $\underline{\theta}^{q^{\underline{k}}}=(\theta^{q^{k_j}}:j\in\Sigma)$. If $N\in\ZZ$, we also set $\underline{k}+N=(k_j+N:j\in\Sigma)$. We also write $D_{\underline{k}}$ in place of $\prod_{j\in\Sigma}D_{k_j}$. For $N\in\NN$
we denote by $\EE_{\Sigma}^{[N]}$ the subvector space of the $\CC_\infty$-vector space $\EE_\Sigma$ generated by the functions $f$ such that
$\frac{f(\underline{\theta}^{q^{\underline{k}}})}{D_{\underline{k}+N}}\rightarrow0$ as $|\underline{k}|\rightarrow\infty$. We have
$\EE_{\Sigma}^{[N]}\subset\EE_{\Sigma}^{[N']}$ if $N\leq N'$.
We denote by $\tau_\Sigma$ the $\FF_q[\underline{t}_\Sigma]$-linear map 
$$\tau_\Sigma:=B_1(\Sigma)\tau=\prod_{j\in\Sigma}(t_j-\theta)\tau.$$
We have:

\begin{Lemma}\label{6.2}
The following properties hold:
\begin{enumerate}
\item For $N\geq 0$, $\tau_\Sigma(\EE_\Sigma^{[N]})\subset\EE_\Sigma^{[N]}$. 
\item If $U\subset \Sigma$, $\Sigma'=\Sigma\setminus U$ and $\underline{r}\in\NN^U$, then
for any $f\in\EE_\Sigma^{[N]}$, $f_{\underline{t}_U=\underline{\theta}^{q^{\underline{r}}}}\in\EE^{[N]}_{\Sigma'}$.
\item We have $\tau^{-N}(\EE_\Sigma^{[0]})\subset\EE_\Sigma^{[N]}$.
\end{enumerate}
\end{Lemma}

\begin{proof}
Let $f$ be an element of $\EE_\Sigma^{[N]}$. 
Note that if $\underline{k}+N$ has a vanishing entry, then $\tau_\Sigma(f)(\underline{\theta}^{q^{\underline{k}+N}})=0$. Observing that $D_i=(\theta^{q^i}-\theta)D_{i-1}^q$ for $i>0$, we see that  for all $\underline{k}\in\NN^\Sigma$ such that $\underline{k}+N$ has no vanishing entry,
$$\frac{\tau_\Sigma(f)(\underline{\theta}^{q^{\underline{k}+N}})}{D_{\underline{k}+N}}=\frac{\prod_{j\in\Sigma}(\theta^{q^{k_j+N}}-\theta)}{\prod_{j\in\Sigma}(\theta^{q^{k_j+N}}-\theta)}
\left(\frac{f(\underline{\theta}^{q^{\underline{k}-1}})}{D_{\underline{k}-1+N}}\right)^q\rightarrow0$$
as $|\underline{k}|\rightarrow\infty$. The second property is obvious.
For the third we observe that if $f\in\EE_\Sigma^{[0]}$ then
$$\frac{\tau^{-N}(f)(\underline{\theta}^{q^{\underline{k}-N}})}{D_{\underline{k}}^{1/q^N}}=\left(\frac{f(\underline{\theta}^{q^{\underline{k}}})}{D_{\underline{k}}}\right)^{\frac{1}{q^N}}\rightarrow0$$ as $|\underline{k}|\rightarrow\infty$. Hence, $\frac{\tau^{-N}(f)(\underline{\theta}^{q^{\underline{j}}})}{D_{\underline{j}+N}}\rightarrow0$.
\end{proof}
The interest of the spaces $\EE_\Sigma^{[N]}$ is inscribed in the next result.

\begin{Proposition}\label{four-properties}
Let $f$ be an element of $\mathcal{Z}_\Sigma$. 
The following properties hold.
\begin{enumerate}
\item There exists $N\geq 0$ such that $f\in\EE_\Sigma^{[N]}$. 
\item There exists $m\in\NN$ such that $\tau^m(f)\in\EE_\Sigma^{[0]}$. 
\item If $\mathcal{C}$ is an admissible composition array of type $\Sigma'$ such that the second row is $(1,\ldots,1)$, then $\tau^{m-1}(f)\in\EE_\Sigma^{[0]}$ with $f=\zeta_A(\mathcal{C})$ and 
$m$ such that $m(q-1)\geq|\Sigma'|$.
\item If $|\Sigma|<q$, $\mathcal{Z}_\Sigma\subset\EE_\Sigma^{[0]}$.
\end{enumerate}
\end{Proposition}

\begin{proof}
The property (4) immediately follows from (3). We prove (3). We first focus on certain estimates involving twisted power sums. By \cite[Theorem 7]{PEL&PER2} we see that 
$$S_d\binom{\Sigma'}{1}=l_d^{-1}B_{d-m+1}(\Sigma')\mathbb{H}(Y)_{Y=\theta^{q^{d-m+1}}},$$
where $\mathbb{H}$ is a polynomial in $A[\underline{t}_{\Sigma'}][Y]$ of degree $\leq m-1$ in each 
variable $t_i$ with $i\in\Sigma'$ and degree $\leq \frac{q^m-1}{q-1}-m$ in $Y$ (this polynomial can be computed by taking the coefficient of $\prod_{i\in\Sigma\setminus\Sigma'}t_i^{d}$ with $\Sigma\supsetneq\Sigma'$ such that $|\Sigma|=m(q-1)+1$).

We have:
\begin{Lemma}\label{auxiliarylemma}
Let $U$ be a subset of $\Sigma'$ and $\underline{k}\in\ZZ^U$ with the property that $k:=\min_ik\geq 1-m$. Then we have:
\begin{enumerate}
\item If there exists $i\in U$ such that $k_i\in[0,d-m]$, $S_d\binom{\Sigma'}{1}_{\underline{t}_U=\underline{\theta}^{q^{\underline{k}}}}=0$.
\item Otherwise let $U_+=\{i\in U:k_i\geq d-m+1\}$ and $U_-:=\{i\in U:k_i\in[1-m,0[\}$, so that 
$U=U_+\sqcup U_-$. Then:
$$\bigg\|S_d\binom{\Sigma'}{1}_{\underline{t}_U=\underline{\theta}^{q^{\underline{k}}}}\bigg\|\leq|\theta|^{H_1},$$
where $$H_1:=-q^{d-m+1}\left(m-\frac{|\Sigma'\setminus U_+|}{q-1}\right)+d\sum_{j\in U_+}q^{k_j}+\kappa_1$$ where $\kappa_1\in\QQ$ does not depend on $\underline{k}$ and $d$.
\end{enumerate}
\end{Lemma}
\begin{proof}
Since the first property of the lemma is clear, we assume that for all $i$, $k_i\in[1-m,0[\sqcup[d-m+1,\infty[$. We have $|l_d^{-1}|=|\theta|^{-q\frac{q^d-1}{q-1}}$ and 
\begin{multline*}
\big\|B_{d-m+1}(\Sigma')_{\underline{t}_U=\underline{\theta}^{q^{\underline{k}}}}\big\|=\\ =
\big\|B_{d-m+1}(\Sigma'\setminus U_+)\big\|\cdot|\theta|^{(d-m+1)\sum_{j\in U_+}q^{k_j}}=|\theta|^{\frac{q^{d-m+1}-1}{q-1}|\Sigma'\setminus U_+|}|\theta|^{(d-m+1)\sum_{j\in U_+}q^{k_j}}.
\end{multline*}
Moreover,
$$\big\|\mathbb{H}_{Y=\theta^{q^{d-m+1}},\underline{t}_U=\underline{\theta}^{q^{\underline{k}}}}\big\|\leq|\theta|^{(\frac{q^m-1}{q-1}-m)q^{d-m+1}}|\theta|^{(m-1)\sum_{j\in U_+}q^{k_j}+\widehat{\kappa}_1}$$
where $\widehat{\kappa}_1\in\QQ$ is independent of $\underline{k}$ and $d$.
We note that
$$|l_d^{-1}|\big\|\mathbb{H}_{Y=\theta^{q^{d-m+1}},\underline{t}_U=\underline{\theta}^{q^{\underline{k}}}}\big\|\leq |\theta|^{-(m+\frac{1}{q-1})q^{d-m+1}+(m-1)\sum_{j\in U_+}q^{k_j}+\widetilde{\kappa}_1},$$ for another rational number $\widetilde{\kappa}_1$ independent of 
$\underline{k}$ and $d$. The lemma follows.
\end{proof}
\begin{Lemma}
We consider $\Sigma\supsetneq \Sigma'\supset U$ subsets of $\NN^*$ with $|\Sigma|=m(q-1)+1$
and $m>0$. Then, uniformly in $d\geq m$, we have 
$$\left\|\frac{\tau^{m-1}\Big(S_d\binom{\Sigma'}{1}\Big)_{\underline{t}_U=\underline{\theta}^{q^{\underline{k}}}}}{D_{\underline{k}}}\right\|\rightarrow0$$ as $|\underline{k}|\rightarrow\infty$.
\end{Lemma}
\begin{proof}
Let $\underline{k}$ be an element of $([-1,\infty[\cap\ZZ)^U$. By Lemma \ref{auxiliarylemma} if there exists $i\in U$ such that $k_i\in[0,d-m]$ then we have 
$S_d\binom{\Sigma'}{1}_{\underline{t}_U=\underline{\theta}^{q^{\underline{k}}}}=0$. Otherwise, from Lemma \ref{auxiliarylemma} we deduce that
\begin{equation}\label{lhsH2}
\left\|\frac{S_d\binom{\Sigma'}{1}_{\underline{t}_U=\underline{\theta}^{q^{\underline{k}}}}}{D_{\underline{k}+m-1}^{1/q^{m-1}}}\right\|\leq |\theta|^{H_2},
\end{equation}
where $$H_2:=-q^{d-m+1}\underbrace{\left(m-\frac{|\Sigma'\setminus U_+|}{q-1}\right)}_{>0}+\underbrace{\sum_{j\in U_+}(d-k_j-m+1)q^{k_j}}_{\leq 0}-\sum_{j\in U_-}q^{k_j}\underbrace{(k_j+m-1)}_{\geq 0}+\kappa_2$$ where $\kappa_2\in\QQ$ does not depend on $\underline{k}$ and $d$. We deduce that the left-hand side of (\ref{lhsH2}) tends to $0$ as $|\underline{k}|\rightarrow\infty$ uniformly in $d$ (the leading term when $\max\underline{k}\neq d-m+1$ corresponds to the sum over $U_+$, otherwise it is the term with the factor 
$(m-\frac{|\Sigma'\setminus U_+|}{q-1})$).
Now we observe the following identity, valid for $f\in\EE_\Sigma$, $\underline{r}\in\ZZ^U$:
$$\left\|\frac{\tau^N(f)_{\underline{t}_U=\underline{\theta}^{q^{\underline{r}}}}}{D_{\underline{r}}}\right\|=\left\|\left(\frac{f}{D_{\underline{r}}^{1/q^N}}\right)_{\underline{t}_U=\underline{\theta}^{q^{\underline{r}-N}}}^{q^N}\right\|.$$
Hence, for $\underline{k}\in[1-m,\infty[^U$:
$$\left\|\frac{\tau^{m-1}\Big(S_d\binom{\Sigma'}{1}\Big)_{\underline{t}_U=\underline{\theta}^{q^{\underline{k}+m-1}}}}{D_{\underline{k}+m-1}}\right\|=\left\|\left(\frac{S_d\binom{\Sigma'}{1}}{D_{\underline{k}+m-1}^{1/q^{m-1}}}\right)_{\underline{t}_U=\underline{\theta}^{q^{\underline{r}-N}}}^{q^N}\right\|,$$
and the lemma follows.
\end{proof}
We can finish the proof of Proposition \ref{four-properties}.
The previous lemma implies the property (3) of Proposition \ref{four-properties}. For the property (2), we note that for any admissible composition array $\mathcal{C}$ of type $\Sigma'$ there exists a composition array $\mathcal{D}$ of type $\Sigma$ with the second line equal to $(1,\ldots,1)$, $N\geq 0$ and $\underline{r}\in\NN^{U}$ where $U=\Sigma\setminus\Sigma'$, with $\max\underline{r}<N$, such that 
$$\zeta_A(\mathcal{C})=\tau^N(\zeta_A(\mathcal{D}))_{\underline{t}_U=\underline{\theta}^{q^{\underline{r}}}}.$$ Therefore this property follows by using the second part of Lemma \ref{6.2}.
Finally, the property (1) follows from the third part of Lemma \ref{6.2}.\end{proof}
\subsubsection{Remark}\label{some-elements-notin-E0} For general $\Sigma$ not all the elements of $\mathcal{Z}_\Sigma$ lie in $\EE_\Sigma^{[0]}$. We give here an explicit example, where $j\in\NN^*$, $\sg=\{j\}$ and $t=t_j$. We recall from \cite[Theorem 14]{PEL} that
\begin{equation}\label{th14-Parma}
\deg_\theta\Big(\zeta_A\binom{\sg}{2}_{t=\theta^{q^k}}\Big)=(k-1)q^k-\frac{2q(q^{k-1}-1)}{q-1},\quad k\geq 1.
\end{equation}
Now consider $\Sigma\subset\NN^*$ finite, such that $|\Sigma|=q^2$ and set $f=\zeta_A\binom{\Sigma}{2}\in\EE_\Sigma$.
\begin{Lemma}
	If $q\geq 3$ then $f\not\in\EE_\Sigma^{[0]}$.
\end{Lemma}
\begin{proof}
	Assuming the converse, we should have
	$\frac{f(\underline{\theta}^{q^{\underline{k}}})}{D_{\underline{k}}}\rightarrow0$. In particular, considering $k\in\NN$ and setting $\underline{k}=(k,\ldots,k)\in\NN^\Sigma$,
	$f(\underline{\theta}^{q^{\underline{k}}})=\zeta_A\binom{\sg}{2}_{t=\theta^{q^{k+2}}}$ of degree
	in $\theta$ equal to $(k+1)q^{k+2}-\frac{2q(q^{k+1}-1)}{q-1}$. At once, we have $D_{\underline{k}}=D_k^{q^2}$ of degree $kq^{k+2}$. We deduce that, with this choice of $\underline{k}$, 
	$$\deg_\theta\left(\frac{f(\underline{\theta}^{q^{\underline{k}}})}{D_{\underline{k}}}\right)=\frac{q-3}{q-1}q^{k+1}+\frac{2q}{q-1}$$
	which does not tend to $-\infty$ as $k\rightarrow\infty$.
\end{proof}
The case $q=2$ is different. Indeed in this case $f=\tau\Big(\zeta_A\binom{\Sigma}{1}\Big)$ so that one sees easily that $f(\underline{\theta}^{q^{\underline{k}}})=0$ for all $\underline{k}\in\NN^\Sigma$ but finitely many elements.

\subsection{The maps $\mathcal{E}_\Sigma$}\label{defi-E}

Let $\Sigma$ be a finite subset of $\NN^*$. 
We have the map $$\EE_\Sigma\xrightarrow{\mathcal{E}_\Sigma}\CC_\infty[[\underline{X}_\Sigma]]^{\operatorname{lin}}$$
defined by 
$$\mathcal{E}_\Sigma(f)=\sum_{\underline{i}\in\NN^\Sigma}\frac{f(\underline{\theta}^{q^{\underline{i}}})}{D_{\underline{i}}}\prod_{j\in \Sigma}\lambda_A\binom{\{j\}}{1}^{q^{i_j}}.$$
\begin{Lemma}\label{Esigma-functional}
	$\mathcal{E}_\Sigma$ satisfies the functional equation:
	\begin{equation}\label{functional-equation}
	\mathcal{E}_\Sigma(\tau_\Sigma(f))=\mathcal{E}_\Sigma(f)^q,\quad \forall f\in\EE_\Sigma.
	\end{equation}
\end{Lemma}
\begin{proof}
	If $\underline{\lambda}=(\lambda_j:j\in\Sigma)=\left(\lambda_A\binom{\{j\}}{1}:j\in\Sigma\right)$ and if we set $\underline{\lambda}^{q^{\underline{i}}}:=\prod_j\lambda_j^{q^{i_j}}$
	with $\underline{i}=(i_j:j\in\Sigma)$, we
	have the following identities:
	\begin{eqnarray*}
		\mathcal{E}_\Sigma(\tau_\Sigma(f))&=&\sum_{\underline{i}\in\NN^\Sigma}\frac{[B_1(\Sigma)\tau(f)](\underline{\theta}^{q^{\underline{i}}})}{D_{\underline{i}}}\underline{\lambda}^{q^{\underline{i}}}\\
		&=&\sum_{\underline{i}\in(\NN^*)^\Sigma}\frac{B_1(\Sigma)(\underline{\theta}^{q^{\underline{i}}})f(\underline{\theta}^{q^{\underline{i}-1}})^q}{D_{\underline{i}}}\underline{\lambda}^{q^{\underline{i}}}\\
		&=&\sum_{\underline{i}\in(\NN^*)^\Sigma}\frac{f(\underline{\theta}^{q^{\underline{i}-1}})^q}{D_{\underline{i}-1}^q}\underline{\lambda}^{q^{\underline{i}}}\\
		&=&\mathcal{E}_\Sigma(f)^q.
	\end{eqnarray*}
	And the lemma follows.
\end{proof}
We note that the particular choice of $\underline{\lambda}$ above is superficial; we can well 
choose an arbitrary element 
$\underline{\lambda}\in\Big(\CC_\infty[[\underline{X}_\Sigma]]^{\operatorname{lin}}\Big)^\Sigma$ and get the same kind of identity.

\subsubsection{An example} In this subsection we revisit the classical identity 
(\ref{carlitz-identity}) below.
We return to the function $f=\zeta_A\binom{\sg}{1}$ with $\sg$ the singleton $\{j\}$. It is easy to see that $f\in\mathbb{E}_\sg^{[0]}$. In fact, we have $f(\theta)=1$
and $f(\theta^{q^k})=0$ for all $k>0$ as follows from simple properties of the values of the Goss zeta function at negative integers. We have seen in (\ref{image-basic-zeta}) that $\mathcal{F}_\sg(f)=\lambda_A\binom{\sg}{1}$ and the above trivial zeroes together with the value $f(\theta)=1$ also imply $\mathcal{E}_\sg(f)=\lambda_A\binom{\sg}{1}$.
Now, let us apply the functional equation (\ref{functional-equation}). By Part (1) of Proposition \ref{four-properties} we see that
$\tau_\sg^N(f)\in\EE_\sg^{[0]}$ for all $N\geq0$ and we can apply $\mathcal{E}_\sg$. We deduce:
$$\mathcal{E}_\sg(\tau^N_\sg(f))=\lambda_A\binom{\sg}{1}^{q^N},\quad \forall N\geq 0.$$
Now, we observe that
$$\sum_{N\geq 0}\mathcal{E}_\sg\Big(D_N^{-1}\tau^N_\sg(f)\Big)=\mathcal{E}_\sg\left(\sum_{N\geq 0}D_N^{-1}\tau^N_\sg(f)\right)=\exp_C\left(\lambda_A\binom{\sg}{1}\right).$$ From 
\cite{PEL0} we recall the functional identity
\begin{equation}\label{functional-identity}
\zeta_A\binom{\sg}{1}=\frac{\widetilde{\pi}}{(\theta-t)\omega(t)}
\end{equation}
with $t=t_j$ such that $\sg=\{j\}$ and $\omega$ the function of Anderson and Thakur defined in (\ref{omega-def}). Recall also the formula
$$\omega(t)=\sum_{N\geq 0}\frac{\widetilde{\pi}^{q^N}}{D_N(\theta^{q^N}-t)},$$
which is easily checked.
Noting that:
$$\tau_\sg^N(f)=\frac{\widetilde{\pi}^{q^N}}{(\theta^{q^N}-t)b_N(t)\omega(t)},\quad N\geq 0,$$
we see that, with $X=X_j$,
$$\exp_C\left(\lambda_A\binom{\sg}{1}\right)=\mathcal{E}_{\sg}\left(\sum_{N\geq 0}\frac{\widetilde{\pi}^{q^N}}{D_N(\theta^{q^N}-t)\omega(t)}\right)=\mathcal{E}_\sg(\omega\cdot\omega^{-1})=\mathcal{E}_\sg(1)=X,$$
and we deduce the (in fact well known) identity 
\begin{equation}\label{carlitz-identity}
\lambda_A\binom{\sg}{1}=\log_C(X).
\end{equation}
\begin{Corollary}\label{E-F-agree}
	The maps $\mathcal{E}_\Sigma$ and $\mathcal{F}_\Sigma$ agree as $\CC_\infty$-linear maps $\EE_\Sigma\rightarrow\CC_\infty[[\underline{X}_\Sigma]]^{\operatorname{lin}}$.
\end{Corollary}
\begin{proof}
	This is an immediate consequence of (\ref{carlitz-identity}).
\end{proof}

\begin{Corollary}\label{Theorem-E-F}
	If $\Sigma\in\mathfrak{S}$ is a finite subset such that $|\Sigma|<q$ the maps $\mathcal{F}_\Sigma$ and $\mathcal{E}_\Sigma$ agree on $\mathcal{Z}_{\Sigma}$.
\end{Corollary}
\subsubsection{Alternative proof of Corollary \ref{Theorem-E-F}}
Let $\Sigma$ be a subset of $\NN^*$ such that $|\Sigma|<q$. By the fourth part of Proposition \ref{four-properties}, $\mathcal{Z}_\Sigma\subset\EE_\Sigma^{[0]}$. Hence $\mathcal{E}_\Sigma$, well defined over $\EE_\Sigma^{[0]}$,
can be defined over $\mathcal{Z}_\Sigma$. Note that
if $f\in\CC_\infty[\underline{t}_\Sigma]$, then 
\begin{equation}\label{E:action}
\mathcal{E}_\Sigma(f)=f*\Big(\prod_{j\in\Sigma}X_j\Big)\end{equation}
in agreement with the $*$-action of \cite[\S 3.2]{ANG&PEL&TAV}. Since 
$\mathcal{E}_\Sigma(B_{\underline{i}}(\Sigma))=\underline{X}_\Sigma^{q^{\underline{i}}}$ for all
$\underline{i}$ by (\ref{functional-equation}), $\mathcal{E}_\Sigma=\mathcal{F}_\Sigma$.

\subsection{Trivial multiple zeta values}\label{Trivial-elements}

We set:
$$\mathcal{Z}_{n,\Sigma}^{\operatorname{triv}}:=\{f\in\mathcal{Z}_{n,\Sigma}:f(\underline{\theta}^{q^{\underline{i}}})=0,\text{ for all but finitely many }\underline{i}\in\NN^{\Sigma}\}.$$ 
This is an $\FF_p$-vector space. We also set $\mathcal{Z}_{\Sigma}^{\operatorname{triv}}=
\oplus_n\mathcal{Z}_{n,\Sigma}^{\operatorname{triv}}$, which is a $\mathcal{Z}_\emptyset$-submodule of 
$\mathcal{Z}_{\Sigma}$, and we define $\mathcal{Z}_{n,\Sigma}^{\operatorname{triv}}(K)$
and $\mathcal{Z}_{\Sigma}^{\operatorname{triv}}(K)$ in a similar way, so that the latter is a 
$\mathcal{Z}_\emptyset(K)$-module. 

The aim of this subsection is to prove the next result.

\begin{Theorem}\label{structure-of-trivial-MZV}
Assume that $\Sigma$ is a subset of $\NN^*$ such that $|\Sigma|<q$. Then $\mathcal{Z}_{\Sigma}^{\operatorname{triv}}(K)$, as a $\mathcal{Z}_\emptyset(K)$-module, is generated by the elements 
$\xi_k(t_j)$, $j\in\Sigma$, with 
$$\xi_k(t_j)=\zeta_A\begin{pmatrix}\{j\} & \emptyset & \cdots & \emptyset \\ 1 & q-1 & \cdots & (q-1)q^{k-1}\end{pmatrix},\quad k\geq 0,\quad j\in \Sigma.$$ 
In particular, $$\mathcal{Z}_{n,\Sigma}^{\operatorname{triv}}(K)=\bigoplus_{n\geq\sum_{i\in\Sigma}q^{k_i}}\mathcal{Z}_{n-\sum_{i\in\Sigma}q^{k_i},\emptyset}(K)\prod_{i\in\Sigma}\xi_{k_i}(t_i)=\bigotimes_{i\in\Sigma}\mathcal{Z}_{\{i\}}^{\operatorname{triv}}(K),$$ where the product $\otimes$ is of $\mathcal{Z}_\emptyset(K)$-modules.
Moreover, for all $n$, we have $\FF_q[\underline{t}_\Sigma]$-linear maps $$\tau_\Sigma:\mathcal{Z}_{n,\Sigma}^{\operatorname{triv}}(K)\rightarrow
\mathcal{Z}_{qn,\Sigma}^{\operatorname{triv}}(K).$$
\end{Theorem}

Before proving the theorem, we need two lemmas. The first one makes 
the statement of Lemma \ref{lemma-Hm} completely explicit, with $t=t_j$ and $\sg=\{j\}$.

\begin{Lemma}\label{Hdm-explicit} Let $\Sigma$ be a subset of $\NN^*$ with $|\Sigma|<q$.
The following formula holds:
$$D_k^{-1}\tau_\Sigma^k\left(S_{<d-k}\binom{\Sigma}{1}\right)=(-1)^kS_{<d}\begin{pmatrix}\Sigma & \emptyset & \cdots & \emptyset \\ 1 & q-1 & \cdots & (q-1)q^{k-1}\end{pmatrix},\quad d\geq k\geq 0.$$
In particular, for $\Sigma=\sg=\{j\}$ and $t=t_j$, 
$$H_m^{(d)}(t)=(-1)^mS_{<d}\begin{pmatrix}\sg & \emptyset & \cdots & \emptyset \\ 1 & q-1 & \cdots & (q-1)q^{m-1}\end{pmatrix},\quad d,m\geq 0.$$
\end{Lemma}

\begin{proof}
We recall that in \cite[Proposition 4.8]{PEL4} it is proved that, in the non-commutative polynomial ring $K[\tau]$ (with $\tau c=c^q\tau$, $c\in K$),
$$l_d\mathcal{E}_d=\left(1-\frac{\tau}{l_{d-1}^{q-1}}\right)\cdots\left(1-\frac{\tau}{l_1^{q-1}}\right)\left(1-\tau\right),$$ where $\mathcal{E}_d$ is the unique element of $K[\tau]$ such that
$E_d(z)=\mathcal{E}_d(z)$ (evaluation at $z$ of polynomial in $z$ and of operator in $K[\tau]$), where we recall that $E_d(z)$ is Carlitz's polynomial $D_d^{-1}\prod_a(z-a)$ (see ibid. Proposition 4.7). This implies a factorisation of the exponential function $\exp_A$ associated to the $A$-lattice $A\subset\CC_\infty$ in the non-commutative ring of formal series $K[[\tau]]$.
By (\ref{simpleformulas}),
the coefficient of $\tau^m$ in 
\begin{equation}\label{E:prof1}
\sum_{j=0}^{d-1}B_j(\Sigma)\mathcal{E}_j=\sum_{j=0}^{d-1}\frac{B_j(\Sigma)}{l_j}\left(1-\frac{\tau}{l_{j-1}^{q-1}}\right)\cdots\left(1-\frac{\tau}{l_1^{q-1}}\right)\left(1-\tau\right)
\end{equation}
is $(-1)^m\sum_{j=0}^{d-1}\frac{B_j(\Sigma)}{l_j}S_{<j}(q-1,\ldots,(q-1)q^{m-1})$ and the latter sum 
is equal to the sum $(-1)^m S_{<d}(\begin{smallmatrix}\Sigma & \emptyset & \cdots & \emptyset \\ 1 & q-1 & \cdots & (q-1)q^{m-1}\end{smallmatrix})$. Now, viewing (\ref{eq3}), the coefficient of $\tau^m$ in $\sum_{j=0}^{d-1}B_j(\Sigma)\mathcal{E}_j$ is $$D_k^{-1}\sum_{j=k}^{d-1}\frac{B_j(\Sigma)}{l_{j-k}^{q^k}}=\frac{B_k(\Sigma)}{D_k}\tau^{k}\Big(S_{<d-k}\binom{\Sigma}{1}\Big)$$ from which the first part of the lemma follows.
For the second part we note that if $\Sigma=\sg=\{j\}$ and $t=t_j$, then $H_m^{(d)}(t)$ is the coefficient of $\tau^m$ in $\sum_{j=0}^{d-1}b_j(t)\mathcal{E}_j$.
\end{proof}
We deduce:
\begin{Lemma}\label{identity-xi}
The following identities hold:
$$\xi_k(t)=\frac{(-1)^k}{D_k}\tau_\sg^k\left(\zeta_A\binom{\sg}{1}\right),\quad k\geq 0.$$
\end{Lemma}

\begin{proof}[Proof of Theorem \ref{structure-of-trivial-MZV}]
By Lemma \ref{identity-xi}, $g_k(t):=(-1)^k\xi_k(t)$ has the property that
$g_k(\theta^{q^i})=1$ if $i=k$ and  $g_k(\theta^{q^i})=0$ otherwise.
Let $f$ be an element of $\mathcal{Z}_\Sigma^{\operatorname{triv}}(K)$. We have 
that $f_{\underline{k}}:=f(\underline{\theta}^{q^{\underline{k}}})\in\mathcal{Z}_\emptyset(K)$ for all but finitely many $\underline{k}\in\NN^\Sigma$.
Hence, with $\underline{k}=(k_j:j\in\Sigma)$,
$$F:=f-\sum_{\underline{k}}(-1)^{|\underline{k}|}f_{\underline{k}}\prod_{j\in\Sigma}\xi_{k_j}(t_j)\in\mathcal{Z}_\Sigma(K)$$
has the property that $\mathcal{E}_\Sigma(F)=0$. Since $\mathcal{E}_\Sigma=\mathcal{F}_\Sigma$ is injective over $\EE_\Sigma^{[0]}$ (due to the fact that $|\Sigma|<q$), we have proved the first property of the theorem.
The second property now follows immediately as an application of Lemma \ref{identity-xi}.
\end{proof}
We deduce:
\begin{Corollary}\label{corollary-doesnt-exceed}
Let $f$ be an element of $\mathcal{Z}^{\operatorname{triv}}_{n,\Sigma}$ with $|\Sigma|<q$, let $\underline{i}$ be an element of $\NN^\Sigma$. Then
either $f(\underline{\theta}^{q^{\underline{i}}})=0$ or $\underline{i}=(i_j:j\in\Sigma)$, 
$n\geq \sum_jq^{i_j}$ and $f(\underline{\theta}^{q^{\underline{i}}})\in\mathcal{Z}_{n-\sum_jq^{i_j},\emptyset}$.
\end{Corollary}

\subsection{$K[\boldsymbol{\tau}]$-structures}\label{link-relations}

All along this subsection we choose $\Sigma\subset\NN^*$ finite with $|\Sigma|<q$.

We introduce an indeterminate $\boldsymbol{\tau}$ and we consider the skew polynomial ring $K[\boldsymbol{\tau}]$ with the commutation rule
$\boldsymbol{\tau} c=c^q\boldsymbol{\tau}$, $c\in K$. By Theorem \ref{structure-of-trivial-MZV}, the $K$-vector space $\mathcal{Z}^{\operatorname{triv}}_{\Sigma}(K)$
carries a structure of left $K[\boldsymbol{\tau}]$-module by $\boldsymbol{\tau} f:=\tau_\Sigma(f)$ for $f\in \mathcal{Z}^{\operatorname{triv}}_{\Sigma}(K)$. This structure is compatible with the grading structure by the monoid $\boldsymbol{S}$ because $
\tau_\Sigma:\mathcal{Z}_{n,\Sigma}^{\operatorname{triv}}(K)\rightarrow\mathcal{Z}_{qn,\Sigma}^{\operatorname{triv}}(K)$, injectively. Note that for this grading structure, multiplication by an element of $K$ leaves the degrees invariants, while the right multiplication by $\boldsymbol{\tau}$ multiplies the weights by $q$. Similarly, the $K$-vector space $\mathcal{Z}_\emptyset(K)=\oplus_n\mathcal{Z}_{n,\emptyset}(K)$ carries a structure of
left $K[\boldsymbol{\tau}]$-module by $\boldsymbol{\tau} f:=f^q$ again compatible with the grading, this time by $\NN$ in a way similar to that described above.

\begin{Lemma}
The composition of the map $\mathcal{F}_\Sigma:\mathcal{Z}_\Sigma(K)\rightarrow
\mathcal{L}_\Sigma(K)$ followed by the evaluation
at $X_i=1$ for all $i\in\Sigma$, induces a morphism $\mathcal{G}_\Sigma$ of $K[\boldsymbol{\tau}]$-modules
$$\mathcal{Z}_{\Sigma}^{\operatorname{triv}}(K)\xrightarrow{\mathcal{G}_\Sigma}\mathcal{Z}_{\emptyset}(K)$$
which is compatible with the grading by $\NN$ on the pre-image and on the target.
\end{Lemma}

\begin{proof} We have a well defined $K$-linear map $\mathcal{Z}_{n,\Sigma}(K)\xrightarrow{\mathcal{F}_\Sigma}\mathcal{L}_{n,\Sigma}(K)$ which is an isomorphism of $K$-vector spaces (Corollary \ref{E-F-agree}). The evaluation map which sends an element of $\CC_\infty[\underline{X}_\Sigma]$ to the value at $X_i=1$ for $i\in\Sigma$ induces a map
$\mathcal{L}_{n,\Sigma}(K)\rightarrow\mathcal{Z}_{n,\emptyset}(K)$ (in general not injective)
so we have a $K$-linear map $\mathcal{G}_\Sigma:\mathcal{Z}_{n,\Sigma}(K)\rightarrow
\mathcal{Z}_{n,\emptyset}(K)$. Now, $\mathcal{F}_\Sigma$ agrees with $\mathcal{E}_\Sigma$
and can be therefore extended to a $K_\infty[\boldsymbol{\tau}]$-linear map 
$\EE_\Sigma^{[0]}\rightarrow\TT_\Sigma$ by Lemma 
\ref{Esigma-functional}. The lemma follows easily from these properties and Theorem \ref{structure-of-trivial-MZV}.
 \end{proof}
We do not know if the $K[\boldsymbol{\tau}]$-linearity of $\mathcal{G}_\Sigma$ extends to $\mathcal{Z}_\Sigma(K)$.
On the other hand, it seems that the following hypothesis is plausible:

\begin{Conjecture}\label{injectivity}
The above map $\mathcal{G}_\Sigma$ is injective.
\end{Conjecture}

Let $\delta_{n,\Sigma}$ be the dimension over $K$ of $\mathcal{Z}_{n,\Sigma}^{\operatorname{triv}}(K)$. It is easy to see that $\delta_{n,\emptyset}=\dim_K(\mathcal{Z}_{n,\emptyset}(K))$ and that if $\Sigma'\subset\Sigma$ is such that $\Sigma\setminus\Sigma'=\{j\}$
(singleton), then 
$\delta_{n,\Sigma}=\sum_{q^k\leq n}\delta_{n-q^k,\Sigma'}$. We recall, from \cite{Todd}, the next
rational function field analogue of Zagier's dimension conjecture:

\begin{Conjecture}[Todd]\label{Todds-conjecture}
We have:
$$\delta_{n,\emptyset}=\left\{\begin{matrix}2^{n-1}\text{ if $1\leq n\leq q-1$},\\
2^{q-1}-1\text{ if $n=q$},\\
\sum_{i=1}^q\delta_{n-i,\emptyset}\text{ if $n>q$}.\end{matrix}\right.$$
\end{Conjecture} If we assume Conjecture \ref{Todds-conjecture}, we get conjectural formulas for the dimensions $\delta_{n,\Sigma}$ with $|\Sigma|<q$ and for all $n$ and $\Sigma'\subset\Sigma$ (with $|\Sigma|<q$), $\delta_{n,\Sigma'}\geq \delta_{n,\Sigma}$.

We consider the $K$-vector space
$\boldsymbol{C}_\Sigma(K)=\oplus_n\boldsymbol{C}_{n,\Sigma}(K)$ which is freely generated by 
the elements $[\mathcal{C}]$ with $\mathcal{C}$ admissible composition array of $\binom{\Sigma}{n}$, with $n\in\NN^*$. The properties of the harmonic product imply that $\boldsymbol{C}_\Sigma(K)$, together with the product $\odot_\zeta$, carries a structure of right $\boldsymbol{C}_\emptyset(K)$-module which is compatible with the grading structures. There is a $K$-linear map
$\boldsymbol{G}_\Sigma:\boldsymbol{C}_{n,\Sigma}(K)\rightarrow\boldsymbol{C}_{n,\emptyset}(K)$ such that $\mathcal{G}_\Sigma\circ\zeta_A=\zeta_A\circ\boldsymbol{G}_\Sigma$, determined by the proof of Theorem \ref{Z-L}.
We consider the subvector space of the $K$-vector space $\boldsymbol{C}_{n,\Sigma}(K)$:
$$\boldsymbol{C}_{n,\Sigma}^{\operatorname{triv}}(K):=\zeta_A^{-1}(\mathcal{Z}_{n,\Sigma}^{\operatorname{triv}}(K))$$ (\footnote{We could also consider an $\FF_p$-vector space $\boldsymbol{C}_{n,\Sigma}^{\operatorname{triv}}$ defined similarly, but using only elements of $\mathcal{Z}_{n,\Sigma}^{\operatorname{triv}}$, but we will be mainly concerned by the above space.}).
We also consider $$\boldsymbol{C}_{\Sigma}^{\operatorname{triv}}(K):=\bigoplus_n\boldsymbol{C}_{n,\Sigma}^{\operatorname{triv}}(K)$$ which also carries, by the harmonic multiplication $\odot_\zeta$, a structure of right $\NN$-graded $\boldsymbol{C}_{\emptyset}(K)$-module so that $\boldsymbol{G}_\Sigma$ becomes a morphism of graded $\boldsymbol{C}_{\emptyset}(K)$-modules. Observe that there are structures of $K[\boldsymbol{\tau}]$-modules over $\boldsymbol{C}_{\Sigma}^{\operatorname{triv}}(K)$ and $\boldsymbol{C}_{\emptyset}(K)$ which are compatible with the
gradings and also with $\boldsymbol{G}_\Sigma$ which is therefore equivariant. On $\boldsymbol{C}_{\emptyset}(K)$, we simply 
set
$$\boldsymbol{\tau}\left[\begin{matrix} n_1 & \cdots & n_r\end{matrix}\right]=\left[\begin{matrix}  n_1 & \cdots & n_r\end{matrix}\right]\underbrace{\odot_\zeta\cdots\odot_\zeta}_{q\text{ times}}\left[\begin{matrix}  n_1 & \cdots & n_r\end{matrix}\right]=\left[\begin{matrix} qn_1 & \cdots & qn_r\end{matrix}\right].$$
We have reached the following diagram of morphisms of $K[\boldsymbol{\tau}]$-modules compatible with the grading structures, where the lines are short exact sequences and all the squares commute (note that the center and the right vertical arrows are surjective by definition):

\begin{center}
\begin{tikzpicture}
  \matrix[matrix of math nodes,column sep={80pt,between origins},row
    sep={60pt,between origins},nodes={asymmetrical rectangle}] (s)
  {
     |[name=03]| 0 &|[name=A]| \operatorname{Ker}(\zeta_A)\cap\boldsymbol{C}_\Sigma^{\operatorname{triv}}(K) &|[name=B]| \boldsymbol{C}_\Sigma^{\operatorname{triv}}(K) &|[name=C]| \mathcal{Z}^{\operatorname{triv}}_\Sigma(K) &|[name=01]| 0 \\
    |[name=02]| 0 &|[name=A']| \operatorname{Ker}(\zeta_A)\cap\operatorname{Im}(\boldsymbol{G}_\Sigma) &|[name=B']| \operatorname{Im}(\boldsymbol{G}_\Sigma) &|[name=C']| \operatorname{Im}(\mathcal{G}_\Sigma) &|[name=04]| 0. \\
    };
  \draw[->] 
            (A) edge (B)
            (B) edge node[auto] {\(\zeta_A\)} (C)
            (C) edge (01)
            (C') edge (04)
            (03) edge (A)
            (A) edge node[auto] {\(\boldsymbol{G}_\Sigma\)} (A')
            (B) edge node[auto] {\(\boldsymbol{G}_\Sigma\)} (B')
            (C) edge node[auto] {\(\mathcal{G}_\Sigma\)} (C')
            (02) edge (A')
            (A') edge node[auto] {\(\)} (B')
            (B') edge node[auto] {\(\zeta_A\)} (C');
\end{tikzpicture} 
\end{center}
A simple application of the snake lemma implies:
\begin{Lemma}
We have that $\mathcal{G}_\Sigma$ is injective if and only if $\operatorname{Ker}(\boldsymbol{G}_\Sigma)\subset\operatorname{Ker}(\zeta_A)$. Moreover, if $\mathcal{G}_\Sigma$ is injective then the left vertical arrow is surjective.
\end{Lemma}
In particular, Conjecture \ref{injectivity} implies that, in the $K[\boldsymbol{\tau}]$-module $\mathcal{G}_\Sigma(\mathcal{Z}_{\Sigma}^{\operatorname{triv}}(K))$, all the $K$-linear relations, generated by homogeneous ones (for the grading by $\NN$), come from homogeneous $K[\boldsymbol{\tau}]$-linear relations in $\mathcal{Z}_{\Sigma}^{\operatorname{triv}}(K)$. Of course the limit of our approach is that the image $\mathcal{G}_\Sigma(\mathcal{Z}_{n,\Sigma}^{\operatorname{triv}}(K))$ is in general smaller than the target $\mathcal{Z}_{n,\emptyset}(K)$ (it also depends on the choice of $\Sigma$, but the larger $\Sigma$ is, the smaller the image is).
However, in all explicit cases we analysed experimentally, this construction is sufficient to detect all the $K$-linear dependence relations among Thakur's multiple zeta values for small weight. In \S \ref{Hadamard} we give some experimental evidences of these phenomena. 

\subsubsection{Remark} Another way to view the above constructions is to consider $M$ the 
$K[\boldsymbol{\tau}]$-module $\mathcal{Z}^{\text{triv}}_\Sigma(K)$ and $N$ the $K[\boldsymbol{\tau}]$-module
$\mathcal{G}_\Sigma(\mathcal{Z}^{\text{triv}}_\Sigma(K))\subset\mathcal{Z}_{\emptyset}(K)$.
If Conjecture \ref{injectivity} holds, then $M$ and $N$ are isomorphic, and there exists an isomorphism which is compatible with the graduations by $\NN$. Hence, if $F_\bullet^M$ and 
$F_\bullet^N$ are free resolutions of $M$ and $N$, they are isomorphic up to $K[\boldsymbol{\tau}]$-linear  homotopy. One can therefore hope to recover several $K$-linear relations between Thakur's multiple zeta values by considering $K[\boldsymbol{\tau}]$-linear relations between elements of $\mathcal{Z}^{\text{triv}}_\Sigma(K)$. This is what seems to happen in small weights and $|\Sigma|<q$; see \S \ref{Hadamard}. 

\subsubsection{Remark} Assume that $|\Sigma|<q$. There is a structure of $\mathcal{Z}_{\emptyset}(K)[\underline{t}_\Sigma,\tau]$-module over $\mathcal{Z}^{\operatorname{triv}}_{\Sigma}(K)$ compatible with the grading. If the multiplication of this module structure is denoted by $*$, for $i\in\Sigma$, we set, with $f\in \mathcal{Z}^{\operatorname{triv}}_{\Sigma}(K)$
$$t_i*f:=\mathcal{F}_\Sigma^{-1}(\mathcal{E}_\Sigma(t_if)).$$
This is well defined thanks to Corollary \ref{E-F-agree}, the identity $\log_C(C_\theta(X))=\theta\log_C(X)$, and the fact that 
$\mathcal{E}_\Sigma(\mathcal{Z}^{\operatorname{triv}}_{\Sigma}(K))$ is equal to $\mathcal{Z}_\emptyset(K)[\log_C(X_i):i\in\Sigma]^{\operatorname{lin}}$ (the map $\mathcal{F}_\Sigma$, as defined, cannot be evaluated at the entire function $t_if$). In particular, we have 
$\mathcal{E}_\Sigma(t_if)=\mathcal{E}_\Sigma(t_i*f)$ for all $f\in\mathcal{Z}^{\operatorname{triv}}_{\Sigma}(K)$ and we can give $\mathcal{Z}_\emptyset(K)[\log_C(X_i):i\in\Sigma]^{\operatorname{lin}}$ a structure of torsion $\mathcal{Z}_{\emptyset}(K)[\underline{t}_\Sigma,\tau]$-module which is isomorphic to the $\mathcal{Z}_{\emptyset}(K)[\underline{t}_\Sigma,\tau]$-module $\mathcal{Z}^{\operatorname{triv}}_{\Sigma}(K)$ via $\mathcal{E}_\Sigma$. Of course, from this structure we recover the $K[\boldsymbol{\tau}]$-module structures previously described.

\section{Applications to linear relations}\label{triviality-properties}

In this section we discuss some consequences of our results on the isomorphism 
$\mathcal{E}_\Sigma=\mathcal{F}_\Sigma$ over the $K[\boldsymbol{\tau}]$-modules $\mathcal{Z}_{\Sigma}(K)$, in the case $|\Sigma|<q$. The first result is a simple consequence of Theorem \ref{structure-of-trivial-MZV} and asserts that we can associate to trivial multiple zeta values certain relations between multiple polylogarithms. We suppose that $\Sigma$ is a subset of $\NN^*$ such that $|\Sigma|<q$.

Remember the $\lambda$-degrees from \S \ref{lambda-degrees}. We write $\boldsymbol{C}^\lambda_{n,\Sigma}$ for the free $\FF_p$-vector space generated by the composition arrays of 
$\lambda$-degree $\binom{\Sigma}{n}$ and $\boldsymbol{C}^\lambda_{n,\Sigma}(K)$ for the free $K$-vector space generated by the same composition arrays. We also recall that $\boldsymbol{C}^{\operatorname{triv}}_{n,\Sigma}=\zeta^{-1}_A(\mathcal{Z}^{\operatorname{triv}}_{n,\Sigma})$ and
we define $\boldsymbol{C}^{\operatorname{triv}}_{n,\Sigma}(K)$ similarly. The map $\mathcal{F}_\Sigma:\mathcal{Z}^{\operatorname{triv}}_{n,\Sigma}\rightarrow\mathcal{L}_{n,\Sigma}$
lifts to an injective $K$-linear map $F_\Sigma:\boldsymbol{C}^{\operatorname{triv}}_{n,\Sigma}\rightarrow \boldsymbol{C}_{n,\Sigma}(K)$ so that $\mathcal{F}_\Sigma\circ\zeta_A=\lambda_A\circ F_\Sigma$. Similarly, 
the map $\mathcal{E}_\Sigma:\mathcal{Z}^{\operatorname{triv}}_{n,\Sigma}\rightarrow\mathcal{L}_{\Sigma}(K)$ lifts to an injective $K$-linear map $E_\Sigma:\boldsymbol{C}^{\operatorname{triv}}_{n,\Sigma}\rightarrow \boldsymbol{C}_{\Sigma}(K)$ (in various ways) so that 
$\mathcal{E}_\Sigma\circ\zeta_A=\lambda_A\circ E_\Sigma$. Let us fix two such maps so that
we have a commutative diagram of $K$-vector spaces with injective horizontal maps and surjective vertical maps:
$$
\begin{tikzpicture}
         \matrix (m) [matrix of math nodes,row sep=3em,column sep=3em,minimum width=2em]{
              \boldsymbol{C}^{\operatorname{triv}}_{n,\Sigma}(K) &   & \boldsymbol{C}_{\Sigma}(K) \\
    \mathcal{Z}^{\operatorname{triv}}_{n,\Sigma}(K) & & \mathcal{L}_{\Sigma}(K).\\   };
         \path[-stealth]
            ([yshift=-3pt]m-1-1.east) edge node [above,yshift=1.0ex] {$E_\Sigma$} ([yshift=-3pt]m-1-3.west);
         \path[-stealth]
            ([yshift=3pt]m-1-1.east) edge node [below,,yshift=-1.0ex] {$F_\Sigma$} ([yshift=3pt]m-1-3.west);
            \path[-stealth]
            (m-1-1) edge node [left] {$\zeta_A$} (m-2-1);
            \path[-stealth]
            (m-1-3) edge node [right] {$\lambda_A$} (m-2-3);
            \path[-stealth]
            (m-2-1) edge node [below] {$\mathcal{E}_\Sigma=\mathcal{F}_\Sigma$} (m-2-3);
        \end{tikzpicture}
$$

\begin{Theorem}\label{Theorem-C}
The difference map $E_\Sigma-F_\Sigma$ defines a $K$-linear map
$$\boldsymbol{C}^{\operatorname{triv}}_{n,\Sigma}(K)\rightarrow\boldsymbol{C}^{\lambda}_{n,\Sigma}(K)$$ whose image is contained in the kernel of the map 
$\lambda_A:\boldsymbol{C}^{\lambda}_{n,\Sigma}(K)\rightarrow\mathcal{L}^{\lambda}_{n,\Sigma}(K)$. Moreover, $\operatorname{Ker}(E_\Sigma-F_\Sigma)\cap\boldsymbol{C}^{\operatorname{triv}}_{n,\Sigma}$ is equal to $\zeta_A^{-1}(\mathcal{Z}_{n-|\Sigma|}\prod_{j\in\Sigma}\zeta_A\binom{\{j\}}{1})$.
\end{Theorem}
In particular, if $f$ is an element of $\mathcal{Z}^{\operatorname{triv}}_{n,\Sigma}$
such that there exists $\underline{i}=(i_j:j\in\Sigma)\in\NN^\Sigma$ with the property that
$f(\underline{\theta}^{q^{\underline{i}}})\neq0$, then there exists a non-trivial $K$-linear dependence relation
in $\mathcal{L}^\lambda_{n,\Sigma}(K)$.
\begin{proof}[Proof of Theorem \ref{Theorem-C}] Let $f$ be an element of $\mathcal{Z}^{\operatorname{triv}}_{n,\Sigma}(K)$.
By Theorem \ref{Z-L}, we have that $\mathcal{F}_\Sigma(f)\in\mathcal{L}_{w,\Sigma}(K)=\mathcal{L}^\lambda_{w,\Sigma}(K)$. We now look at $$\mathcal{E}_\Sigma(f)=\sum_{\underline{i}}\frac{f(\underline{\theta}^{q^{\underline{i}}})}{D_{\underline{i}}}\lambda^{q^{\underline{i}}},$$
where $$\lambda^{q^{\underline{i}}}=\prod_{j\in\Sigma}\lambda_A\binom{\{j\}}{1}^{q^{i_j}},$$
which agrees with $\mathcal{F}_\Sigma(f)$ by Corollary \ref{Theorem-E-F}.
By Corollary \ref{corollary-doesnt-exceed}, the elements $f(\underline{\theta}^{q^{\underline{i}}})$
are either zero, or in $\mathcal{Z}_{n-\sum_jq^{i_j},\emptyset}(K)$ when $n\geq\sum_jq^{i_j}$. The properties of the harmonic product of power sums imply that for all $\underline{i}$ such that
$f(\underline{\theta}^{q^{\underline{i}}})\neq0$, 
$$f(\underline{\theta}^{q^{\underline{i}}})\lambda^{q^{\underline{i}}}\in\mathcal{L}_{n,\Sigma}^\lambda(K).$$ This suffices to show the first part of the Theorem, namely, that the image of
$E_\Sigma-F_\Sigma$ is in $\boldsymbol{C}^\lambda_{n,\Sigma}(K)$ and is contained in the kernel of $\lambda_A:\boldsymbol{C}^{\lambda}_{n,\Sigma}(K)\rightarrow\mathcal{L}^\lambda_{n,\Sigma}(K)$.
Now, let us suppose that
$f$ is an element of $\mathcal{Z}^{\operatorname{triv}}_{n,\Sigma}$. Then we have 
$\mathcal{F}_\Sigma(f)\in\mathcal{L}_{n,\Sigma}$ and for all $\underline{i}$, 
$f(\underline{\theta}^{q^{\underline{i}}})\lambda^{q^{\underline{i}}}\in\mathcal{L}_{n,\Sigma}^\lambda$ because by Corollary \ref{corollary-doesnt-exceed}, the elements $f(\underline{\theta}^{q^{\underline{i}}})$
are either zero, or in $\mathcal{Z}_{n-\sum_jq^{i_j},\emptyset}$ when $n\geq\sum_jq^{i_j}$. Additionally, since the monomials $\prod_{j}X_j^{q^{i_j}}$ are $K$-linearly independent, with $\underline{i}\in\NN^\Sigma$, so are the elements $f(\underline{\theta}^{q^{\underline{i}}})\lambda^{q^{\underline{i}}}$ with $f(\underline{\theta}^{q^{\underline{i}}})\neq0$. Hence, with $\underline{i}$ not equal to the zero vector and $f(\underline{\theta}^{q^{\underline{i}}})$, in the equality 
$\mathcal{F}_\Sigma(f)=\mathcal{E}_\Sigma(f)$, the right-hand side can be written as the image by $\lambda_A$ of a non-trivial 
$K$-linear combination in $\boldsymbol{C}_{n,\Sigma}^\lambda(K)$ with at least one coefficient which is not in $\FF_p$, because the polynomials $D_{\underline{i}}$ are all non-constant if $\underline{i}$ is not the zero vector. As for the left-hand side, this is the image by $\lambda_A$ of a non-zero $\FF_p$-linear combination in $\boldsymbol{C}_{n,\Sigma}^\lambda$ so that the difference between the corresponding two linear combinations in $\boldsymbol{C}_{n,\Sigma}^\lambda(K)$ projects via $\lambda_A$ to a non-trivial $K$-linear dependence relation in $\mathcal{L}_{n,\Sigma}^\lambda(K)$. This means that the map induced by $E_\Sigma-F_\Sigma$
over the quotient of $\boldsymbol{C}^{\operatorname{triv}}_{n,\Sigma}$ by 
$\zeta_A^{-1}(\mathcal{Z}_{n-|\Sigma|}\prod_{j\in\Sigma}\zeta_A\binom{\{j\}}{1})$ is injective.
\end{proof}
To give the simplest example of application of this result, we consider $f=\zeta_A(\begin{smallmatrix}\sg & \emptyset \\ 1 & q-1\end{smallmatrix})$. We know that $f\in\mathcal{Z}^{\operatorname{triv}}_{q,\sg}$. Note that, with $\sg=\{j\}$ and $t=t_j$, we have (with $f=f(t)$) $f(\theta^{q^k})=0$ for all $k\neq 1$ and $f(\theta^q)=-1$ (see the proof of Theorem \ref{structure-of-trivial-MZV}). Hence
$\mathcal{E}_\sg(f)=-D_1^{-1}\lambda_A\binom{\sg}{1}^q$. On another hand one easily sees that
$\mathcal{F}_\sg(f)=\lambda_A(\begin{smallmatrix}\sg & \emptyset \\ 1 & q-1\end{smallmatrix})$
so that the identity $\mathcal{F}_\sg(f)=\mathcal{E}_\sg(f)$ implies the identity
$$\lambda_A\binom{\sg}{1}^q=(\theta-\theta^q)\lambda_A\begin{pmatrix}\sg & \emptyset \\ 1 & q-1\end{pmatrix}$$ which is equivalent to (\ref{homogeneous-relations-sg}) for $k=1$ and 
which implies, by substitution $X=1$ (with $X=X_j$) $k=1$ a formula that was first discovered by Thakur in \cite[Theorem 5 (case $m=1$)]{THA1}:
$$\zeta_A(q)=(\theta-\theta^q)\zeta_A(1,q-1).$$ In particular, Thakur's identity lifts to an identity
in $\mathcal{L}^\lambda_{q,\sg}(K)$ (note that this is equivalent to the collection of 
{\em binary relations} $S_d(q)+D_1S_{d+1}(1,q-1)=0$, $d\geq0$, following Todd in \cite{Todd}. In the next subsection we are going to generalise this; see Theorem \ref{T:ThR} and (\ref{non-homogeneous-relations}).

\subsection{Some families of linear relations}
In this subsection we give examples of how our methods allow to construct infinite families of 
linear relations of multiple zeta values of Thakur. We illustrate two examples. In the first one we recover a result of Lara Rodr\'iguez and Thakur (Theorem \ref{T:ThR}). In the second one, we give a new family of relations (Theorem \ref{theorem-qk+1}).

\subsubsection{A family of relations of weight $q^k$ by Lara Rodr\'iguez and Thakur}
The next result has been originally obtained by Lara Rodr\'iguez and Thakur in \cite{LAR&THA}. 
The formula below in the case $k=2$ can be found in Todd's paper \cite[Theorem 4.7]{Todd}. 
\begin{Theorem}[Lara Rodr\'iguez-Thakur]\label{T:ThR}
We have, for all $k\geq 0$: $$\frac{\zeta_A(q^k)}{D_k}=(-1)^k\zeta_A(1,q-1,\ldots,q^{k-1}(q-1)).$$
\end{Theorem}

\begin{proof}
Using Lemma  \ref{Hdm-explicit} and summing over $d$ we reach the formula (for $|\Sigma|<q$):
\begin{equation}\label{homogeneous-relations-sg}
\frac{1}{D_k}\tau_\Sigma^k\left(\zeta_A\binom{\Sigma}{1}\right)=(-1)^k\zeta_A\begin{pmatrix}\Sigma & \emptyset & \cdots & \emptyset\\ 1 & q-1 & \cdots & q^{k-1}(q-1)\end{pmatrix},\quad k\geq 0.
\end{equation}
In the case $\Sigma=\sg$, applying $\mathcal{E}_\sg=\mathcal{F}_\sg$ we deduce the identity
\begin{equation}\label{non-homogeneous-relations}
\frac{1}{D_k}\lambda_A\binom{\sg}{1}^{q^k}=(-1)^k\lambda_A\begin{pmatrix}\sg & \emptyset & \cdots & \emptyset\\ 1 & q-1 & \cdots & q^{k-1}(q-1)\end{pmatrix}.
\end{equation}
Note that this is a linear dependence relation between elements of $\mathcal{L}$ that have 
different $\zeta$-types (but same $\lambda$-types). The left-hand side has $\zeta$-type $\sg^{q^k}=\sg\cdots\sg\in\mathfrak{S}$, while the right-hand side has
$\zeta$-type $\sg$. The result follows after evaluation $X=1$.
\end{proof}

\subsubsection{A family of relations of weight $q^{k+n}-1$ by Chen} Note that if we replace $t=\theta$ in (\ref{homogeneous-relations-sg}), case $\Sigma=\sg$, after having applied the operator $\tau^k$, then we recover Chen's \cite[Theorem 5.1]{CHE}. More explicitly, we note that, for $n,r\geq0$, after application of $\tau^n$ to both sides of the identity (\ref{homogeneous-relations-sg}) with $r=k$ and $\Sigma=\sg$:
$$(t-\theta^{q^n})\cdots(t-\theta^{q^{n+r-1}})\zeta_A\binom{\sg}{q^{n+r}}=(-1)^rD_r^{q^n}
\zeta_A\begin{pmatrix}\sg &\emptyset & \cdots & \emptyset \\ q^n & q^n(q-1)& \cdots & q^{n+r-1}(q-1)\end{pmatrix}.$$ Both sides of this identity are entire functions and we can evaluate at $t=\theta$. We get:
\begin{Theorem}[Chen]\label{th-chen}
For all $n,r\geq 0$, the following identity holds:
$$(\theta-\theta^{q^n})\cdots(\theta-\theta^{q^{n+r-1}})\zeta_A(q^{n+r}-1)=(-1)^rD_r^{q^n}\zeta_A(q^n-1,q^n(q-1),\ldots,q^{n+r-1}(q-1)).$$
\end{Theorem}
In other words, the identities (\ref{homogeneous-relations-sg}) unify the formulas of Lara Rodr\'iguez and Chen in the case $\Sigma=\sg$. For more general $\Sigma$ we deduce multiple parameter families of more general identities, we omit the details. An advantage of our method is that we need only very little computations to construct our linear relations.

\subsubsection{A family of relations of weight $q^k+1$}\label{family}
We can produce a variety of families of relations in the same vein as in Theorem \ref{T:ThR}. We 
illustrate another (slightly less explicit) example: a family of relations between Thakur's multiple zeta values each one of weight $q^k+1$ where $k$ varies in $\NN$. 

\begin{Theorem}\label{theorem-qk+1} For all $k>0$ there exists a non-trivial linear relation in $\mathcal{Z}_{q^k+1,\emptyset}(K)$ of the form:
\begin{equation}\label{E:eq5}
	\zeta_{A}(1,1,q-1,\dots,q^{k-1}(q-1))+\frac{(-1)^{k+1}}{D_k}\zeta_{A}(1,q^k)+\Xi_k=0
	\end{equation}
where $\Xi_k\in\mathcal{Z}_{q^k+1,\emptyset}(K)$ is given by the series
$$\Xi_k=\frac{1}{D_k}\sum_{j=1}^k(-1)^{k+j}D_{j}^{q^{k-j}}\sum_{d\geq 0}S_d(1)S_{d}(q^{k-j},q^{k-j}(q-1),\dots,q^{k-1}(q-1)).$$
\end{Theorem}
The fact that $\Xi_k\in\mathcal{Z}_{q^k+1,\emptyset}(K)$ as expressed in the series above follows from Lemma \ref{hadamard-product} and Proposition \ref{elementary-facts}. Indeed, we have that the $j$-th term in the sum over $j$ defining $\Xi_k$ is $K$-proportional to the evaluation at $X=1$ of the Hadamard product:
$$\lambda_A\binom{\sg}{1}\odot_\sg\lambda_A\begin{pmatrix} \sg & \emptyset & \cdots & \emptyset \\ q^{k-j} & q^{k-j}(q-1) & \cdots & q^{k-1}(q-1)\end{pmatrix}.$$

To prove Theorem \ref{theorem-qk+1} we need a preliminary lemma.

	\begin{Lemma}\label{L:L1} Assume that $|\Sigma'|<q$.
	For any $1\leq j \leq k$ and $d\geq j$, we have 
	\[
	\begin{split}
	&\tau_{\Sigma'}^k\left(S_{d-j}\begin{pmatrix}
	\Sigma'\\1
	\end{pmatrix}\right)=(-1)^jD_{j}^{q^{k-j}}B_{k-j}(\Sigma')S_{d}\begin{pmatrix}
	\Sigma'&\emptyset&\dots & \emptyset\\q^{k-j}&q^{k-j}(q-1)&\dots &q^{k-1}(q-1)
	\end{pmatrix}.
	\end{split}
	\]
	\end{Lemma}
\begin{proof} There are several ways to obtain these formulas. One of them uses the 
principle of the remark in \S \ref{formalism} if we assume that $|\Sigma'|<q-1$.
Let $h$ be an element of $\NN^*\setminus\Sigma'$ and set $\Sigma=\Sigma'\sqcup\{h\}$. Then,
from the computation of the image by $\mathcal{F}_{\{h\}}$ of both sides of (\ref{homogeneous-relations-sg}) we obtain that for integers $d\geq k\geq 0$:
	\[
	\tau_{\Sigma'}^k\left(S_{d-k}\binom{\Sigma'}{1}\right)=(-1)^kD_kS_{d}\begin{pmatrix}
	\Sigma'&\emptyset&\dots & \emptyset\\1&q-1&\dots &q^{k-1}(q-1)
	\end{pmatrix}.
	\]
This is the identity of the lemma in the case $j=k$. Another way to prove it, under the assumption that $|\Sigma'|<q$, is to apply Lemma \ref{Hdm-explicit} directly on two consecutive $d$'s.

The general case follows easily observing that 
$$S_{d-j}\begin{pmatrix}
	\Sigma'\\ q^k
	\end{pmatrix}=\tau^{k-j}\Big(S_{d-j}\begin{pmatrix}
	\Sigma'\\ q^j
	\end{pmatrix}\Big)$$
	and using the identity $b_k(t)=b_{k-j}(t)\tau^{k-j}(b_j(t))$.
\end{proof}

\begin{proof}[Proof of Theorem \ref{theorem-qk+1}]
From Lemma \ref{Hdm-explicit} we deduce, with $|\Sigma|<q$: 
\begin{multline}\label{E:eqPel0}
S_{<d}\begin{pmatrix}\Sigma&\emptyset&\dots &\emptyset \\n&q-1&\dots &q^{k-1}(q-1)
\end{pmatrix}=(-1)^k\frac{B_k(\Sigma)}{D_k}S_{<d-k}\begin{pmatrix}
\Sigma\\q^k
\end{pmatrix}=\\ =(-1)^k\frac{B_k(\Sigma)}{D_k}\left(S_{<d}\begin{pmatrix}
\Sigma\\q^k
\end{pmatrix}-\sum_{j=1}^kS_{d-j}\begin{pmatrix}
\Sigma\\q^k
\end{pmatrix}\right)
\end{multline}
for all $d\geq k$. Let $n$ be a positive integer. Multiplying both sides of \eqref{E:eqPel0} by $S_d(n)$ we get
	\begin{multline}\label{E:eqPel}
S_d\begin{pmatrix}\emptyset&\Sigma &\emptyset&\dots &\emptyset \\n&1&q-1&\dots &q^{k-1}(q-1)
\end{pmatrix}=(-1)^k\frac{B_k(\Sigma)}{D_k}S_{d}\begin{pmatrix}
\emptyset&\Sigma\\n&q^k
\end{pmatrix}-\\ -(-1)^k\frac{B_k(\Sigma)}{D_k}\sum_{j=1}^{k}S_d(n)S_{d-j}\binom{\Sigma}{q^k}.
\end{multline}

By using Lemma \ref{L:L1}, one can rewrite \eqref{E:eqPel} as 
\begin{equation}\label{E:eq4}
\begin{split}
&S_d\begin{pmatrix}\emptyset&\Sigma&\emptyset&\dots &\emptyset \\n&1&q-1&\dots &q^{k-1}(q-1)
\end{pmatrix}+(-1)^{k+1}\frac{B_k(\Sigma)}{D_k}S_d\begin{pmatrix}\emptyset&\Sigma \\n&q^k
\end{pmatrix}\\
&+\sum_{j=1}^k\frac{(-1)^{k+j}D_{j}^{q^{k-j}}B_{k-j}(\Sigma)}{D_k}S_d(n)S_{d}\begin{pmatrix}
\Sigma & \emptyset & \cdots & \emptyset \\ q^{k-j} & q^{k-j}(q-1) & \dots & q^{k-1}(q-1)
\end{pmatrix}=0.
\end{split}
\end{equation}
Summing over $d\geq 0$ yields the formula:
\begin{equation}\label{E:eq4bis}
\begin{split}
&\zeta_A\begin{pmatrix}\emptyset&\Sigma&\emptyset&\dots &\emptyset \\n&1&q-1&\dots &q^{k-1}(q-1)
\end{pmatrix}+(-1)^{k+1}\frac{B_k(\Sigma)}{D_k}\zeta_A\begin{pmatrix}\emptyset&\Sigma \\n&q^k
\end{pmatrix}\\
&+\sum_{j=1}^k\frac{(-1)^{k+j}D_{j}^{q^{k-j}}B_{k-j}(\Sigma)}{D_k}\sum_{d\geq 0}S_d(n)S_{d}\begin{pmatrix}
\Sigma & \emptyset & \cdots & \emptyset \\ q^{k-j} & q^{k-j}(q-1) & \dots & q^{k-1}(q-1)
\end{pmatrix}=0
\end{split}
\end{equation}
and we know that the series over $d\geq 0$ converges to an element of $\mathcal{Z}_{q^k+n,\Sigma}(K)$.
 We suppose that $\Sigma=\sg$ and we apply $\mathcal{F}_{\sg}$ to both sides of \eqref{E:eq4bis}. Evaluating at $X=1$ implies
 \begin{equation}\label{E:eq55}
 \begin{split}
 &\zeta_A(n,1,q-1,\dots,q^{k-1}(q-1))+\frac{(-1)^{k+1}}{D_k}\zeta_A(n,q^k)\\
 &+\sum_{j=1}^k\frac{(-1)^{k+j}D_{j}^{q^{k-j}}}{D_k}\sum_{d\geq0}S_d(n)S_{d}(q^{k-j},q^{k-j}(q-1),\dots,q^{k-1}(q-1))=0.
 \end{split}
 \end{equation}
 Now setting $n=1$ yields \eqref{E:eq5}. 
\end{proof}
For example, using \eqref{E:eq5} with $k=1$, we get
		\begin{multline*}
		\zeta_{A}(1,1,q-1)+\frac{1}{D_1}\zeta_{A}(1,q)+\sum_{d\geq 0}S_d(1)S_d(1)S_{<d}(q-1)=\\ =\zeta_{A}(1,1,q-1)+\frac{1}{D_1}\zeta_{A}(1,q)+\zeta_{A}(2,q-1)=0.
		\end{multline*}
				This is \cite[(5.2)]{Todd}. If we choose $k=2$, some explicit calculations yield the identity 
				\begin{multline*}\zeta_{A}(1,1,q-1,q^2-q)-\frac{\zeta_{A}(1,q^2)}{D_2}-\frac{1}{[2]}\zeta_{A}(q+1,q^2-q)-\Big(\frac{1}{[2]}-1\Big)\zeta_{A}(2,q-1,q^2-q)-\\ -\frac{1}{[2]}\zeta_{A}(2,q^2-1)=0,\end{multline*}
where $[k]=\theta^{q^k}-\theta$ for $k\in\NN$.

\subsection{Relations in small weights}\label{Hadamard} Using the methods of \S \ref{link-relations} we can recover all the relations in weight $q,q+1$ and $q+2$ described by Todd in \cite[Sec. 5.2-5.3]{Todd}. One of the patterns that seems to emerge is that the interplay between 
the harmonic product of the algebra $\mathcal{Z}$ and the Hadamard product on $\mathcal{L}$
via the maps $\mathcal{E}_\Sigma=\mathcal{F}_\Sigma$ works as a substitute of the double shuffle relations for Euler-Zagier multiple zeta values. In the following, we are going to examine the case $\Sigma=\sg=\{j\}\subset\NN^*$ and $n=q,q+1,q+2$. We denote by $t$ the variable $t_j$ and by $X$ the variable $X_j$, for simplicity.

\subsubsection{Relations in weight $q$}\label{relq}

By Theorem \ref{structure-of-trivial-MZV} a basis of 
$\mathcal{Z}_{q,\sg}^{\operatorname{triv}}(K)$ (the dimension is denoted by $\delta_{q,\sg}$) is given by the image through $\zeta_A$ of a family
$$\left(\left[\begin{matrix}\sg \\ 1\end{matrix}\right]\odot_{\zeta}F:(F) \text{ a $K$-basis of $\mathcal{Z}_{q-1,\emptyset}(K)$}\right)\bigsqcup\left(\left[\begin{matrix}\sg & \emptyset\\ 1 & q-1\end{matrix}\right]\right)\subset\boldsymbol{C}_{q,\sg}^{\operatorname{triv}}(K).$$ Note that the element 
$[\begin{smallmatrix}\sg\\ q\end{smallmatrix}]$ also belongs to $\boldsymbol{C}_{q,\sg}^{\operatorname{triv}}(K)$ but it is submitted to the $\FF_p$-linear relation
$$\left[\begin{matrix}\sg \\ q\end{matrix}\right]+\left[\begin{matrix}\sg & \emptyset\\ 1 & q-1\end{matrix}\right]-\left[\begin{matrix}\sg \\ 1\end{matrix}\right]\odot_\zeta\left[\begin{matrix}\emptyset \\ q-1\end{matrix}\right]=0$$ which is trivial.
By Theorem \ref{structure-of-trivial-MZV} we have the non-trivial element \[
\boldsymbol{\tau}\left[\begin{matrix}\sg \\ 1\end{matrix}\right]+D_1
\left[\begin{matrix}\sg & \emptyset \\ 1 & q-1\end{matrix}\right]\in \operatorname{Ker}(\zeta_A)\cap\boldsymbol{C}_{q,\sg}^{\operatorname{triv}}(K)
\]
 which can be identified with an element of $K[\boldsymbol{\tau}]^{\oplus \delta_{q,\sg}}$ (first syzygy $K[\boldsymbol{\tau}]$-module). Applying $\boldsymbol{G}_\Sigma$, we get the element (note that we omit the first line of composition arrays of trivial type)
$$\boldsymbol{\tau}[1]+D_1[1,q-1]\in\operatorname{Ker}(\zeta_A)\cap\boldsymbol{C}_{q,\emptyset}(K).$$ Since $\boldsymbol{\tau} [1]=[q]$, this yields Thakur's relation
\begin{equation}\label{q1}
\zeta_A(q)+D_1\zeta_A(1,q-1)=0.
\end{equation}
It seems difficult to find non-trivial elements of $K^{\oplus\delta_{q,\sg}}$ determining non-trivial elements in the above kernel. This goes in the same direction of Todd's analogue of Zagier's dimension conjecture, which predicts that $\dim_K(\mathcal{Z}_{q,\emptyset}(K))=2^{q-1}-1$.

\subsubsection{Relations in weight $q+1$}\label{relq+1}
The $K$-vector space $\boldsymbol{C}_{q+1,\sg}^{\operatorname{triv}}(K)$ contains the elements $\left[\begin{matrix}\sg\\1\end{matrix}\right]\odot_{\zeta}F$ where $F\in \boldsymbol{C}_{q,\emptyset}(K)$ and $\left[\begin{matrix}
	\sg&1\\1&q-1
	\end{matrix}\right]\odot_{\zeta}\left[\begin{matrix}
	\emptyset\\1
	\end{matrix}\right]
	$.
The next lemma describes two non-trivial elements in $\operatorname{Ker}(\zeta_A)\cap\boldsymbol{C}_{q+1,\sg}^{\operatorname{triv}}(K)$ which are sent, via $\boldsymbol{G}_\sg$, to elements which generate the $K$-linear relations among Thakur's multiple zeta values of weight $q+1$. 

\begin{Lemma}\label{lemmaq+1A}
We have the following two non-zero elements of the $K[\boldsymbol{\tau}]$-module
$\operatorname{Ker}(\zeta_A)\cap\boldsymbol{C}_{q+1,\sg}^{\operatorname{triv}}(K)$:
\begin{eqnarray}
\left[\begin{matrix}
\sg&\emptyset\\1&q
\end{matrix}\right]-D_1\left[\begin{matrix}
\emptyset&\sg&\emptyset\\1&1&q-1
\end{matrix}\right]-D_1\left[\begin{matrix}
\emptyset&\sg\\1&q
\end{matrix}\right]+\left[\begin{matrix}
\sg\\1
\end{matrix}\right]\odot_{\zeta}D_1\left[\begin{matrix}
\emptyset&\emptyset\\1&q-1
\end{matrix}\right]\label{1q+1},\\
\left[\begin{matrix}\sg\\q+1\end{matrix}\right]+\left[\begin{matrix}\sg&\emptyset\\1&q\end{matrix}\right]+(1+D_1)\left[\begin{matrix}\sg&\emptyset\\2&q-1\end{matrix}\right]-\left[\begin{matrix}\emptyset&\sg\\1&q\end{matrix}\right] +D_1\left[\begin{matrix}\sg&\emptyset&\emptyset\\1&1&q-1\end{matrix}\right]
\label{E:todd003}.
\end{eqnarray}
\end{Lemma}
\begin{proof} 
	The element \eqref{E:todd003} is determined by the product of $\zeta_{A}\binom{\sg}{1}$ and the two sides of \eqref{q1}. We now prove that the element (\ref{1q+1}) belongs to $\operatorname{Ker}(\zeta_A)\cap\boldsymbol{C}_{q+1,\sg}^{\operatorname{triv}}(K)$; we will use the Hadamard product of Lemma \ref{hadamard-product}. Taking $k=1$ in Theorem \ref{T:ThR} we obtain:
	$$\lambda_A\binom{\sg^q}{q}+D_1\lambda_A\begin{pmatrix} \sg & \emptyset \\ 1 & q-1\end{pmatrix}=0.$$ Now we compute the Hadamard product $\lambda_A\begin{pmatrix}
	\sg^q\\1
	\end{pmatrix}\odot_\sg(\cdots)$ of both left- and right-hand sides of the above identity:
\begin{equation}\label{lambdaform2}
	\lambda_A\begin{pmatrix}
	\sg^q\\1
	\end{pmatrix} \odot_\sg \left(\lambda_A\begin{pmatrix}
	\sg^q\\q
	\end{pmatrix}+D_1\lambda_A\begin{pmatrix}
	\sg&\emptyset\\1&q-1
	\end{pmatrix}\right)=0.
\end{equation}
	Equivalently, we have:
	$
	\sum_{d\geq 0}S_d(1)\Big(S_d(q)+D_1S_{d+1}(1,q-1)\Big)X^{q^{d+1}}=0.
	$
	Observe on another hand that for any $d\geq0$, we have
	\[
	S_d(1)b_{d+1}(t)=\frac{1}{l_d}(t-\theta)\tau(b_d(t))=(t-\theta)S_{<d+1}\begin{pmatrix}
	\sg\\1
	\end{pmatrix},
	\]
	where the last equality follows from the well-known identity (see e.g. \cite[Lem. 8]{PEL}):
	\begin{equation}\label{lambdaform3}
	\tau(b_d(t))=l_d\sum_{i=0}^d\frac{b_i(t)}{l_i}.
	\end{equation}
	In particular, if we set
	$$Z:=\sum_{d\geq 0}S_{d+1}(1,q-1)S_{<d+1}\binom{\sg}{1}\in\mathcal{Z}_{q+1,\sg},$$
	we have $$\mathcal{F}_\sg\Big((t-\theta)Z\Big)=\lambda_A\binom{\sg^q}{1}\odot_\sg\lambda_A\begin{pmatrix}
	\sg & \emptyset \\ 1 & q-1\end{pmatrix}.$$ Similarly, there exists a unique element $W\in \mathcal{Z}_{q+1,\sg}$ such that $$\mathcal{F}_\sg\Big((t-\theta)W\Big)=\lambda_A\binom{\sg^q}{1}\odot_\sg\lambda_A\binom{\sg^q}{q}$$ and this element can be easily computed. 
	Thus applying the map $\mathcal{F}_{\sg}^{-1}$ (an isomorphism) to \eqref{lambdaform2} and dividing by 
	$t-\theta$ implies
	\begin{equation}\label{E:todd002}
	\zeta_{A}\binom{\sg}{q+1}+\zeta_{A}\begin{pmatrix}
	\sg&\emptyset\\2&q-1
	\end{pmatrix}-\zeta_{A}\begin{pmatrix}
	\emptyset & \sg\\1&q
	\end{pmatrix}+D_1\zeta_{A}\begin{pmatrix}
	\emptyset & \sg&\emptyset\\1&1&q-1
	\end{pmatrix} +D_1\zeta_{A}\begin{pmatrix}
	\emptyset & \sg\\1&q
	\end{pmatrix}=0.
	\end{equation}
	Taking the term-wise difference between \eqref{E:todd002} and \eqref{E:todd003} we find 
	\begin{multline}\label{E:rel77}
	\zeta_{A}\begin{pmatrix}
	\sg&\emptyset\\1&q
	\end{pmatrix}+D_1\zeta_{A}\begin{pmatrix}
	\sg&\emptyset\\2&q-1
	\end{pmatrix}+D_1\zeta_{A}\begin{pmatrix}
	\sg&\emptyset&\emptyset\\1&1&q-1
	\end{pmatrix}-\\ -D_1\zeta_{A}\begin{pmatrix}
	\emptyset&\sg&\emptyset\\1&1&q-1
	\end{pmatrix}-D_1\zeta_{A}\begin{pmatrix}
	\emptyset&\sg\\1&q
	\end{pmatrix}=0.
	\end{multline}
	By explicit versions of Proposition \ref{elementary-facts} (see also \cite[Rem. 3.2]{Chen}  and \cite[Thm. 3.1]{PEL1}) we get
	\[
	\zeta_{A}\binom{\sg}{1}\zeta_{A}(1,q-1)=\zeta_{A}\begin{pmatrix}
	\sg & \emptyset & \emptyset\\ 1 & 1 & q-1\end{pmatrix}
	+\zeta_{A}\begin{pmatrix}
	\sg  &\emptyset\\ 2 & q-1\end{pmatrix}.
	\]
Hence,	\begin{equation}\label{M1}
	\zeta_{A}\begin{pmatrix}
	\sg&\emptyset\\1&q
	\end{pmatrix}-D_1\zeta_{A}\begin{pmatrix}
	\emptyset&\sg&\emptyset\\1&1&q-1
	\end{pmatrix}-D_1\zeta_{A}\begin{pmatrix}
	\emptyset&\sg\\1&q
	\end{pmatrix}=-D_1\zeta_{A}(1,q-1)\zeta_{A}\binom{\sg}{1}\in\mathcal{Z}^{\operatorname{triv}}_{q+1,\sg}(K)
	\end{equation}
	which determines the element \eqref{1q+1}.
\end{proof}

\begin{Lemma}
Assuming Conjecture \ref{Todds-conjecture} for the weight $q+1$, the map $\boldsymbol{G}_{\sg}$ sends the elements of Lemma $\ref{lemmaq+1A}$ to a basis of  the $K$-subvector space of $\operatorname{Ker}(\zeta_A)\cap\boldsymbol{C}_{q+1,\emptyset}^{\operatorname{triv}}(K)$.
\end{Lemma}
\begin{proof}
In fact, we recover exactly Todd's relations in \cite[Sec. 5.2]{Todd}.
	  By an explicit version of \eqref{results} and applying the map $\boldsymbol{G}_{\sg}$ to  \eqref{1q+1}, we get
	\begin{equation}\label{E:todd111}
	D_1\zeta_{A}(2,q-1)+D_1\zeta_{A}(1,1,q-1)+\zeta_{A}(1,q)=0.
	\end{equation}
	In other words, we have reached \cite[(5.2)]{Todd}.  On the other hand, applying $\boldsymbol{G}_{\sg}$ to \eqref{E:todd003} and subtracting \eqref{E:todd111} from the result, we obtain
	\[
	\zeta_{A}(q+1)+D_1\zeta_A(1,q)+D_1\zeta_{A}(1,1,q-1)+D_1\zeta(1,q-1,1) + \zeta_{A}(q,1) + \zeta_{A}(2,q-1) = 0
	\]
	which verifies the relation \cite[(5.3)]{Todd}.
\end{proof}

\subsubsection{Relations in weight $q+2$} 
We determine five $K[\boldsymbol{\tau}]$-linearly independent elements in $\operatorname{Ker}(\zeta_A)\cap\boldsymbol{C}_{q+2,\sg}^{\operatorname{triv}}(K)$ by using the methods of \S \ref{relq} and \ref{relq+1}. We do not give full details of the computations.

\begin{Lemma}\label{lemmaq+2A}
	We have the following five non-zero elements in the $K[\boldsymbol{\tau}]$-module
	$\operatorname{Ker}(\zeta_A)\cap\boldsymbol{C}_{q+2,\sg}^{\operatorname{triv}}(K)$:
	\begin{eqnarray}
	\left[\begin{matrix}
	\sg&\emptyset&\emptyset\\1&1&q
	\end{matrix} \right]+D_1\left[\begin{matrix}
	\sg&\emptyset&\emptyset\\1&2&q-1
	\end{matrix} \right]-D_1\left[\begin{matrix}
	\sg&\emptyset&\emptyset\\2&1&q-1
	\end{matrix} \right]+\left[\begin{matrix}
	\sg\\1
	\end{matrix}\right]\odot_{\zeta}D_1\left[\begin{matrix}
	\emptyset&\emptyset&\emptyset\\1&1&q-1
	\end{matrix} \right] \label{5q+2},\\
-\left[\begin{matrix}\sg&\emptyset&\emptyset\\1&q&1
\end{matrix}\right]-\left[\begin{matrix}\emptyset&\sg\\q&2
\end{matrix}\right]+D_1\left[\begin{matrix}\emptyset&\sg&\emptyset&\emptyset\\1&1&1&q-1
\end{matrix}\right]+D_1\left[\begin{matrix}\emptyset&\sg&\emptyset&\emptyset\\1&1&q-1&1
\end{matrix}\right]\nonumber\\ +D_1\left[\begin{matrix}\emptyset&\sg&\emptyset\\1&q&1
\end{matrix}\right]-D_1\left[\begin{matrix}\emptyset&\emptyset&\sg\\1&q-1&2
\end{matrix}\right]+D_1\left[\begin{matrix}\emptyset&\sg&\emptyset\\1&1&q
\end{matrix}\right]+\left[\begin{matrix}
\sg\\1
\end{matrix}\right]\odot_{\zeta}\left[\begin{matrix}
\emptyset&\emptyset\\q&1
\end{matrix}\right],\label{new1}\\
2D_1\left[\begin{matrix}
\emptyset&\sg&\emptyset\\1&2&q-1
\end{matrix}\right]-\left[\begin{matrix}
\emptyset&\sg\\2&q
\end{matrix}\right]+D_1\left[\begin{matrix}
\emptyset&\emptyset&\sg\\1&q-1&2
\end{matrix}\right]+\left[\begin{matrix}
\emptyset&\sg\\q&2
\end{matrix}\right]+D_1\left[\begin{matrix}
\emptyset&\sg\\1&q+1
\end{matrix}\right]\nonumber\\
-D_1\left[\begin{matrix}
\emptyset&\emptyset&\sg\\1&1&q
\end{matrix}\right]-\left[\begin{matrix}
\sg&\emptyset\\q&2
\end{matrix}\right]+\left[\begin{matrix}
\sg\\q
\end{matrix}\right]\odot_{\zeta}\left[\begin{matrix}
\emptyset\\2
\end{matrix}\right] \label{new3},\\
\bigg( \boldsymbol{\tau}\left[\begin{matrix}\sg \\ 1\end{matrix}\right]+D_1
\left[\begin{matrix}\sg & \emptyset \\ 1 & q-1\end{matrix}\right]	\bigg)\odot_{\zeta}\left[\begin{matrix}
\emptyset\\2
\end{matrix}\right] \label{7q+2},\\
\bigg( \boldsymbol{\tau}\left[\begin{matrix}\sg \\ 1\end{matrix}\right]+D_1
\left[\begin{matrix}\sg & \emptyset \\ 1 & q-1\end{matrix}\right]	\bigg)\odot_{\zeta}\left[\begin{matrix}
\emptyset&\emptyset\\1&1
\end{matrix}\right]	\label{new2}.
	\end{eqnarray}
\end{Lemma}
An immediate consequence of Lemma \ref{lemmaq+2A} and some explicit computations is the following:
\begin{Lemma}
Assuming Conjecture \ref{Todds-conjecture} for the weight $q+2$ the map $\boldsymbol{G}_{\sg}$ sends the five elements of Lemma $\ref{lemmaq+2A}$ to a basis of  the $K$-subvector space of $\operatorname{Ker}(\zeta_A)\cap\boldsymbol{C}_{q+2,\emptyset}^{\operatorname{triv}}(K)$.
\end{Lemma}
In fact, the above five elements are precisely sent to Todd's relations in \cite[Sec. 5.3]{Todd} and are therefore linearly independent.

\end{document}